\documentclass[a4paper,10pt]{article}
\usepackage[english,frenchb]{babel}
\usepackage[T1]{fontenc}
\usepackage[utf8]{inputenc}
\usepackage{amsmath}
\usepackage{array}
\usepackage{amsthm}
\usepackage{amssymb}
\input xy
\xyoption{all}

\newcommand{\A}{{\mathcal{A}}}
\newcommand{\B}{{\mathcal{B}}}
\newcommand{\bd}{\bar{{\rm d}}_}
\newcommand{\C}{{\mathcal{C}}}
\newcommand{\D}{{\mathcal{D}}}
\newcommand{\dst}{{\mathbf{D}_{st}}}
\newcommand{\dsto}{{\mathbf{D}_{st}^\omega}}
\newcommand{\F}{{\mathcal{F}}}
\newcommand{\Hh}{{\mathcal{H}}}
\newcommand{\Pp}{{\mathcal{P}}}
\newcommand{\pol}{{\mathcal{P}ol}}

\newcommand{\R}{{\mathcal{R}}}
\newcommand{\s}{{\mathcal{S}}}

\newcommand{\St}{\mathbf{St}}

\newcommand{\M}{{\mathcal{M}}}

\newcommand{\G}{{\mathcal{G}}}
\newcommand{\gr}{{\mathbf{gr}}}

\newcommand{\T}{{\mathcal{T}}}

\newcommand{\kk}{{\Bbbk}}

\newcommand{\md}{\text{-}\mathbf{Mod}}
\newcommand{\mdd}{\mathbf{Mod}\text{-}}

\title{De l'homologie stable des groupes de congruence}
\author{Aur\'elien DJAMENT\thanks{CNRS, Laboratoire de mathématiques Jean Leray, 2 rue de la Houssinière, BP 92208, 44322 NANTES CEDEX 3, FRANCE, aurelien.djament@univ-nantes.fr, http://www.math.sciences.univ-nantes.fr/\textasciitilde djament/.}}

\newtheorem{thi}{Th\'eor\`eme}
\newtheorem{pri}[thi]{Proposition}

\newtheorem{thm}{Th\'eor\`eme}[section]
\newtheorem{pr}[thm]{Proposition}
\newtheorem{cor}[thm]{Corollaire}
\newtheorem{lm}[thm]{Lemme}

\newtheorem{prdef}[thm]{Proposition et d\'efinition}
\newtheorem{conj}[thm]{Conjecture}

\theoremstyle{definition}
\newtheorem{defi}[thm]{D\'efinition}
\newtheorem{nota}[thm]{Notation}
\newtheorem{hyp}[thm]{Hypoth\`ese}

\theoremstyle{remark}
\newtheorem{remi}[thi]{Remarque}
\newtheorem{rem}[thm]{Remarque}
\newtheorem{ex}[thm]{Exemple}

\begin{document}

\maketitle

\begin{abstract}

On montre dans ce travail que l'homologie en degré $d$ d'un groupe de congruence, dans un contexte très général, définit un foncteur faiblement polynomial de degré au plus $2d$ et l'on décrit ce foncteur modulo les foncteurs polynomiaux de degré strictement inférieur. Notre outil principal est une suite spectrale reliant homologie de groupes de type congruence (dans un cadre formel voisin de celui développé avec Vespa en 2010 pour les groupes orthogonaux) et homologie des foncteurs. Nous montrons et utilisons de façon cruciale des propriétés de certaines structures tensorielles et de certaines extensions de Kan dérivées sur des foncteurs polynomiaux.

Nos résultats généralisent notamment, par des méthodes différentes, les travaux de Suslin sur l'excision en $K$-théorie algébrique entière et une prépublication récente de Church-Miller-Nagpal-Reinhold.

\begin{center}
 \textbf{Abstract}
\end{center}

We show in this work that homology in degree $d$ of a congruence group, in a very general framework, defines a weakly polynomial functor of degree at most $2d$ and we describe this functor modulo polynomial functors of smaller degree. Our main tool is a spectral sequence connecting homology of congruence-like groups (in a formal setting close to the one introduced with Vespa in 2010 for orthogonal groups) and functor homology. We prove and use in a crucial way properties of some tensor structures and derived Kan extensions on polynomial functors.

Our results extend especially, with different methods, the work by Suslin on excision in integer algebraic $K$-theory and a recent preprint by Church-Miller-Nagpal-Reinhold.

\end{abstract}

\bigskip

\noindent
{\em Mots-clefs : } groupes de congruence, homologie des groupes, homologie des foncteurs, foncteurs polynomiaux, extensions de Kan, foncteurs dérivés, suites spectrales, excision en $K$-théorie algébrique.

\medskip

\noindent
{\em Classification MSC 2010 : } 18A25, 18A40, 18E15, 18E35, 18G40, 20H05 ; 18D10, 18E05, 18E30, 19D99.

\newpage

\tableofcontents

\section*{Introduction}

Dans son article \cite{S-exc}, Suslin, simplifiant et généralisant des résultats obtenus, dans le cas rationalisé, avec Wozicki dans \cite{SW}, caractérise l'excisivité en $K$-théorie algébrique par un critère homologique simple. Si $I$ est un anneau {\em sans unité} (ce qui est la même chose qu'un idéal bilatère d'un anneau unitaire), notons $\mathbb{Z}\ltimes I$ l'anneau obtenu en ajoutant formellement une unité à $I$ (le groupe abélien sous-jacent est $\mathbb{Z}\oplus I$, $I$ en est un idéal bilatère et le passage au quotient par $I$ fournit une rétraction au morphisme d'anneaux unitaires $\mathbb{Z}\to\mathbb{Z}\ltimes I$). Les groupes linéaires sur $I$ sont définis par $GL_n(I):={\rm Ker}\,\big(GL_n(\mathbb{Z}\ltimes I)\to GL_n(\mathbb{Z})\big)$ : ce sont des groupes de congruence ; ils sont munis d'une action naturelle de $GL_n(\mathbb{Z})$. Il n'est pas très difficile de voir que le caractère excisif de $I$ pour la $K$-théorie algébrique jusqu'en degré $d$ équivaut au fait que l'action de $GL_n(\mathbb{Z})$ sur $H_i(GL_n(I);\mathbb{Z})$ est triviale, {\em stablement en $n$} (c'est-à-dire lorsqu'on passe à la colimite sur $n\in\mathbb{N}$), pour tout $i\leq d$ (on renvoie le lecteur au début de \cite{SW} pour plus de précisions sur ces questions). Le résultat principal de \cite{S-exc} (voir aussi le corollaire~\ref{csus} ci-après pour une forme plus générale du résultat de Suslin) peut s'exprimer comme suit :
\begin{thi}[Suslin]\label{this} Soit $I$ un anneau sans unité.
\begin{enumerate}
 \item Cet anneau est excisif pour la $K$-théorie algébrique en degrés $<e$ si et seulement si ${\rm Tor}^{\mathbb{Z}\ltimes I}_i(\mathbb{Z},\mathbb{Z})=0$ pour $0<i<e$. 
 \item On dispose d'un morphisme
 $$H_e(GL_n(I);\mathbb{Z})\to\mathfrak{gl}_n\big({\rm Tor}^{\mathbb{Z}\ltimes I}_e(\mathbb{Z},\mathbb{Z})\big)$$
 compatible à la stabilisation en $n$ et aux actions naturelles de $GL_n(\mathbb{Z})$, où $\mathfrak{gl}_n(V)$ désigne le groupe abélien des matrices $n\times n$ à coefficients dans $V$. Si la condition d'excisivité précédente est vérifiée, l'action de $GL_n(\mathbb{Z})$ sur le noyau et le conoyau de ce morphisme est triviale stablement en~$n$.
\end{enumerate}
\end{thi}

(Rappelons que ${\rm Tor}^{\mathbb{Z}\ltimes I}_1(\mathbb{Z},\mathbb{Z})\simeq I/I^2$, qui est non nul si $I$ est un idéal non trivial d'un anneau commutatif noethérien intègre, par exemple.)

Afin d'expliciter notre résultat principal, qui généralise le théorème précédent, introduisons quelques notations. Désignons par $\mathbf{ab}$ la sous-catégorie pleine de la catégorie $\mathbf{Ab}$ des groupes abéliens constituée des groupes abéliens libres de rang fini et par $\mathbf{S}(\mathbf{ab})$ la catégorie avec les mêmes objets dont les morphismes sont les monomorphismes scindés de $\mathbf{ab}$, le scindement étant donné dans la structure. L'action de $GL_n(\mathbb{Z})$ sur $H_*(GL_n(I);\mathbb{Z})$ et les morphismes de stabilisation font de $H_*(GL(I);\mathbb{Z})$ un foncteur sur $\mathbf{S}(\mathbf{ab})$, où l'on note $GL(I)$ le foncteur $n\mapsto GL_n(I)$. On renvoie à ce sujet à \cite[§\,5.2]{DV-pol}, dont on utilise également de façon cruciale la notion de {\em foncteur faiblement polynomial}, dont la définition et plusieurs propriétés seront rappelées dans le présent travail. La condition d'action stablement triviale de $GL_n(\mathbb{Z})$ sur $H_i(GL_n(I);\mathbb{Z})$ est équivalente au caractère stablement constant, ou encore faiblement polynomial de degré au plus $0$, du foncteur $H_i(GL(I);\mathbb{Z})$.

La deuxième assertion du théorème~\ref{this} montre que, si $e$ est le premier degré non excisif pour la $K$-théorie algébrique de $I$, $H_e(GL(I);\mathbb{Z})$ est faiblement polynomial de degré $2$, et décrit ce foncteur modulo les foncteurs faiblement polynomiaux de degré au plus $0$. Le théorème suivant, établi au §\,\ref{sconc}, étend ce résultat aux degrés homologiques supérieurs.

\begin{thi}\label{thip}
 Soient $I$ un anneau sans unité non excisif pour la $K$-théorie algébrique et $e>0$ le plus petit degré pour lequel cette excisivité tombe en défaut. Pour tout entier $n>0$, le foncteur $H_n(GL(I);\mathbb{Z}) : \mathbf{S}(\mathbf{ab})\to\mathbf{Ab}$ est faiblement polynomial de degré au plus $2m$, où $m$ désigne la partie entière de $n/e$.
 
 Si de plus $n$ est multiple de $e$, alors ce foncteur est exactement de degré faible $2m$. De surcroît, il est isomorphe modulo la catégorie des foncteurs faiblement polynomiaux de degré au plus $2m-2$ au foncteur
 $$\big(\mathfrak{gl}\big({\rm Tor}^{\mathbb{Z}\ltimes I}_e(\mathbb{Z},\mathbb{Z})\big)^{\otimes m}\big)_{\mathfrak{S}_m}$$
 où l'action du groupe symétrique $\mathfrak{S}_m$ par permutation des facteurs du produit tensoriel est tordue par la signature lorsque $e$ est impair.
\end{thi}

On obtient même un résultat plus général, qui vaut dans une situation hermitienne, sur une petite catégorie additive à dualité (théorèmes~\ref{thcong3} et~\ref{sus-gal}).

Notre stratégie (esquissée dans le mémoire non publié \cite{HDR}) consiste à utiliser l'homologie des foncteurs en développant une variante relative du cadre général de \cite{DV}, lui-même légèrement retravaillé dans \cite{Dja-JKT}, et dont s'inspire la notion de {\em catégorie homogène} introduite par Randal-Williams et Wahl \cite{RWW} à des fins d'étude unifiée de problèmes de stabilité homologique. On obtient ainsi (théorème~\ref{th-spf}) des suites spectrales donnant des renseignements sur l'homologie stable (à coefficients constants ou tordus par des foncteurs) de groupes <<~de type congruence~>>, dans un contexte très général, qui inclut les sous-groupes des groupes d'automorphismes des groupes libres induisant l'identité sur l'abélianisation, à propos de l'homologie stable desquels on discute une conjecture (déjà formulée dans \cite{HDR}) dans l'appendice~\ref{a1}.

La démonstration du théorème~\ref{thip} nécessite ensuite l'étude de plusieurs produits tensoriels sur les catégories de foncteurs en jeu (notamment sur $\mathbf{S}(\mathbf{ab})$), ainsi que de leurs dérivés, sur des foncteurs polynomiaux. L'apparition du {\em produit tensoriel le long de la structure monoïdale} ambiante (la somme directe sur $\mathbf{S}(\mathbf{ab})$) constitue de fait la principale nouveauté des suites spectrales générales que l'on construit, en comparaison de \cite[§\,1]{DV}. Son comportement s'avère beaucoup plus difficile à comprendre que celui du produit tensoriel usuel. On l'élucide, pour les foncteurs polynomiaux sur $\mathbf{S}(\mathbf{ab})$ (ou des généralisations appropriées de cette catégorie), en montrant d'abord des résultats de comparaison de groupes d'extensions ou de torsion. Ceux-ci généralisent des propriétés homologiques qui interviennent de façon cruciale dans \cite{Dja-JKT} (dont des cas particuliers avaient déjà été obtenus dans \cite{DV}), à l'aide de variantes convenables du critère d'annulation (co)homologique de Scorichenko \cite{Sco}.

\paragraph*{Lien avec des résultats antérieurs}
Outre les travaux de Suslin et de Suslin-Wodzicki susmentionnés, plusieurs cas particuliers du théorème~\ref{thip} étaient connus avant la réalisation du présent article. Dans la lignée du travail de Charney \cite{Ch-cong} sur le problème de la stabilité homologique pour les groupes de congruence, Putman \cite{Pu} a démontré que l'homologie en chaque degré fixé, à coefficients dans un corps, des groupes de congruence associés aux idéaux de l'anneau des entiers a des dimensions polynomiales à partir d'un certain rang. Son travail a été amélioré à l'aide de l'utilisation des $FI$-modules (foncteurs sur les ensembles finis avec injections) par Church-Ellenberg-Farb-Nagpal \cite{CEFN} et Church-Ellenberg \cite{CE-cong}, qui donnent des résultats analogues pour de plus grandes classes d'anneaux sans unité, à coefficients entiers, avec parfois des bornes explicites sur le degré polynomial. Tout récemment, Church, Miller, Nagpal et Reinhold \cite{CMNR} ont effectué un progrès considérable, montrant que l'homologie en degré $n$ des groupes de congruence associés à un anneau sans unité de rang stable de Bass fini $d+2$ définit un foncteur polynomial de degré faible au plus $2n+d$. Signalons également que, pour les groupes de congruence associés à un idéal de l'anneau des entiers, Calegari \cite[lemme~4.5]{Cal} a montré un résultat essentiellement équivalent à notre théorème~\ref{thip}, par des méthodes complètement différentes.

Mesurer à quel point l'homologie des groupes de congruence associés à un anneau sans unité $I$ s'éloigne (stablement) d'un foncteur constant constitue une forme, homologique, de contrôle du défaut d'excisivité pour la $K$-théorie algébrique pour $I$. Une autre manière, homotopique, d'estimer ce défaut d'excisivité est donnée par la {\em $K$-théorie algébrique birelative}. Une description de celle-ci est donnée par Corti{\~n}as \cite{Cor} à coefficients rationnels et par Geisser-Hesselholt \cite{GHel} à coefficients de torsion. Toutefois, non seulement nos méthodes sont indépendantes de celles utilisées par ces auteurs, mais encore il semble fort délicat de tisser un lien entre nos {\em énoncés} homologiques et les leurs (au-delà du plus petit degré non excisif, pour lequel les deux se réduisent au résultat de Suslin susmentionné).

\paragraph*{Organisation de l'article} La section~\ref{ls1} introduit un cadre catégorique (largement inspiré de \cite{DV,Dja-JKT,RWW,DV-pol}) permettant de définir des notions appropriées de stabilisation (§\,\ref{pps}), de foncteurs polynomiaux (§\,\ref{pfp}), d'homologie stable de familles de groupes (§\,\ref{pcf}, utilisant les préparatifs du §\,\ref{pprep}), à coefficients constants ou tordus, dans un contexte <<~absolu~>> (dont un archétype est la famille des groupes linéaires sur un anneau unitaire) ou <<~relatif~>> (dont les groupes de congruence sont un archétype), puis au §\,\ref{sprt} une notion de produit tensoriel (éventuellement dérivé) le long d'une structure monoïdale convenable. Cette section préparatoire culmine au §\,\ref{sssf} où est construite (théorème~\ref{th-spf}) une suite spectrale pour l'homologie stable relative (reliant cette homologie à coefficients tordus à la même homologie à coefficients constants), qui fait intervenir ce produit tensoriel dérivé et des extensions de Kan dérivées. Il s'agit d'un pendant relatif des suites spectrales fondamentales de \cite[§\,1]{DV}. Cela permet de discuter, dans le cas général, la notion d'excisivité en homologie stable en termes de foncteurs dérivés (proposition~\ref{cnex}).

Dans la section~\ref{sf}, on applique ces considérations aux groupes de congruence (dans un cadre hermitien général, rappelé au §\,\ref{sherm}, qui inclut les groupes de congruence dans les groupes linéaires, unitaires ou symplectiques). Grâce à un résultat fondamental (§\,\ref{shd}) de comparaison de catégories dérivées inspiré de Scorichenko \cite{Sco} et de \cite{Dja-JKT}, on peut étudier, d'abord séparément, sur des foncteurs polynomiaux, les protagonistes des suites spectrales du §\,\ref{sssf} : extensions de Kan dérivées le long d'un foncteur de changement de base entre catégories d'objets hermitiens (§\,\ref{pkdf}) et produit tensoriel dérivé le long de la somme hermitienne (§\,\ref{spt}). Le §\,\ref{sconc} combine tous les résultats précédents pour donner les résultats principaux de l'article : le théorème~\ref{sus-gal}, qui généralise le théorème~\ref{this} de Suslin, le théorème~\ref{thcong3}, qui, combiné au précédent, précise le théorème~\ref{thip}, dans un cadre catégorique hermitien plus général, de considérations relatives aux structures (co)multiplicatives, ainsi qu'un résultat à coefficients tordus, le corollaire~\ref{cor-ctf}.

L'appendice~\ref{a1} présente nos vues sur la façon dont les suites spectrales du §\,\ref{sssf} pourraient permettre d'aborder l'homologie stable des sous-groupes $IA$ des groupes d'automorphismes des groupes libres.

L'appendice~\ref{a2} expose, sans prétention à l'originalité, quelques résultats auxiliaires sur des foncteurs quadratiques sur une catégorie additive intervenant dans un contexte hermitien, qui sont utilisés çà et là dans la section~\ref{sf}.

\paragraph*{Remerciements} L'auteur est reconnaissant envers Christine Vespa, Antoine Touzé, Jacques Darné, Christian Ausoni, Vincent Franjou et Serge Bouc de discussions utiles reliées à différentes questions abordées dans ce travail. Celui-ci a par ailleurs bénéficié du soutien partiel du projet ANR-16-CE40-0003 ChroK ; l'auteur ne soutient pas pour autant l'ANR, dont il revendique la restitution des moyens aux laboratoires sous forme de crédits récurrents.

\subsubsection*{Conventions et notations générales}

Sauf mention expresse du contraire, toutes les catégories considérées seront localement petites (i.e. telles que la classe des morphismes entre deux objets donnés forme toujours un ensemble). Si $\C$ est une catégorie, on désigne par ${\rm Ob}\,\C$ sa classe d'objets, par $\C(a,b)$ (ou quelquefois ${\rm Hom}_\C(a,b)$) l'ensemble des morphismes d'un objet $a$ vers un objet $b$ de $\C$, et par $\C^{op}$ la catégorie opposée à $\C$.

Si $\C$ est une catégorie (essentiellement) petite et $\A$ une catégorie, on note $(\C,\A)$ la catégorie des foncteurs de $\C$ dans $\A$. Si $\Phi : \C\to\D$ est un foncteur entre petites catégories, on note $\Phi^* : (\D,\A)\to (\C,\A)$ le foncteur de précomposition par $\Phi$. On note $\Phi_* : (\C,\A)\to (\D,\A)$ l'extension de Kan à gauche le long de $\Phi$, lorsqu'elle existe (ce qui est automatique si $\A$ est cocomplète), c'est-à-dire l'adjoint à gauche à $\Phi^*$.

On note $\mathbf{Ens}$, $\mathbf{Grp}$ et $\mathbf{Ab}$ les catégories des ensembles, des groupes et des groupes abéliens respectivement. On note $\Theta$ la sous-catégorie de $\mathbf{Ens}$ dont les objets sont les $\mathbf{n}:=\{1,\dots,n\}$ ($n\in\mathbb{N}$) et les morphismes les fonctions {\em injectives}.

Si $M$ est un objet d'une catégorie abélienne cocomplète et $E$ un ensemble, on note $M[E]$ la somme directe de copies de $M$ indexées par $E$.

Si $F$ est un foncteur d'une catégorie $\C$ vers $\mathbf{Ens}$, on note $\C_F$ la catégorie d'éléments associée (ses objets sont les couples $(c,\xi)$ formés d'un objet $c$ de $\C$ et d'un élément $\xi$ de $F(c)$) et $\pi_F : \C_F\to\C$ le foncteur d'oubli ; la construction duale à partir d'un foncteur $G : \C^{op}\to\mathbf{Ens}$ est notée $\C^G$ (c'est la catégorie opposée à $(\C^{op})_G$) et $\pi^G : \C^G\to\C$ désigne le foncteur d'oubli. Si $\A$ est une catégorie avec coproduits, $\C$ une petite catégorie et $F : \C\to\mathbf{Ens}$ un foncteur, on dispose d'un foncteur $\Omega_F : (\C_F,\A)\to (\C,\A)$ donné sur les objets par
\begin{equation}\label{eqom}
 \Omega_F(X)(c)=\underset{\xi\in F(c)}{\bigsqcup}X(c,\xi)
\end{equation}
pour $X$ dans $(\C_F,\A)$ et $c$ dans $\C$. Le foncteur $\Omega_F$ est adjoint à gauche au foncteur $\pi_F^* : (\C,\A)\to (\C_F,\A)$ --- autrement dit, $\Omega_\F\simeq (\pi_F)_*$.

\smallskip

La lettre $\kk$ désigne un anneau commutatif ; la mention à cet anneau sera parfois omise des produits tensoriels, par exemple. On désigne par $\kk\md$ la catégorie des $\kk$-modules à gauche et on note $(\C,\kk)$ pour $(\C,\kk\md)$ lorsque $\C$ est une petite catégorie. On note $P^\C_x$, pour $x\in {\rm Ob}\,\C$, l'objet $\kk[\C(x,-)]$ de $(\C,\kk)$ ; il représente l'évaluation en $x$ et est donc projectif de type fini.

Si $\C$ et $\M$ sont des catégories $\kk$-linéaires (c'est-à-dire enrichies sur $\kk\md$), avec $\C$ petite, on note $(\C,\M)_\kk$ la catégorie des foncteurs $\kk$-linéaires de $\C$ vers $\M$. On note $\C\md_\kk$ pour $(\C,\kk\md)_\kk$ ; la catégorie $\C\md_\mathbb{Z}$ sera notée simplement $\C\md$, et $\mdd\C$ désignera $\C^{op}\md$. On dispose ainsi d'une équivalence canonique $(\kk[\D],\M)_\kk\simeq (\D,\M)$ (et en particulier $\kk[\D]\md_\kk\simeq (\D,\kk)$) pour toute catégorie petite $\D$, $\kk[\D]$ désignant la catégorie $\kk$-linéarisée associée.

Si $\M$ est une catégorie abélienne $\kk$-linéaire cocomplète, on dispose d'un produit tensoriel $\underset{\kk}{\otimes} : \kk\md\times\M\to\M$ caractérisé par l'adjonction
$$\M(V\underset{\kk}{\otimes} A,B)\simeq(\kk\md)\big(V,\M(A,B)\big)$$
(par commodité, on notera parfois $A\underset{\kk}{\otimes} V$ plutôt que $V\underset{\kk}{\otimes}A$).

La plupart des catégories abéliennes qu'on considérera dans cet article seront des catégories de {\em Grothendieck} (i.e. cocomplètes, avec colimites filtrantes exactes et possédant un générateur) ; on pourra trouver des généralités à leur sujet dans \cite{Gab}, par exemple.

Si $\C$ est une petite catégorie $\kk$-linéaire et $\M$ une catégorie de Grothendieck $\kk$-linéaire, le foncteur {\em produit tensoriel au-dessus de $\C$}
$$-\underset{\C}{\otimes}- : (\C^{op},\M)_\kk\times (\C\md_\kk)\to\M$$
(on intervertira parfois les deux arguments, ce qui ne pose pas de problème en raison de la symétrie de cette structure monoïdale lorsque $\M=\kk\md$) est défini comme la composée du produit tensoriel (au-dessus de $\kk$, au sens défini précédemment) extérieur $(\C^{op},\M)_\kk\times (\C\md_\kk)\to (\C^{op}\times\C,\M)_\kk$ et du foncteur cofin ($\kk$-linéaire). Ce bifoncteur se dérive à gauche par rapport à l'un ou l'autre des arguments de la façon usuelle, donnant lieu à des foncteurs notés ${\rm Tor}^\C_\bullet$.

\begin{remi}\label{rq-comptor}
Si $\C$ est une catégorie préadditive (i.e. $\mathbb{Z}$-linéaire), alors la catégorie de foncteurs préadditifs $\C\md$ est une sous-catégorie pleine de $(\C,\mathbb{Z})$ ; si $F$ et $G$ sont des objets de $\mdd\C$ et $\C\md$ respectivement, on dispose d'un morphisme canonique ${\rm Tor}_\bullet^{\kk[\C]}(F,G)\to {\rm Tor}_\bullet^\C(F,G)$ qui est un isomorphisme en degré $0$, mais généralement pas en degré supérieur.
\end{remi}

Si $\A$ est une catégorie abélienne, on note $\mathbf{D}(\A)$ sa catégorie dérivée (sans aucune condition de borne sur les complexes\,\footnote{On pourrait se limiter pour nos considérations à des complexes bornés d'un côté (d'ailleurs, l'usage des catégories dérivées dans le présent travail n'est pas essentiel), mais il est plus commode et satisfaisant de traiter de complexes arbitraires.}). Une telle catégorie n'a pas a priori de raison d'être localement petite (même en supposant que $\A$ l'est, bien sûr), mais c'est le cas si $\A$ est une catégorie Grothendieck --- cf. \cite{AJS,Serp} (le travail antérieur \cite{BN93} suffit pour les catégories de foncteurs du type $\C\md_\kk$ où $\C$ est une petite catégorie $\kk$-linéaire). Si $\B$ est une sous-catégorie épaisse d'une catégorie abélienne $\A$, on note $\mathbf{D}_\A(\B)$ la sous-catégorie pleine (qui est triangulée) de $\mathbf{D}(\A)$ formée des complexes dont l'homologie appartient à $\B$. Cette sous-catégorie contient l'image essentielle du foncteur $\mathbf{D}(\B)\to\mathbf{D}(\A)$ induit par l'inclusion $\B\hookrightarrow\A$ ; en fait, on vérifie facilement :
\begin{pri}\label{compext-gal}
 Soit $\B$ une sous-catégorie localisante d'une catégorie de Grothendieck $\A$. Notons $i : \B\to\A$ le foncteur d'inclusion. Les assertions suivantes sont équivalentes :
 \begin{enumerate}
 \item pour tous objets $X$ et $Y$ de $\B$, le morphisme naturel ${\rm Ext}^\bullet_\B(X,Y)\to {\rm Ext}^\bullet_\A(iX,iY)$ est un isomorphisme ;
  \item le foncteur $\mathbf{D}(i) : \mathbf{D}(\B)\to\mathbf{D}(\A)$ est pleinement fidèle ;
  \item l'image essentielle de $\mathbf{D}(i)$ est égale à $\mathbf{D}_\A(\B)$.
 \end{enumerate}
\end{pri}

La catégorie $\mathbf{D}_\A(\B)$ n'est autre que le noyau du foncteur canonique $\mathbf{D}(\A)\to\mathbf{D}(\A/\B)$ ; de plus, le foncteur triangulé induit $\mathbf{D}(\A)/\mathbf{D}_\A(\B)\to\mathbf{D}(\A/\B)$ est une équivalence (on renvoie à Neeman \cite[chapitre~2]{Nee} pour ce qui concerne la localisation des catégories triangulées). Par conséquent, si les conditions de la proposition précédente sont vérifiées, on dispose d'une équivalence canonique $\mathbf{D}(\A)/\mathbf{D}(\B)\simeq\mathbf{D}(\A/\B)$.

\section{Cadre formel pour l'homologie stable relative}\label{ls1}

Le cadre formel introduit dans \cite[§\,1]{DV} (ou ses variantes), bien adapté pour l'étude de l'homologie stable de groupes tels que les groupes symétriques, linéaires ou les groupes d'automorphismes des groupes libres, ne peut s'appliquer directement aux groupes de congruence. En effet, ceux-ci ne s'expriment pas naturellement comme groupes d'automorphismes de sommes itérées d'objets d'une catégorie monoïdale {\em symétrique} (Randal-Williams et Wahl ont montré dans \cite{RWW} qu'on pouvait affaiblir l'hypothèse de symétrie en une condition {\em pré-tressée}, mais celle-ci ne s'applique pas non plus ici). De fait, Charney a observé dans \cite{Ch-cong} que l'action des groupes linéaires sur les entiers (ou sa restriction aux groupes symétriques) sur l'homologie des groupes de congruence joue un rôle crucial dans le comportement stable de celle-ci : il faut tenir compte de toute la structure pour bien poser le problème, c'est-à-dire examiner {\em fonctoriellement} l'homologie des groupes de congruence, son étude {\em stable} ne pouvant se limiter à celle de sa colimite comme dans le cadre de \cite{DV}. Néanmoins, nous resterons dans ce même cadre, à quelques variations près : la différence proviendra de considérations {\em relatives}, c'est-à-dire liées à un foncteur monoïdal entre deux catégories adaptées à l'étude de l'homologie stable, à partir duquel on construit des groupes de type congruence, mais aussi une extension de Kan à gauche dérivée, entre lesquels une suite spectrale tisse des liens. Cette suite spectrale nécessite également l'introduction d'une structure multiplicative liée à la structure monoïdale ; celle-ci s'avère délicate à étudier (une partie importante de la section~\ref{sf} y sera dédiée), contrairement au produit tensoriel usuel qui seul intervient dans \cite{DV}. Nous ferons aussi quelques rappels et préparatifs autour de la notion de foncteur faiblement polynomial introduite dans \cite{DV-pol}.

\subsection{Foncteurs d'automorphismes}\label{pprep}

\begin{defi}
 Soit $\C$ une catégorie. On appelle {\em structure de foncteur d'automorphismes} sur $\C$ la donnée d'un foncteur ${\rm Aut}_\C : \C\to\mathbf{Grp}$ tel que, pour tout objet $x$ de $\C$, ${\rm Aut}_\C(x)$ soit le groupe des automorphismes de $x$ dans $\C$, et que l'axiome d'équivariance suivant soit vérifié : pour tout morphisme $f\in\C(x,y)$ de $\C$ et tout élément $u$ de ${\rm Aut}_\C(x)$, le diagramme
$$\xymatrix{x\ar[r]^f\ar[d]_u & y\ar[d]^{f_*(u)} \\
x\ar[r]^f & y
}$$
 de $\C$ commute, où l'on note $f_*$ pour ${\rm Aut}_\C(f) : {\rm Aut}_\C(x)\to {\rm Aut}_\C(y)$.

 Si $\C$ est munie d'une structure monoïdale, on dira qu'une structure de foncteur d'automorphismes sur $\C$ est {\em monoïdale} si les morphismes structuraux
 ${\rm Aut}_\C(x)\times {\rm Aut}_\C(y)\to {\rm Aut}_\C(x+y)$ sont fonctoriels en les objets $x$ et $y$ de $\C$.
 
 Soit $\Phi : \C\to\D$ un foncteur entre catégories munies de structures de foncteurs d'automorphismes. On dit que $\Phi$ est compatible aux foncteurs d'automorphismes si les morphismes de groupes ${\rm Aut}_\C(c)\to {\rm Aut}_\D(\Phi(c))$ induits par $\Phi$ définissent une transformation naturelle ${\rm Aut}_\C\to{\rm Aut}_\D\circ\Phi$ de foncteurs $\C\to\mathbf{Grp}$.
\end{defi}

\begin{rem}
 \begin{enumerate}
  \item L'axiome d'équivariance implique que l'image par le foncteur ${\rm Aut}_\C$ d'un isomorphisme est la conjugaison par celui-ci.
  \item Supposons que $(\C,+,0)$ est une catégorie monoïdale dont l'unité $0$ est objet initial et qui est munie d'une structure de foncteur d'automorphismes monoïdale. Alors pour tous objets $x$ et $y$ de $\C$, l'image du morphisme canonique $x\simeq x+0\to x+y$ par le foncteur ${\rm Aut}_\C$ coïncide avec le morphisme ${\rm Aut}_\C(x)\to {\rm Aut}_\C(x+y)$ induit par l'endofoncteur $-+y$ de $\C$.
  \item Voyant $\mathbf{Grp}$ comme sous-catégorie pleine de la catégorie des petites catégories de la façon usuelle, on voit que l'axiome d'équivariance permet de construire un foncteur $\int {\rm Aut}_\C\to\C$ (où l'intégrale désigne la construction de Grothendieck) qui est l'identité sur les objets et associe à une flèche $x\to y$ de la source, c'est-à-dire un couple formé d'un morphisme $u : x\to y$ de $\C$ et d'un automorphisme $g$ de $y$ (dans $\C$), la composée $g.u\in\C(x,y)$.
 \end{enumerate}
\end{rem}

La vérification des propriétés suivantes est immédiate.

\begin{prdef}\label{prdf-hfct}
 Soient $\C$ une catégorie munie d'une structure de foncteur d'automorphismes et $G : \C\to\mathbf{Grp}$ un sous-foncteur de ${\rm Aut}_\C$. 
 \begin{enumerate}
  \item Pour tout objet $t$ de $\C$, $G(t)$ est un sous-groupe distingué de ${\rm Aut}_\C(t)$.
  \item On définit une catégorie $\C/G$ de la manière suivante : ses objets sont les mêmes que ceux de $\C$, et ses ensembles de morphismes sont donnés par $(\C/G)(a,b):=\C(a,b)/G(b)$, et la composition est induite par celle de $\C$. On définit un foncteur (plein et essentiellement surjectif) $\Pi_G : \C\to\C/G$ qui est l'identité sur les objets et la surjection canonique sur les morphismes.
  \item Il existe une et une seule structure de foncteurs d'automorphismes sur $\C/G$ telle que $\Pi_G$ soit compatible aux structures de foncteurs d'automorphismes.
  \item Supposons que $\C$ est munie d'une structure monoïdale et que sa structure de foncteurs d'automorphismes lui est compatible. Supposons également que le sous-foncteur $G$ de ${\rm Aut}_\C$ est monoïdal (c'est-à-dire que l'image de $G(x)\times G(y)$ par ${\rm Aut}_\C(x)\times {\rm Aut}_\C(y)\to {\rm Aut}_\C(x+y)$ est incluse dans $G(x+y)$ pour tous objets $x$ et $y$ de $\C$). Alors il existe une et une seule structure monoïdale sur $\C/G$ telle que le foncteur $\Pi_G$ soit strictement monoïdal. La structure de foncteur d'automorphismes sur $\C/G$ est de plus monoïdale.
  \item Tout produit de catégories munies de structures de foncteurs d'automorphismes hérite d'une structure de foncteurs d'automorphismes (unique) telle que les foncteurs de projection y soient compatibles ; on l'appelle structure de foncteur d'automorphismes produit.
 \end{enumerate}
 \end{prdef}

 La vérification de la propriété suivante est également directe, à partir du fait que les automorphismes intérieurs agissent trivialement sur l'homologie d'un groupe. 

 \begin{pr}\label{prh1}
  Soient $\C$ une petite catégorie munie d'une structure de foncteur d'automorphismes, $G : \C\to\mathbf{Grp}$ un sous-foncteur de ${\rm Aut}_\C$ et $F$ un objet de $(\C,\kk)$. Pour tout $d\in\mathbb{N}$, il existe un (et un seul) foncteur $H_d(G;F)$ de $(\C/G,\kk)$ dont les valeurs sont données par l'homologie de groupe
  $$H_d(G;F)(x):=H_d(G(x);F(x))$$
pour tout objet $x$ de $\C$ et dont l'effet sur les morphismes est induit par celui de $F$. On définit ainsi un foncteur homologique $H_*(G;-) : (\C,\kk)\to (\C/G,\kk)$.
   
  La projection canonique sur les coïnvariants procure une transformation naturelle ${\rm Id}\to\Pi_G^* H_0(G;-)$ d'endofoncteurs de $(\C,\kk)$. 
 
 De plus, le foncteur $H_0(G;-)$ est canoniquement isomorphe à l'extension de Kan $(\Pi_G)_* : (\C,\kk)\to (\C/G,\kk)$ ; via cet isomorphisme, la transformation naturelle précédente s'identifie à l'unité de l'adjonction. Cet isomorphisme s'étend en un unique morphisme de foncteurs homologiques $H_\bullet(G;-)\to\mathbf{L}_\bullet(\Pi_G)_*$.
 \end{pr}
 
 Il est important de noter que la construction précédente ne se relève pas au niveau des catégories dérivées : sauf cas très particulier, $H_*(G;-)$ ne s'obtient pas en prenant l'homologie à partir d'un foncteur $\mathbf{D}(\C,\kk)\to\mathbf{D}(\C/G,\kk)$ (cette dernière catégorie n'étant généralement pas équivalente à la sous-catégorie de $\mathbf{D}(\C,\kk)$ des complexes dont l'homologie appartient à l'image essentielle de $\Pi_G^* : (\C/G,\kk)\to (\C,\kk)$).
 
\subsection{Un procédé de stabilisation}\label{pps}

On introduit maintenant une généralisation du procédé de stabilisation introduit dans \cite[§\,1.2]{DV-pol}.

\begin{defi}\label{df-sn}
 Soient $(\T,+,0)$ une petite catégorie monoïdale symétrique dont l'unité $0$ est objet initial et $\A$ une catégorie abélienne cocomplète.  On note ${\rm End}(\A)$ la {\em grosse} (i.e. pas nécessairement localement petite) catégorie constituée des endofoncteurs exacts de $\A$ commutant aux colimites ; on munit cette catégorie de la structure monoïdale (non symétrique) donnée par la composition des foncteurs.
 
 On suppose donné un foncteur $\tau : \T\to {\rm End}(\A)$ monoïdal {\em au sens fort} --- i.e. tel que les morphismes structuraux ${\rm Id}_\A\to\tau_0$ et $\tau_x\circ\tau_y\to\tau_{x+y}$ soient des {\em isomorphismes}, où l'on note $\tau_x$ pour $\tau(x)$.
 
 Pour tout objet $x$ de $\T$, on note $\kappa_x$ (respectivement $\delta_x$) le noyau (resp. conoyau) du morphisme naturel $i_x : {\rm Id}_\A\simeq\tau_0\to\tau_x$ (induit par l'unique flèche $0\to x$). On note $\kappa$ le sous-foncteur de ${\rm Id}_\A$ somme des sous-foncteurs $\kappa_x$ pour $x\in {\rm Ob}\,\T$.
 
 Un objet $A$ de $\A$ est dit {\em stablement nul} (relativement au foncteur $\tau$) si l'inclusion $\kappa(A)\subset A$ est une égalité. On note $\s n_\tau(\A)$, ou simplement $\s n(\A)$, la sous-catégorie pleine des objets stablement nuls de $\A$. 
\end{defi}

Se donner un foncteur $\T\to {\rm End}(\A)$ revient à se donner un foncteur $\T\times\A\to\A$ ; l'hypothèse monoïdale sur le premier revient à des hypothèses usuelles d'action sur le deuxième qu'on laisse au lecteur le soin d'écrire.

La première partie du résultat suivant s'établit comme \cite[corollaire~1.15]{DV-pol}, la deuxième est un exercice laissé au lecteur.
\begin{prdef}\label{df-st}
 Sous les hypothèses de la définition précédente, supposons également que les colimites filtrantes sont exactes dans $\A$. Alors $\s n_\tau(\A)$ est une sous-catégorie épaisse et stable par colimites de $\A$. On notera $\St_\tau(\A)$ (ou simplement $\St(\A)$) la catégorie quotient $\A/\s n_\tau(\A)$. Si $\A$ est une catégorie de Grothendieck, c'est donc également le cas de $\s n_\tau(\A)$.
 
 De plus, la classe des morphismes de $\A$ dont l'image dans $\St(\A)$ est inversible s'obtient en saturant par colimites filtrantes la classe des $f : A\to B$ tels qu'existe un objet $x$ de $\T$ et un morphisme $h : B\to\tau_x(A)$ de rendant commutatif le diagramme
$$\xymatrix{A\ar[r]^f\ar[d]_{i_x(A)} & B\ar[ld]^-h\ar[d]^{i_x(B)}\\
\tau_x(A)\ar[r]_-{\tau_x(f)} & \tau_x(B)
}$$
de $\A$.
\end{prdef}

\begin{rem}\label{rq-dfstgl}
 La dernière partie de cet énoncé fournit des généralisations <<~naturelles~>> de la définition~\ref{df-sn} au cas de l'action d'une petite catégorie monoïdale symétrique (noter que cette dernière condition pourrait elle-même être affaiblie) $\T$ dont l'unité est objet initial sur une catégorie $\A$ non nécessairement abélienne. Pour que cette définition se comporte bien, on peut supposer par exemple que $\A$ est complète et cocomplète et que $\T$ agit sur $\A$ par endofoncteurs commutant aux limites et aux colimites. On peut également supposer que $\A$ est triangulée et possède des sommes directes arbitraires et que $\T$ agit par endofoncteurs triangulés commutant aux sommes directes.
\end{rem}

Le cas qui nous intéressera est le suivant. Supposons que $(\T,+,0)$ est une petite catégorie monoïdale symétrique avec $0$ initial et $\C$ une petite catégorie sur laquelle $\T$ agit. Pour toute catégorie de Grothendieck $\M$, $\T$ agit (par précomposition par l'action sur $\C$) sur la catégorie de foncteurs $\A=(\C,\M)$. Le cas où $\C=\T$ (munie de l'action tautologique) est celui de \cite{DV-pol}. Dans ce contexte, on notera $\dst(\C,\M)$ la catégorie dérivée $\mathbf{D}(\St(\C,\M))$.

\begin{rem}\label{rq-dst}
 Supposons d'une manière générale donnée une action d'une catégorie monoïdale $\T$ sur une catégorie de Grothendieck comme dans la définition~\ref{df-sn}. On dispose de deux candidats <<~naturels~>> de {\em catégorie dérivée stabilisée} de $\A$ relativement à cette action :
 \begin{enumerate}
  \item la catégorie dérivée $\mathbf{D}(\St(\A))$ ;
  \item la catégorie triangulée $\St(\mathbf{D}(\A))$ construite à partir de l'action de $\T$ sur $\mathbf{D}(\A)$ induite par celle sur $\A$ (cf. la remarque~\ref{rq-dfstgl}).
 \end{enumerate}
 
 On vérifie sans trop de peine que ces deux définitions sont canoniquement équivalentes, en utilisant la remarque qui suit la proposition~\ref{compext-gal}. En revanche, les conditions de ladite proposition ne sont pas toujours vérifiées pour la sous-catégorie localisante $\s n(\A)$ de $\A$ : en utilisant \cite{Dja-pol} (notamment son appendice), on peut voir qu'elles sont vérifiées pour $\A=(\Theta,\kk)$ ou $\A=(\mathbf{S}(\mathbb{Z}),\kk)$, mais pas pour $\A=(\mathbf{S}(A),\kk)$ (où $\mathbf{S}(A)$ désigne la catégorie des $A$-modules à gauche $A^n$, pour $n\in\mathbb{N}$, les morphismes étant les monomorphismes $A$-linéaires munis d'un scindement) si l'anneau $A$ est assez compliqué. Les détails apparaîtront dans un travail ultérieur.
\end{rem}

\begin{prdef}\label{fonct-st}
 Soient $\C$ (resp. $\D$) une petite catégorie munie d'une action d'une petite catégorie monoïdale symétrique $\T$ (resp. $\T'$) dont l'unité est objet initial, $\Phi : \C\to\D$ un foncteur, $\Psi : \T\to\T'$ un foncteur monoïdal fort et $\M$ une catégorie de Grothendieck. On suppose que $\Phi$ et $\Psi$ sont compatibles aux actions au sens où le diagramme
 $$\xymatrix{\T\times\C\ar[r]\ar[d]_{\Psi\times\Phi} & \C\ar[d]^\Phi \\
 \T'\times\D\ar[r] & \D
 }$$
 dont les flèches horizontales sont les actions commute (à isomorphisme près).
 \begin{enumerate}
\item Supposons que, pour tout objet $a$ de $\T'$, la catégorie $\T_{\Psi^*\T'(a,-)}$ est non vide (c'est-à-dire qu'il existe un morphisme de source $a$ et de but appartenant à l'image de $\Psi$) --- c'est par exemple le cas si $\Psi$ est essentiellement surjectif. Alors le foncteur de précomposition $\Phi^* : (\D,\M)\to (\C,\M)$ préserve les foncteurs stablement nuls, il induit donc un foncteur encore noté $\Phi^* : \St(\D,\M)\to\St(\C,\M)$. Ce foncteur est exact et commute aux colimites. 

Ces foncteurs exacts $(\D,\M)\to (\C,\M)$ et $\St(\D,\M)\to\St(\C,\M)$ induisent des foncteurs triangulés $\mathbf{D}(\D,\M)\to\mathbf{D}(\C,\M)$ et $\dst(\D,\M)\to\dst(\C,\M)$ qu'on notera toujours $\Phi^*$.
\item L'extension de Kan $\Phi_* : (\C,\M)\to (\D,\M)$ induit un foncteur encore noté $\Phi_* : \St(\C,\M)\to\St(\D,\M)$. Celui-ci est adjoint à gauche à $\Phi^*$ sous l'hypothèse du point précédent.
\item Le foncteur dérivé total à gauche $\mathbf{L}\Phi_* : \mathbf{D}(\C,\M)\to\mathbf{D}(\D,\M)$, qui est adjoint à gauche à $\Phi^* : \mathbf{D}(\D,\M)\to\mathbf{D}(\C,\M)$, induit un foncteur triangulé encore noté $\mathbf{L}\Phi_* : \dst(\C,\M)\to\dst(\D,\M)$, qui n'est autre que le foncteur dérivé total à gauche de $\Phi_* : \St(\C,\M)\to\St(\D,\M)$. Si l'hypothèse du premier point est satisfaite, ce foncteur est adjoint à gauche à $\Phi^* : \dst(\D,\M)\to\dst(\C,\M)$. De plus, pour $\M=\kk\md$, le foncteur $\mathbf{L}\Phi_* : \dst(\C,\kk)\to\dst(\D,\kk)$ est donné explicitement par
$$\mathbf{L}\Phi_*(X)=\Phi^*{\rm Yon}_{st}^\D\overset{\mathbf{L}}{\underset{\kk[\C]}{\otimes}}X$$
où ${\rm Yon}_{st}^\D : \D^{op}\to\St(\D,\kk)$ désigne le foncteur composé du plongement de Yoneda $\kk$-linéarisé $\D^{op}\to (\D,\kk)$ et du foncteur canonique $(\D,\kk)\to\St(\D,\kk)$.
 \end{enumerate}

\end{prdef}

\begin{proof}
On note d'abord que l'hypothèse de compatibilité de $\Phi$ et $\Psi$ aux actions peut se reformuler en $\tau_t\Phi^*\simeq\Phi^*\tau_{\Psi(t)}$ pour tout $t\in {\rm Ob}\,\T$.

 Comme $\Phi^* : (\D,\M)\to (\C,\M)$ commute aux colimites, pour établir que ce foncteur préserve les objets stablement nuls sous l'hypothèse du premier point, il suffit de voir qu'il envoie un foncteur $G$ de $(\D,\M)$ tel que $i_a(G) : G\to\tau_a(G)$ soit nul pour un $a\in {\rm Ob}\,\T'$ sur un foncteur stablement nul. Soit $t$ un objet de $\T$ tel qu'existe une flèche $a\to\Psi(t)$ dans $\T'$. La composée
 $$\Phi^*G\xrightarrow{\Phi^*(i_a(G))}\Phi^*\tau_a(G)\to\Phi^*\tau_{\Psi(t)}(G)\simeq\tau_t\Phi^*(G)$$
 (où la flèche du milieu est induite par $a\to\Psi(t)$) est nulle car $i_a(G)=0$ et égale à $i_t(\Phi^* G)$, ce qui montre bien que $\Phi^*$ préserve les foncteurs stablement nuls. Le reste de la première assertion s'en déduit aussitôt.

 Montrons que $\Phi_* : (\C,\M)\to (\D,\M)$ passe au quotient pour induire un foncteur $\St(\C,\M)\to\St(\D,\M)$. La proposition~\ref{df-st} et la commutation de $\Phi_*$ aux colimites montrent qu'il suffit de vérifier que, si 
 $f : A\to B$ est un morphisme de $(\C,\M)$ tel qu'existe un objet $t$ de $\T$ et un morphisme $h$ rendant le diagramme 
 $$\xymatrix{A\ar[r]^f\ar[d]_{i_t(A)} & B\ar[d]^{i_t(B)}\ar[ld]_h\\
 \tau_t(A)\ar[r]_{\tau_t(f)} & \tau_t(B)
 }$$
 commutatif, alors $\Phi_*(f)$ devient inversible dans $\St(\D,\M)$. Comme on dispose d'une transformation naturelle $\Phi_*\tau_t\to\tau_{\Psi(t)}\Phi_*$, adjointe à la composée du morphisme $\tau_t\to\tau_t\Phi^*\Phi_*$ obtenu en composant l'unité avec $\tau_t$ et de l'isomorphisme $\tau_t\Phi^*\Phi_*\simeq\Phi^*\tau_{\Psi(t)}\Phi_*$, on déduit du diagramme précédent un diagramme commutatif
 $$\xymatrix{\Phi_*(A)\ar[rr]^{\Phi_*(f)}\ar[d]_{\Phi_*i_t(A)} & & \Phi_*(B)\ar[d]^{\Phi_*i_t(B)}\ar[lld]_{\Phi_*(h)}\\
 \Phi_*\tau_t(A)\ar[rr]_{\Phi_*\tau_t(f)}\ar[d] & & \Phi_*\tau_t(B)\ar[d]\\
 \tau_{\Psi(t)}\Phi_*(A)\ar[rr]_{\tau_{\Psi(t)}\Phi_*(f)} & & \tau_{\Psi(t)}\Phi_*(B).
 }$$
 On obtient la conclusion souhaitée sur $\Phi_*(f)$ en observant que la composée de $\Phi_*(i_t) : \Phi_*\to\Phi_*\tau_t$ et de notre transformation naturelle $\Phi_*\tau_t\to\tau_{\Psi(t)}\Phi_*$ coïncide avec $i_{\Psi(t)}\Phi_*$. Toute la fin de la démonstration s'établit de manière formelle.
\end{proof}

\begin{rem}
\begin{enumerate}
 \item Dans la démonstration de la deuxième assertion, il ne suffit {\em a priori} pas de montrer que $\Phi_*$ préserve les foncteurs stablement nuls, puisque $\Phi_*$ n'est généralement pas exact.
 \item La deuxième assertion montre que, sous les hypothèses du premier point, le foncteur $\Phi^* : \St(\D,\M)\to\St(\C,\M)$ commute aux limites, ce qui n'est pas clair à première vue.
\end{enumerate}
\end{rem}

Nous utiliserons un peu plus loin le critère simple suivant garantissant que la précomposition par un foncteur induit une équivalence entre catégories $\St$ (dans le cadre de \cite{DV-pol} ; on pourrait généraliser mais nous n'en aurons pas usage).

\begin{defi}\label{df-stabf}
 Soit $\Phi : \C\to\D$ un foncteur fortement monoïdal entre petites catégories monoïdales symétriques dont l'unité est objet initial. On dit que $\Phi$ est :
 \begin{enumerate}
  \item \textbf{faiblement surjectif} si, pour tout objet $d$ de $\D$, il existe des objets $c$ de $\C$ et $t$ de $\D$ tels que $t+d\simeq\Phi(c)$ ;
   \item \textbf{stablement fidèle} si deux flèches de $\C(t,x)$ ayant la même image dans $\D(\Phi t,\Phi x)$ ont la même image dans $\C(t,x+u)$ (via la post-composition par le morphisme canonique $x\to x+u$) pour un objet $u$ de $\C$ convenable.
  \item \textbf{stablement plein} si, pour toute flèche $\alpha$ dans $\D(\Phi a,\Phi b)$, il existe un objet $t$ de $\C$ tel que l'image canonique de $\alpha$ dans $\D(\Phi a,\Phi t+\Phi b)$ appartienne à l'image de la fonction $\C(a,t+b)\to\D(\Phi a,\Phi(t+b))\simeq\D(\Phi a,\Phi t+\Phi b)$ induite par $\Phi$ ; 
 \end{enumerate}
\end{defi}

\begin{pr}\label{pr-eqstn}
 Soient $\Phi : \C\to\D$ un foncteur fortement monoïdal entre petites catégories monoïdales symétriques dont l'unité est objet initial et $\M$ une catégorie de Grothendieck. On suppose que $\Phi$ est faiblement surjectif, stablement plein et stablement fidèle.

 Alors la précomposition par $\Phi$ induit une équivalence de catégories $\Phi^* : \St(\D,\M)\to\St(\C,\M)$.
\end{pr}

\begin{proof}
La proposition~\ref{fonct-st} montre que $\Phi^* : (\D,\M)\to (\C,\M)$ et son adjoint à gauche $\Phi_*$ induisent une adjonction entre $\St(\C,\M)$ et $\St(\D,\M)$. Le foncteur $\Phi_*$ est caractérisé par sa commutation aux colimites et un isomorphisme $\Phi_* M[\C(c,-)]\simeq M[\D(\Phi(c),-)]$ naturel en les objets $M$ de $\M$ et $c$ de $\C$. L'unité de l'adjonction, sur $M[\C(c,-)]$, n'est autre que le morphisme $M[\C(c,-)]\to\Phi^*M[\D(\Phi(c),-]$ qu'induit $\Phi$. L'hypothèse que $\Phi$ est stablement plein (resp. stablement fidèle) implique que ce morphisme induit un épimorphisme (resp. monomorphisme) dans la catégorie $\St(\C,\M)$. Comme $\Phi^*$ et $\Phi_*$ commutent aux colimites et que les $M[\C(c,-)]$ engendrent la catégorie $(\C,\M)$, on en déduit que l'unité ${\rm Id}\to\Phi^*\Phi_*$ est un isomorphisme sur la catégorie $\St(\C,\M)$, c'est-à-dire que $\Phi_* : \St(\C,\M)\to\St(\D,\M)$ est pleinement fidèle.

Si $t$ et $d$ sont deux objets de $\D$ et $\M$ un objet de $\M$, le morphisme canonique $d\to t+d$ induit un morphisme $M[\D(t+d,-)]\to M[\D(d,-)]$ qui devient un épimorphisme dans $\St(\D,\M)$, car $i_t$ est nul sur son conoyau. Par conséquent, l'hypothèse de faible surjectivité de $\Phi$ montre que les images des $M[\D(\Phi(c),-)]$ dans $\St(\D,\M)$ engendrent cette catégorie. Comme ils appartiennent à l'image essentielle de $\Phi_*$, qui commute aux colimites, ce foncteur est essentiellement surjectif.

Ainsi $\Phi_* : \St(\C,\M)\to\St(\D,\M)$ et donc son adjoint $\Phi^* : \St(\D,\M)\to\St(\C,\M)$ sont des équivalences de catégories.
\end{proof}

\subsection{Foncteurs polynomiaux}\label{pfp}

\begin{defi}\label{df-pol-gl}
 Plaçons-nous dans le cadre de la définition~\ref{df-sn}. On dit qu'un objet $A$ de $\A$ est {\em fortement polynomial} de degré (fort) au plus $d\in\mathbb{N}\cup\{-1\}$ si pour toute famille $(t_0,\dots,t_d)$ de $d+1$ objets de $\T$, on a $\delta_{t_0}\dots\delta_{t_d}(A)=0$.
 
 Si de plus les hypothèses de la proposition-définition~\ref{df-st} sont vérifiées, on dit qu'un objet $X$ de $\St(\A)$ est {\em polynomial} de degré au plus $d$ s'il vérifie la même condition : pour toute famille $(t_0,\dots,t_d)$ d'objets de $\T$, $\delta_{t_0}\dots\delta_{t_d}(X)=0$.
 
Un objet de $\A$ est dit {\em faiblement polynomial} de degré (faible) au plus $d$ si son image par le foncteur canonique $\A\to\St(\A)$ vérifie cette condition.

On note $\pol_d(\A)$ la sous-catégorie pleine de $\St(\A)$ des objets polynomiaux de degré au plus $d$.
\end{defi}

On vérifie comme dans \cite{DV-pol} que $\pol_d(\A)$ est une sous-catégorie localisante de $\St(\A)$. Les objets de la sous-catégorie localisante de $\St(\A)$ engendrée par les différents $\pol_d(\A)$ sont dits {\em analytiques}\,\footnote{En général, les objets analytiques n'ont pas de raison de s'exprimer comme colimites d'objets polynomiaux, la classe des colimites d'objets polynomiaux n'étant pas {\em a priori} stable par extensions.}.

Dans la suite, on se placera exclusivement dans le cadre où la catégorie monoïdale $\T$ opère sur $(\C,\M)$, où $\C$ est une petite catégorie et $\M$ une catégorie de Grothendieck, par précomposition par une action de $\T$ sur $\C$. Lorsque $\T=\C$ (avec l'action tautologique), c'est exactement le cadre de \cite{DV-pol}, lequel nous suffira d'ailleurs pour une très grande partie des considérations ultérieures.

Dans ce cadre, on notera $\dsto(\C,\M)$ la sous-catégorie pleine de $\dst(\C,\M)$ des complexes dont l'homologie est analytique.

\medskip

Le résultat qui suit constitue une préparation à la démonstration de la proposition~\ref{pr-swp}.

\begin{pr}\label{pr-sw1}
 Soit $X$ un objet de $\St(\C,\A)$ (resp. $(\C,\A)$) et $t$ un objet de $\T$. On suppose que les deux morphismes canoniques $t\to t+t$ induisent le même morphisme $\tau_t(X)\to\tau_{t+t}(X)$. Alors $\delta_t(X)=0$ (resp. $\kappa_t\delta_t(X)=\delta_t(X)$).
 
 Si $d\geq -1$ est un entier tel que ces deux morphismes $\tau_t(X)\to\tau_{t+t}(X)$ induisent pour tout $t$ dans ${\rm Ob}\,\T$ (ou dans un ensemble de générateurs monoïdaux faibles) la même flèche dans $\St(\C,\A)/\mathcal{P}ol_d(\C,\A)$, alors $X$ appartient à $\mathcal{P}ol_{d+1}(\C,\A)$.
\end{pr}

\begin{proof}
 On considère le diagramme commutatif
 $$\xymatrix{\tau_t(X)\ar@{->>}[r]\ar[d] & \delta_t(X)\ar[d]\\
 \tau_t\tau_t(X)\ar@{->>}[r] & \tau_t\delta_t(X)
 }$$
 dont les flèches horizontales sont les épimorphismes canoniques et les flèches verticales des évaluations de la tranformation naturelle ${\rm Id}\to\tau_t$ de noyau $\kappa_t$ (et injective dans $\St(\C,\A)$). On dispose d'un isomorphisme canonique $\tau_t\tau_t(X)\simeq\tau_{t+t}(X)$ par l'intermédiaire duquel la flèche verticale de gauche provient de la transformation naturelle $\tau_t\to\tau_{t+t}$ induite par l'un des morphismes canoniques $t\to t+t$. Mais l'autre morphisme canonique induit un morphisme $\tau_t(X)\to\tau_{t+t}(X)\simeq\tau_t\tau_t(X)$ dont la composée avec la flèche horizontale inférieure du diagramme est {\em nulle}, d'où la proposition (le cas des catégories quotients se traite de la même façon, puisque $\tau_t$ et $\delta_t$ y passent).
\end{proof}

La variante suivante de la proposition~\ref{pr-sw1} nous servira dans la démonstration du théorème~\ref{thcong2}.

\begin{pr}\label{swa}
 Soit $X$ un foncteur de $(\C,\A)$ notons $Y$ la précomposition de $X$ par $\C\to\C\quad x\mapsto x+x$ et $i_1, i_2 : X\to Y$ les morphismes induits par les deux morphismes canoniques $x\to x+x$. Alors le noyau $N$ du morphisme $X\oplus X\to Y$ de composantes $i_1$ et $i_2$ est faiblement polynomial de degré au plus $0$.
\end{pr}

\begin{proof}
On montre en effet que, pour tous objets $x$ et $t$ de $\C$, l'inclusion $\kappa_{t+x}\delta_t(N)(x)\subset\delta_t(N)(x)$ est une égalité. Cette assertion équivaut à dire que l'image de $N(t+x)\to N(t+x+t+x)$ induit par le morphisme canonique $t+x=(t+x)+0\to (t+x)+(t+x)$ est incluse dans l'image du morphisme $N(t+x+x)\to N(t+x+t+x)$ induit par le morphisme canonique $t+x+x\simeq 0+x+(t+x)\to t+x+(t+x)$. Un énoncé encore plus fort est vrai : les images des morphismes $N(t+x)\to N(t+x+t+x)$ induit par les deux morphismes canoniques $t+x\to (t+x)+(t+x)$ sont égales (la définition de $N$ montre que pour tout objet $a$ de $\C$, les deux morphismes canoniques $a\to a+a$ induisent des morphismes $N(a)\to N(a+a)$ qui coïncident à l'involution $(r,s)\mapsto (-s,-r)$ de $N\subset X\oplus X$ près). Cela termine la démonstration.
\end{proof}

La définition et les deux propositions qui suivent seront utilisées au début du §\,\ref{sconc}, où apparaîtra clairement la raison de l'emploi de l'adjectif {\em triangulaire}.

\begin{defi}\label{df-triang1}
Soient $(\C,+,0)$ une petite catégorie monoïdale symétrique dont l'unité $0$ est objet initial et $t$ un objet de $\C$.
On suppose que $\C$ est munie d'une structure monoïdale de foncteur d'automorphismes. On appelle {\em donnée triangulaire en $t$ sur $\C$} tout ensemble $\mathfrak{T}$ de sous-foncteurs $T : \C\to\mathbf{Grp}$ de $\tau_t {\rm Aut}_\C$ tel que l'involution de tressage de $t+t$ appartienne au sous-groupe de ${\rm Aut}_\C(t+t)$ engendré par les $T(t)$ pour $T\in\mathfrak{T}$ et par l'image de ${\rm Aut}_\C(t)\xrightarrow{f\mapsto f+t}{\rm Aut}_\C(t+t)$.

Si $\M$ est une catégorie de Grothendieck et que $\C$ est munie d'une telle structure, nous appellerons {\em extension triangulaire} d'un foncteur $F : \C\to\M$ en $t$ relativement à cette structure la donnée, pour tout $T\in\mathfrak{T}$, d'un foncteur $\hat{F}_T : \C\to\M$ et de morphismes $F\xrightarrow{a_T}\hat{F}_T$,  $\hat{F}_T\xrightarrow{b_T}F$ et $\hat{F}_T\xrightarrow{j_T}\tau_t(F)$ de sorte que :
\begin{enumerate}
 \item la composée $b_T a_T : F\to F$ est l'identité ;
 \item la composée $j_T a_T : F\to\tau_t F$ est le morphisme canonique $i_t$ ;
 \item pour tout objet $x$ de $\C$ et tout $T\in\mathfrak{T}$, $\hat{F}_T(x)$ est muni d'une action de $T(x)$, naturelle en $x$, de sorte que $j_T(x) : \hat{F}_T(x)\to\tau_t(F)(x)=F(t+x)$ soit équivariant (où l'action au but provient de l'inclusion $T(x)\subset {\rm Aut}_\C(t+x)$), de même que $b_T(x) : \hat{F}_T(x)\to T(x)$, où le but est muni de l'action triviale de $T(x)$.
\end{enumerate}

On dira qu'une telle extension triangulaire est de degré au plus $d$ si, pour tout $T\in\mathfrak{T}$, le noyau de $b_T$ (qui s'identifie au conoyau de $a_T$) est faiblement polynomial de degré au plus $d$.
\end{defi}

\begin{pr}\label{pr-extri}
 Soient  $d$ et $n$ des entiers naturels et $(\C,+,0)$ une petite catégorie monoïdale symétrique dont l'unité $0$ est objet initial et $t$ un objet de $\C$.
On suppose que $\C$ est munie d'une structure monoïdale de foncteur d'automorphismes et d'une donnée triangulaire $\mathfrak{T}$ en $t$, et que $G : \C\to\mathbf{Grp}$ est un sous-foncteur monoïdal de ${\rm Aut}_\C$. On suppose donné pour tout $T\in\mathfrak{T}$ un sous-foncteur $T_G : \C\to\mathbf{Grp}$ de $T$ tel que :
\begin{enumerate}
 \item pour tout objet $x$ de $\C$, tout $g\in G(x)$ et tout $u\in T(x)$, l'élément $(g_*u).u^{-1}$ de $T(x)$ (où $g_* u$ désigne l'effet de l'action de $g\in G(x)\subset {\rm Aut}_\C(x)$ sur $u\in T(x)$) appartient à $T_G(x)$ ;
 \item pour tout objet $x$ de $\C$, $T_G(x)$ est un sous-groupe distingué de $T(x)$ ;
 \item le sous-foncteur $T_G$ de $T\subset\tau_t {\rm Aut}_\C$ est inclus dans $\tau_t G$ ;
 \item pour tous entiers $i\geq 0$ et $j\geq 1$ tels que $i+j=n$, le foncteur $H_i(G;H_j(T_G;\kk))$ de $(\C,\kk)$ est faiblement polynomial de degré au plus $d$. 
\end{enumerate}
Alors le foncteur $H_n(G;\kk)$ de $(\C,\kk)$ possède une extension triangulaire en $t$ de degré au plus $d$.
\end{pr}

(Même si le foncteur $H_n(G;\kk)$ est défini sur $\C/G$, il nous est plus commode de le considérer ici sur $\C$, c'est-à-dire de précomposer par le foncteur canonique $\C\to\C/G$, ce qui n'a aucun effet sur le degré polynomial.)

\begin{proof}
 Soit $T\in\mathfrak{T}$. Les morphismes canoniques $T\hookrightarrow\tau_t {\rm Aut}_\C$ et ${\rm Aut}_\C\to\tau_t {\rm Aut}_\C$ s'assemblent en un morphisme ${\rm Aut}_\C\ltimes T\to\tau_t {\rm Aut}_\C$ (parce que la fonctorialité de ${\rm Aut}_\C$ sur les automorphismes est donnée par la conjugaison). Les trois premières hypothèses faites sur $T_G$ fournissent par ailleurs un diagramme commutatif de foncteurs $\C\to\mathbf{Grp}$
 $$\xymatrix{G\ar[r]\ar[d] &  G\ltimes T_G\ar[r]\ar[d] &  \tau_t G\ar[d]\\
 {\rm Aut}_\C\ar[r] & {\rm Aut}_\C\ltimes T\ar[r] &  \tau_t {\rm Aut}_\C
 }$$
 où les flèches verticales sont des inclusions de sous-groupes distingués (la première assertion peut en effet se reformuler comme $[G(x),T(x)]\subset T_G(x)$, où tous les groupes en jeu sont vus comme sous-groupes de ${\rm Aut}_\C(x)\ltimes T(x)$). Posons $F:=H_n(G;\kk)$ et $\hat{F}_T:=H_n(G\ltimes T_G;\kk)$ : la ligne horizontale supérieure du diagramme précédent induit des morphismes
 $$F\xrightarrow{a_T}\hat{F}_T\xrightarrow{j_T}\tau_t(F)$$
 dans $(\C,\kk)$ dont la composée est le morphisme canonique ; $a_T$ possède une rétraction $b_T : \hat{F}_T\to F$ induite par la projection canonique $G\ltimes T_G\to G$ ; l'action de $T$ sur $\hat{F}_T$ induite par la conjugaison sur $G\ltimes T_G$ (ce foncteur vers les groupes et $T$ étant vus dans ${\rm Aut}_\C\ltimes T$) vérifie clairement les propriétés d'équivariance requises pour faire de $\hat{F}_T$ une extension triangulaire de $F$. La propriété de degré de cette extension se déduit directement de la dernière hypothèse et de la suite spectrale de Lyondon-Hochschild-Serre.
\end{proof}

Le résultat suivant généralise Suslin-Wodzicki \cite[proposition~1.5]{SW}.

\begin{pr}\label{pr-swp}
 Soient $\C$ une petite catégorie monoïdale symétrique dont l'unité est objet initial, munie d'une structure monoïdale de foncteur d'automorphismes, $\M$ une catégorie de Grothendieck, $F : \C\to\M$ un foncteur et $d\geq -1$ un entier. On suppose que, pour tout $t\in {\rm Ob}\,\C$, $\C$ possède une donnée triangulaire $\mathfrak{T}_t$ en $t$ et que $F$ possède une extension triangulaire y afférente de degré au plus $d$.
 
 Alors $F$ est faiblement polynomial de degré au plus $d+1$.
\end{pr}

\begin{proof}
Fixons $t\in {\rm Ob}\,\C$. Notons $G$ la somme pour $T\in\mathfrak{T}_t$ de l'image par $j_T : \hat{F}_T\to\tau_t(F)$ de ${\rm Ker}\,b_T$. Ainsi, $G$ est faiblement polynomial de degré au plus $d$. Les hypothèses d'équivariance pour $b_T$ et $j_T$ montrent que, pour tout $T$, sur l'image de $\hat{F}_T\xrightarrow{j_T}\tau_t(F)\twoheadrightarrow\tau_t(F)/G$, l'action de $T$ est {\em triviale}. On déduit maintenant de la relation $j_T a_T=i_t$ qu'il en est de même pour l'image de $F\xrightarrow{i_t}\tau_t(F)\twoheadrightarrow\tau_t(F)/G$, et ce pour tout $T\in\mathfrak{T}$.

On en déduit que, pour tout $x\in {\rm Ob}\,\C$, l'image de $\tau_t(F)(x)=F(t+x)\xrightarrow{i_t(F)(t+x)}F(t+t+x)=\tau_{t+t}(F)(x)\twoheadrightarrow\tau_{t+t}(F)(x)/G(t+x)$ est invariante par le sous-groupe de ${\rm Aut}_\C(t+t+x)$ engendré par les $T(t+x)$ pour $T\in\mathfrak{T}_t$ et par l'image de ${\rm Aut}_\C(t+x)\xrightarrow{f\mapsto f+t+x}{\rm Aut}_\C(t+t+x)$. Or cette image contient l'automorphisme de $t+t+x$ induite par le tressage sur $t+t$ (par hypothèse et par fonctorialité de $T$), de sorte qu'on peut appliquer la proposition~\ref{pr-sw1} pour conclure.
\end{proof}

Le résultat suivant se place dans le même cadre que \cite{DV-pol} dont il améliore légèrement, sous une hypothèse supplémentaire, le théorème~3.8, et utilise la notation tilde introduite au début de sa section~3.

\begin{thm}[Cf. \cite{DV-pol}, théorème~3.8]\label{th-tilde}
 Soient $(\C,+,0)$ une petite catégorie monoïdale symétrique dont l'unité est objet initial et $\M$ une catégorie de Grothendieck. On suppose que pour tout foncteur constant $C : \C\to\M$ et tout foncteur fortement monoïdal $A : \C\to\M$, on a ${\rm Ext}^i_{(\C,\M)}(A,C)=0$ pour $i\in\{0,1\}$. Alors le foncteur canonique $o : \C\to\widetilde{\C}$ induit pour tout entier $d\geq 1$ une équivalence
 $$\pol_d(\widetilde{\C},\M)/\pol_{d-2}(\widetilde{\C},\M)\to\pol_d(\C,\M)/\pol_{d-2}(\C,\M).$$
\end{thm}

\begin{proof}
 On reprend la démonstration du théorème~3.8 de \cite{DV-pol} qu'on adapte. Pour $d=1$, il s'agit de démontrer que $o^* : \pol_1(\widetilde{\C},\M)\to\pol_1(\C,\M)$ est une équivalence. Si $X$ est un objet de $\pol_1(\C,\M)$, comme $\pol_1(\widetilde{\C},\M)$ se scinde en le produit de la catégorie des foncteurs fortement monoïdaux $\widetilde{\C}\to\M$ et de la catégorie $\M$ (correspondant aux foncteurs constants), parce que l'unité de $\widetilde{\C}$ en est un objet {\em nul}, on peut écrire $\tilde{X}\simeq A\oplus C$ où $A$ est fortement monoïdal et $C$ constant. Le noyau $N$ et le conoyau $K$ de l'unité $X\to o^*\tilde{X}$ sont polynomiaux de degré au plus $0$, c'est-à-dire constants (cf. \cite{DV-pol}). La suite exacte
 $$0\to N\to X\to o^*A\oplus C\to K\to 0$$
 (on se permet d'écrire $C$ plutôt que $o^* C$, ce foncteur étant constant) ainsi que l'hypothèse ${\rm Hom}(o^*A,K)=0$ permettent de la réduire à une suite exacte courte
 $$0\to N\to X\to o^*A\oplus C'\to 0$$
 avec $C'={\rm Ker}\,(C\to K)$ constant, puis l'hypothèse ${\rm Ext}^1(o^*A,N)=0$ (noter que l'Ext peut être calculé indifféremment dans $(\C,\M)$ ou $\St(\C,\M)$ --- cf. \cite{DV-pol}) garantit $X\simeq o^* A\oplus C''$ où $C''$, constant, est une extension de $C'$ par $N$. Cela montre que $X$ appartient à l'image essentielle de $o^* : \pol_1(\widetilde{\C},\M)\to\pol_1(\C,\M)$. On en déduit aisément que ce foncteur est une équivalence.
 
 Le cas général s'en déduit par une récurrence sur le degré $d$ analogue à celle de la démonstration du théorème~3.8 de \cite{DV-pol}.
\end{proof}

\begin{rem}
 \begin{enumerate}
  \item Réciproquement, on vérifie sans peine (à partir du scindement entre foncteurs constants et foncteurs réduits sur une catégorie possèdant un objet nul) que, si le foncteur canonique $\C\to\tilde{\C}$ induit une équivalence de catégories $\pol_1(\tilde{\C},\M)\to\pol_1(\C,\M)$, alors ${\rm Ext}^i_{(\C,\M)}(A,C)=0$ pour $i\in\{0,1\}$, tout foncteur constant $C : \C\to\M$ et tout foncteur fortement monoïdal $A : \C\to\M$.
  \item L'hypothèse du théorème précédent n'est pas toujours satisfaite. Ainsi, dans la catégorie $(\Theta,\kk)$, le foncteur de linéarisation $\kk[-](=P^\Theta_{\mathbf{1}}$) est fortement monoïdal, mais ${\rm Hom}_{(\Theta,\kk)}(\kk[-],\kk)\simeq\kk$. En revanche, dans $(\Theta,\kk)$, on a ${\rm Ext}^1(A,C)=0$ si $A$ est fortement monoïdal et $C$ constant. On donne maintenant un exemple où cette condition n'est pas satisfaite.
  
  Considérons la catégorie $\C$ dont les objets sont les couples $(V,v)$ formés d'un espace vectoriel à gauche de dimension finie sur un corps $k$ fixé et d'un élément non nul $v$ de $V$, les morphismes $(V,v)\to (W,w)$ étant les applications linéaires $f : V\to W$ telles que $f(v)=w$. On définit sur $\C$ une structure monoïdale symétrique $+$ par $(V,v)+(W,w):=(U\underset{k}{+}V, \bar{u}=\bar{v})$ (où la barre désigne l'image canonique dans $U\underset{k}{+}V$ d'un élément de $U$ ou de $V$) ; l'objet $(k,1)$ de $\C$ est à la fois unité pour cette structure et objet initial. Le foncteur $A : (\C,k)\to k\md$ associant $V/v$ à $V$ est fortement monoïdal, mais ${\rm Ext}^1_{(\C,k)}(A,k)$ est non nul, la suite exacte canonique
  $$0\to k\to V\to V/v\to 0$$
  définissant une extension de foncteurs non triviale.
 \end{enumerate}
\end{rem}

\subsection{Catégories de modèles pour l'homologie stable}\label{pcf}

On donne maintenant une variante du cadre formel pour l'étude de l'homologie stable de familles de groupes introduit dans \cite[§\,1]{DV} et retravaillé dans \cite[§\,1]{Dja-JKT} (voir aussi la notion de {\em catégorie homogène} de Randal-Williams et Wahl \cite{RWW}). Au prix de quelques raffinements techniques, la définition ci-après pourrait être étendue (cf. remarque~\ref{rq-gldf} ci-après) sans que cela modifie fondamentalement les considérations ultérieures, mais cela n'aurait pas d'incidence sur les applications que nous avons en vue --- à savoir celles de la section~\ref{sf} et de l'appendice.

\begin{defi}\label{df-cform}
 On appelle {\em catégorie de modèles pour l'homologie stable} (CMHS en abrégé) une petite catégorie monoïdale symétrique $(\C,+,0)$ munie d'une structure de foncteur d'automorphismes (noté ${\rm Aut_\C}$) monoïdal vérifiant les conditions suivantes.
\begin{enumerate}
 \item L'unité $0$ de la structure monoïdale est objet initial de $\C$.
\item (Axiome de transitivité stable) Si $f, g : a\to b$ sont des flèches parallèles de $\C$, il existe un objet $t$ de $\C$ et un automorphisme $\alpha$ de $t+b$ faisant commuter le diagramme 
$$\xymatrix{a\ar[d]_f\ar[r]^g & b\ar[r] & t+b\ar[d]^\alpha\\
b\ar[rr] & & t+b
}$$
dans lequel les flèches non spécifiées sont le morphisme canonique déduit de ce que $0$ est objet initial de $\C$.
\item Pour tout objet $a$ de $\C$, le foncteur $\C\to\C_{\C(a,-)}$ associant à $t$ l'objet $a+t$ muni de la flèche canonique $a\to a+t$ est une équivalence.
\end{enumerate}

On dit que la catégorie de modèles pour l'homologie stable est {\em régulière} si le renforcement suivant de l'axiome de transivité stable est vérifié : pour tous objets $a$ et $b$ de $\C$, le groupe ${\rm Aut}_\C(b)$ opère transitivement sur l'ensemble $\C(a,b)$.

Un {\em morphisme de catégories de modèles pour l'homologie stable} est un foncteur monoïdal symétrique {\em au sens fort} (i.e. tel que les morphismes structuraux de compatibilité aux structures monoïdales soient des isomorphismes) qui est compatible aux foncteurs d'automorphismes.

Les catégories de modèles pour l'homologie stable forment ainsi une catégorie notée $\mathbf{CMHS}$.
\end{defi}

La dernière condition permet d'associer à toute flèche $f : a\to b$ de $\C$ un objet qu'on notera $b\underset{f}{-}a$ tel que $f$ soit isomorphe à la flèche canonique $a\to a+(b\underset{f}{-}a)$.

\begin{rem}\label{rq-quillen}
 À tout petit groupoïde monoïdal symétrique $(\mathfrak{G},+,0)$ on peut associer une CMHS, selon une construction classique due à Quillen, notée $<\mathfrak{G},\mathfrak{G}>$ dans \cite[page~3]{Gray} (voir aussi \cite[§,1.1]{RWW}, et \cite[§\,3]{DV-pol}, où la construction duale est utilisée). Cette CMHS est régulière si et seulement si $\mathfrak{G}$ est régulier au sens où $a+t\simeq b+t$ entraîne $a\simeq b$, où $a$, $b$ et $t$ sont des objets de $\mathfrak{G}$.
 
 Il n'est pas difficile de voir que, si $\C$ est une CMHS et que $\mathfrak{G}$ désigne le groupoïde monoïdal sous-jacent, on dispose d'un morphisme canonique de CMHS $<\mathfrak{G},\mathfrak{G}>\to\C$ qui est une {\em équivalence}.
\end{rem}

On vérifie aussitôt :

\begin{pr}\label{cmhs-quot}
 Soient $\C$ une CMHS et $G : \C\to\mathbf{Grp}$ un sous-foncteur monoïdal de ${\rm Aut}_\C$. Les conditions suivantes sont équivalentes :
 \begin{enumerate}
  \item le foncteur ${\rm Aut}_\C/G : \C\to\mathbf{Grp}$ envoie les flèches canoniques $a\to a+t$ sur des monomorphismes ;
  \item pour tous objets $a$ et $t$ de $\C$, le diagramme commutatif de groupes
$$\xymatrix{G(a)\ar[r]\ar[d] & G(a+t)\ar[d] \\
{\rm Aut}_\C(a)\ar[r] & {\rm Aut}_\C(a+t)
}$$
dont les flèches verticales sont les inclusions et les flèches horizontales sont induites par le morphisme canonique $a\to a+t$  est cartésien ;
  \item il existe une structure de CMHS sur $\C/G$ telle que le foncteur canonique $\Pi_G : \C\to\C/G$ soit un morphisme de CMHS.
 \end{enumerate}
Si elles sont vérifiées, cette structure de CMHS sur $\C/G$ est alors unique.
\end{pr}

(La structure de foncteur d'automorphismes sur $\C/G$ est celle de la proposition-définition~\ref{prdf-hfct}.)

\begin{rem}\label{rq-gldf}
\begin{enumerate}
 \item Les conditions précédentes ne sont pas automatiques ; par exemple, on peut <<~tronquer~>> le foncteur ${\rm Aut}_\C$ pour obtenir un sous-foncteur n'y satisfaisant pas (sur $\Theta$, cela revient à poser $G(\mathbf{n})=\Sigma_n$ pour $n>N$ et $G(\mathbf{n})=\{1\}$ sinon, où $N$ est un entier fixé).
 \item On pourrait affaiblir légèrement les axiomes des CMHS de manière à ce qu'un quotient $\C/G$ comme dans l'énoncé ci-dessus soit toujours une CMHS. Il faudrait pour cela relâcher, dans la dernière condition, l'hypothèse que le foncteur canonique $\C\to\C_{\C(a,-)}$ est fidèle en une hypothèse de stable fidélité. On pourrait également affaiblir les hypothèses d'essentielle surjectivité et de plénitude en des versions stables.
\end{enumerate}
\end{rem}

\begin{pr}\label{cmhs-prod}
 Le produit de deux CMHS, muni de la structure monoïdale produit et de la structure de foncteurs d'automorphismes produit, est une CMHS.
\end{pr}

La situation générale qui nous intéresse est la suivante. Supposons que $\Phi : \C\to\D$ est un morphisme de CMHS. Définissons, pour $x\in {\rm Ob}\,\C$,
$$\Gamma_\Phi(x):={\rm Ker}\,\big({\rm Aut}_\C(x)\to{\rm Aut}_\D(\Phi x)\big)$$
(le morphisme de groupes étant induit par la fonctorialité de $\Phi$).
On obtient ainsi un sous-foncteur monoïdal $\Gamma_\Phi : \C\to\mathbf{Grp}$ de ${\rm Aut}_\C$. 

On note que les conditions de la proposition~\ref{cmhs-quot} sont automatiquement vérifiées par le sous-foncteur $\Gamma_\Phi$ de ${\rm Aut}_\C$.

\begin{rem}\label{rq-sl}
 On pourrait généraliser les résultats du présent travail à la situation suivante : $G_\C : \C\to\mathbf{Grp}$ (resp. $G_\D : \D\to\mathbf{Grp}$) désignant un sous-foncteur de ${\rm Aut}_\C$ (resp. ${\rm Aut}_\D$) {\em contenant le sous-foncteur des commutateurs} $[{\rm Aut}_\C,{\rm Aut}_\C]$ (resp. $[{\rm Aut}_\D,{\rm Aut}_\D]$), si l'on suppose que la transformation naturelle ${\rm Aut}_\C\to\Phi^*{\rm Aut}_\D$ qu'induit $\Phi$ envoie $G_\C$ dans $\Phi^*G_\D$, on pourrait remplacer $\Gamma_\Phi$ par le noyau de la transformation naturelle naturelle $G_\C\to\Phi^*G_\D$, au prix de quelques raffinements techniques.
 
 Cette généralisation (qui permet notamment de traiter de groupes de con\-gruence dans le groupe {\em spécial} linéaire sur un anneau commutatif) est classique pour le problème de la stabilité homologique (voir \cite{RWW} pour un cadre général, fonctoriel, à propos de cette question). 
\end{rem}

Par la proposition--définition~\ref{prh1}, on obtient en considérant l'homologie des groupes $\Gamma_\Phi(x)$ à coefficients dans $F(x)$, où $F$ est un objet de $(\C,\kk)$, un foncteur homologique $H_\bullet(\Gamma_\Phi;-) : (\C,\kk)\to (\C/\Gamma_\Phi,\kk)$. Celui-ci induit un foncteur homologique $\St(\C,\kk)\to\St(\C/\Gamma_\Phi;\kk)$. Cela provient du résultat plus général suivant :

\begin{pr}\label{pr-hstd}
 Soient $\C$ une petite catégorie monoïdale symétrique dont l'unité est objet initial, munie d'une structure de foncteur d'automorphismes monoïdale, et $G$ un sous-foncteur de ${\rm Aut}_\C$. Si $F$ est un foncteur appartenant à $\s n(\C,\kk)$, alors $H_\bullet(G;F)$ appartient à $\s n(\C/G,\kk)$.
\end{pr}

\begin{proof}
 Par commutation de $H_\bullet(G;-)$ aux colimites filtrantes, on peut se ramener au cas où existe un $t\in {\rm Ob}\,\C$ tel que $i_t : F\to\tau_t(F)$ soit nul. La conclusion provient alors de ce que le morphisme
 $$i_t : H_\bullet(G(x);F(x))\to H_\bullet(G(t+x);F(t+x))$$
 se factorise de manière évidente par le morphisme
 $$H_\bullet(G(x);F(x))\to H_\bullet(G(x);F(t+x))$$
 induit par $i_t : F(x)\to F(t+x)$, qu'on a supposé nul.
\end{proof}

\begin{prdef}\label{prdfh}
 Supposons que $\Phi : \C\to\D$ est un morphisme de CMHS faiblement surjectif et stablement plein. Alors $\Phi^*$ induit une équivalence de catégories $\St(\D,\kk)\to\St(\C/\Gamma_\Phi;\kk)$. On note
 $$\Hh_\bullet(\Gamma_\Phi ; -) : \St(\C,\kk)\to\St(\D,\kk)$$
 le foncteur homologique composé du foncteur $\St(\C,\kk)\to\St(\C/\Gamma_\Phi;\kk)$ induit par $H_\bullet(\Gamma_\Phi;-)$ (via la proposition~\ref{prh1}) et d'un quasi-inverse de l'équivalence précédente.
 
 La proposition~\ref{prh1} fournit une transformation naturelle $\Hh_\bullet(\Gamma_\Phi;-)\to\mathbf{L}_\bullet(\Phi_*)$ de foncteurs homologiques $\St(\C,\kk)\to\St(\D,\kk)$ qui est un isomorphisme en degré $0$.
\end{prdef}

(Noter que remplacer $\D$ par l'image $\bar{\D}$ du foncteur $\Phi$ et ce foncteur par le foncteur $\bar{\Phi} : \C\to\bar{\D}$ qu'il induit ne modifie par $\Gamma_\Phi$, et que $\bar{\Phi}$ vérifie les hypothèses précédentes : celles-ci ne sont donc pas tellement restrictives pour l'étude de l'homologie de $\Gamma_\Phi$.)

\begin{proof}
 De manière générale, $\Phi$ induit un morphisme de CMHS
 $$\bar{\Phi} : \C/\Gamma_\Phi\to\D\;;$$
 celui-ci est faiblement surjectif (resp. stablement plein) si $\Phi$ l'est. Pour conclure par application de la proposition~\ref{pr-eqstn}, il suffit donc de montrer que $\bar{\Phi}$ est stablement fidèle.
 
 Soient $x$, $y$ des objets de $\C$ et $f$, $g$ des éléments de $\C(x,y)$ ayant la même image dans $\D(\Phi x,\Phi y)$. Par l'axiome de transitivité stable on voit que, quitte à remplacer $f$ et $g$ par leur composée avec un morphisme canonique du type $y\to t+y$, on peut supposer l'existence d'un automorphisme $\alpha$ de $y$ tel que $g=\alpha f$ ; par le troisième axiome on voit qu'on peut également supposer que $y$ est de la forme $x+t$ et que $g : x\to x+t$ est le morphisme canonique. De $\Phi(g)=\Phi(\alpha)\Phi(f)=\Phi(\alpha)\Phi(g)$ on tire donc que $\Phi(\alpha)$ appartient à l'image du morphisme canonique ${\rm Aut}_\D(\Phi(t))\to {\rm Aut}_\D(\Phi(x)+\Phi(t))\simeq {\rm Aut}_\D(\Phi(x+t))$. Il s'ensuit que l'automorphisme $\Phi(x+\alpha)$ de $\Phi(x+x+t)$ commute avec l'involution $\epsilon$ induite par le tressage sur $x+x$. Posons $\beta=(x+\alpha)\epsilon (x+\alpha^{-1})\epsilon$ : on a donc $\beta\in\Gamma_\Phi(x+x+t)$. On vérifie facilement que, en notant $\tilde{f}$ (resp. $\tilde{g}$) la composée de $f$ (resp. $g$) avec le morphisme canonique $y\to x+y(=x+x+t)$, on a $\tilde{g}=\beta\tilde{f}$, de sorte que $\tilde{f}$ et $\tilde{g}$ définissent le même morphisme dans $\C/\Gamma_\Phi$, ce qui achève la démonstration.
\end{proof}

\begin{rem}\label{rreg}
 Si $\Phi : \C\to\D$ est un morphisme de CMHS plein et essentiellement surjectif et que $\C$ est une CMHS {\em régulière}, on vérifie facilement que $\Phi$ induit même une équivalence $\C/\Gamma_\Phi\to\D$.
\end{rem}

Notre but consiste à étudier les $\Hh_\bullet(\Gamma_\Phi;X)$, notamment lorsque $X$ est constant --- mais pour ce faire nous aurons besoin de l'étude de cette homologie stable pour des $X$ assez généraux (au moins polynomiaux).

\medskip

La construction qu'on vient d'introduire possède la fonctorialité suivante. Supposons donné un diagramme de CMHS
\begin{equation}\label{dc-cmhs}
 \xymatrix{\C_1\ar[r]^{\Phi_1}\ar[d]_{\Psi} & \D_1\ar[d]^\Xi \\
\C_2\ar[r]_{\Phi_2} & \D_2
}
\end{equation}
commutatif (à isomorphisme près). Alors $\Psi$ induit des morphismes ${\rm Aut}_{\C_1}\to\Psi^* {\rm Aut}_{\C_2}$ et $\Gamma_{\Phi_1}\to\Psi^*\Gamma_{\Phi_2}$ de foncteurs $\C_1\to\mathbf{Grp}$. Si l'on suppose $\Phi_1$ et $\Phi_2$ faiblement surjectifs et stablement pleins, alors $\Psi$ induit un morphisme gradué
\begin{equation}\label{foncf}
 \Hh_\bullet(\Gamma_{\Phi_1};\Psi^* X)\to\Xi^*\Hh_\bullet(\Gamma_{\Phi_2};X)
\end{equation}
naturel en l'objet $X$ de $\St(\C_2,\kk)$.

La définition qui suit est calquée sur la situation usuelle de la $K$-théorie algébrique (cf. l'introduction). La construction $+$ ou la construction $Q$ de Quillen \cite{QK} peuvent d'ailleurs être appliquées dans le cadre des CMHS, ce qui permet, comme dans le cas des groupes de congruence classiques, de donner un critère homotopique équivalent à la définition homologique qui suit.

\begin{defi}\label{df-exci}
 On dit que $\Phi$ est {\em $\kk$-excisif pour l'homologie stable} si l'objet $\Hh_i(\Gamma_\Phi;\kk)$ de $\St(\D,\kk)$ est (stablement) constant pour tout entier $i$. On dit que $\Phi$ est $\kk$-excisif pour l'homologie stable jusqu'en degré $d$ si cela est vrai pour $i\leq d$.
\end{defi}

Si cela est vrai pour $\kk=\mathbb{Z}$, on dira simplement que $\Phi$ est excisif (jusqu'en degré $d$), ce qui implique que ce foncteur est $\kk$-excisif (jusqu'en degré $d$) pour tout anneau commutatif $\kk$ par la formule des coefficients universels.

\subsection{Produits tensoriels}\label{sprt}

Soit $\C$ une petite catégorie.

Le produit tensoriel usuel $\otimes=\underset{\kk}{\otimes} : (\C,\kk)\times (\C,\kk)\to (\C,\kk)$ se calcule au but --- c'est donc la composition du produit tensoriel extérieur avec la précomposition par la diagonale $\C\to\C\times\C$. Il passe aux catégories dérivées, fournissant $\overset{\mathbf{L}}{\underset{\kk}{\otimes}} : \mathbf{D}(\C,\kk)\times\mathbf{D}(\C,\kk)\to\mathbf{D}(\C,\kk)$. Il passe également aux catégories $\St$, de sorte qu'on obtient des foncteurs $\otimes : \St(\C,\kk)\times\St(\C,\kk)\to\St(\C,\kk)$ et $\overset{\mathbf{L}}{\underset{\kk}{\otimes}} : \dst(\C,\kk)\times\dst(\C,\kk)\to\dst(\C,\kk)$ (les foncteurs dérivés usuels seront notés ${\rm Tor}^\kk_\bullet$, ils se calculent aussi au but). Ces différents produits tensoriels usuels définissent des structures monoïdales symétriques.

Supposons maintenant que $\C$ est munie d'une structure monoïdale symétrique, notée $+$. On note alors $\underset{[\C,+]}{\otimes}$ le foncteur, appelé {\em produit tensoriel le long de $+$} et noté
$$\underset{[\C,+]}{\otimes} : (\C,\kk)\times (\C,\kk)\xrightarrow{\boxtimes}(\C\times\C,\A)\xrightarrow{\Delta[\C,+]}(\C,\kk)$$
composé du produit tensoriel extérieur $\boxtimes$ et de la {\em diagonale le long de $+$}
$$\Delta[\C,+] : (\C\times\C,\kk)\to (\C,\kk)$$
définie comme l'extension de Kan à gauche le long du foncteur $+ : \C\times\C\to\C$.
Cette diagonale et ce produit tensoriel coïncident avec les notions usuelles lorsque la structure monoïdale $+$ sur $\C$ est une somme catégorique.

Ce produit tensoriel se dérive à gauche relativement à l'une ou l'autre des variables de la façon usuelle si $P^\C_c\underset{[\C,+]}{\otimes}-$ est un foncteur exact pour tout objet $c$ de $\C$. C'est le cas si $\C$ est une CMHS, auquel cas on a, pour tout $c\in {\rm Ob}\,\C$, un endofoncteur $\omega^\C_c$ (noté parfois simplement $\omega_c$) de $(\C,\A)$ (pour toute catégorie de Grothendieck $\A$) défini sur les objets par 
$$\omega_c^\C(F)(t)=\bigoplus_{f\in\C(c,t)}F(t\underset{f}{-}c).$$
On déduit facilement du troisième axiome d'une CMHS (et de l'adjonction générale donnée par (\ref{eqom})) que $P^\C_c\underset{[\C,+]}{\otimes}-\simeq\omega_c^\C$. Les foncteurs dérivés usuels de $\underset{[\C,+]}{\otimes}$ seront notés ${\rm Tor}^{[\C,+]}_\bullet$.

Soient $t$ un objet de $\C$, $A$ et $F$ des objets de $(\C,\kk)$. Considérons le morphisme naturel $\tau_t(A)\underset{[\C,+]}{\otimes}F\to\tau_t(A\underset{[\C,+]}{\otimes}F)$ adjoint au morphisme $\tau_t(A)\boxtimes F\to\tau_t(A\underset{[\C,+]}{\otimes}F)\circ +$ dont l'évaluation sur $(x,y)$ est le morphisme
$$A(t+x)\otimes F(y)\to(A\underset{[\C,+]}{\otimes}F)(t+x+y)$$
unité de l'adjonction évaluée sur $t+x$ et $y$. On vérifie aisément que la composée de ce morphisme $\tau_t(A)\underset{[\C,+]}{\otimes}F\to\tau_t(A\underset{[\C,+]}{\otimes}F)$ et du morphisme $A\underset{[\C,+]}{\otimes}F\to\tau_t(A)\underset{[\C,+]}{\otimes}F$ induit par le morphisme canonique $A\to\tau_t(A)$ est le morphisme canonique $A\underset{[\C,+]}{\otimes}F\to\tau_t(A\underset{[\C,+]}{\otimes}F)$.
En raisonnant comme dans la démonstration du deuxième point de la proposition~\ref{fonct-st}, on en déduit que $-\underset{[\C,+]}{\otimes}F$ induit un endofoncteur de $\St(\C,\kk)$, puis que $\underset{[\C,+]}{\otimes}$ induit un foncteur (toujours noté de la même façon) $\St(\C,\kk)\times\St(\C,\kk)\to\St(\C,\kk)$.

Là encore, $\underset{[\C,+]}{\otimes}$ (resp. $\overset{\mathbf{L}}{\underset{[\C,+]}{\otimes}}$, dans le cas où le produit tensoriel le long de $+$ est équilibré)  définit des structures monoïdales symétriques sur $(\C,\kk)$ (resp. $\mathbf{D}(\C,\kk)$) et sur $\St(\C,\kk)$ (resp. $\dst(\C,\kk)$).  

Les produits tensoriels usuel, $\otimes$, et le long de $+$, $\underset{[\C,+]}{\otimes}$, sont liés de la façon suivante. Supposons que l'unité de $\C$ en est un objet initial. On note déjà que l'unité de $\underset{[\C,+]}{\otimes}$  est alors $\kk$, comme pour le produit tensoriel usuel. On dispose aussi d'un morphisme
$$P^\C_a\underset{[\C,+]}{\otimes}P^\C_b=P^\C_{a+b}\to P^\C_a\otimes P^\C_b$$
naturel en les objets $a$ et $b$ de $\C$ (il est induit par les morphismes canoniques $a\to a+b$ et $b\to a+b$), qui est un isomorphisme si $a$ ou $b$ égale $0$. On en déduit pour des raisons formelles un morphisme
\begin{equation}\label{ecpt}
F\underset{[\C,+]}{\otimes}G\to F\otimes G
\end{equation}
naturel en les objets $F$ et $G$ de $(\C,\kk)$, qui est un isomorphisme lorsque $F$ ou $G$ est constant, et qui est compatible aux structures monoïdales symétriques des produits tensoriels. Ce morphisme induit bien sûr des transformations naturelles analogues entre les deux structures monoïdales symétriques qu'on discute sur $\St(\C,\kk)$, $\mathbf{D}(\C,\kk)$ et $\dst(\C,\kk)$ (sous les hypothèses où l'on peut dériver raisonnablement).

L'étude du produit $\underset{[\C,+]}{\otimes}$ semble difficile en général ; nous montrerons plus tard des propriétés cruciales de ce foncteur sur des objets polynomiaux lorsque $\C$ est une catégorie d'objets hermitiens (ce qui nous suffira amplement). On a toutefois le résultat général simple suivant.

\begin{prdef}\label{pg-ptt}
 Soit $\C$ une CMHS.
 \begin{enumerate}
  \item Soient $t$ un objet de $\C$ et $F$ un objet de $(\C,\kk)$. Il existe un diagramme commutatif
  $$\xymatrix{P^\C_t\underset{[\C,+]}{\otimes}F\ar[r]\ar[d]_\simeq & P^\C_t\otimes F\ar[d]^\simeq \\
  \omega_t F\ar[r] & \omega_t\tau_t F
  }$$
  naturel en $F$, dans lequel la flèche supérieure est la transformation naturelle définie ci-dessus et la flèche inférieure est induite par le morphisme $i_t(F) : F\to\tau_t(F)$.
  \item L'endofoncteur $\omega_t$ de $\St(\C,\kk)$ est exact et fidèle ; il est adjoint à gauche à $\tau_t$.
  \item\label{p3ptt} Si la transformation naturelle $-\underset{[\C,+]}{\otimes}X\to -\otimes X$ d'endofoncteurs de $\St(\C,\kk)$ est un épimorphisme, alors l'objet $X$ de $\St(\C,\kk)$ appartient à $\pol_0(\C,\kk)$.
  \item Notons $\underset{[\C,+]}{\bar{\otimes}} : \St(\C,\kk)\times\St(\C,\kk)\to\St(\C,\kk)$ le conoyau de la transformation naturelle $\underset{[\C,+]}{\otimes}\to\otimes$, conoyau appelé {\em produit tensoriel réduit le long de $+$}. Ce bifoncteur commute aux colimites en chaque variable et est exact lorsqu'on fixe une des variables en l'image dans $\St(\C,\kk)$ d'un $P^\C_t$. On peut donc le dériver à gauche comme $\otimes$, obtenant des foncteurs dérivés usuels notés
  $$\overline{{\rm Tor}}^{[\C,+]}_\bullet : \St(\C,\kk)\times\St(\C,\kk)\to\St(\C,\kk)$$
  et un foncteur dérivé total
  $$\overset{\mathbf{L}}{\underset{[\C,+]}{\bar{\otimes}}} : \dst(\C,\kk)\times\dst(\C,\kk)\to\dst(\C,\kk).$$
  \item On dispose dans $\dst(\C,\kk)$ d'un triangle distingué
  $$X\overset{\mathbf{L}}{\underset{[\C,+]}{\otimes}}Y\to X\overset{\mathbf{L}}{\underset{\kk}{\otimes}}Y\to X\overset{\mathbf{L}}{\underset{[\C,+]}{\bar{\otimes}}}Y\to (X\overset{\mathbf{L}}{\underset{[\C,+]}{\otimes}}Y)[1]$$
  naturel en les objets $X$ et $Y$ de $\dst(\C,\kk)$ ; cela induit dans $\St(\C,\kk)$ une suite exacte longue
  $$\cdots\to{\rm Tor}^{[\C,+]}_n(F,G)\to {\rm Tor}^\kk_n(F,G)\to\overline{{\rm Tor}}^{[\C,+]}_n(F,G)\to{\rm Tor}^{[\C,+]}_{n-1}(F,G)\to\cdots$$
  $$\cdots\to F\underset{[\C,+]}{\otimes}G\to F\otimes G\to F\underset{[\C,+]}{\bar{\otimes}}G\to 0$$
  naturelle en les objets $F$ et $G$ de $\St(\C,\kk)$. En particulier, on dispose d'un isomorphisme naturel $\overline{{\rm Tor}}^{[\C,+]}_n(F,G)\simeq{\rm Tor}^{[\C,+]}_{n-1}(F,G)$ pour si $\kk$ est de dimension homologique inférieure ou égale à $n-2$ (et, si $\kk$ est un corps, $\overline{{\rm Tor}}^{[\C,+]}_1(F,G)$ s'identifie au noyau de la transformation naturelle $F\underset{[\C,+]}{\otimes}G\to F\otimes G$).
 \end{enumerate}
\end{prdef}

\begin{proof}
 La première assertion est immédiate à partir de la définition de la transformation naturelle $\underset{[\C,+]}{\otimes}\to\otimes$. Le seul point qui mérite un commentaire pour en déduire la deuxième assertion est la fidélité de $\omega_t$. Celle-ci est équivalente au fait que l'unité $X\to\tau_t\omega_t(X)$ est un monomorphisme. Or c'est déjà le cas dans $(\C,\kk)$, comme on peut le voir soit en notant que la formule définissant $\omega_t$ montre directement sa fidélité (sur $(\C,\kk)$), soit en notant que cette unité s'insère dans un diagramme commutatif
 $$\xymatrix{X\ar[r]\ar[rd] & \tau_t\omega_t(X)\ar[d]^\simeq\\
 & \tau_t(P^\C_t)\otimes X
 }$$
 dont la flèche oblique est induite par $\kk\to\tau_t(P^\C_t)$ venant de la transformation naturelle ensembliste $*\to\C(t,t+-)$ déduite du morphisme canonique $t\to t+x$.
 
 On déduit du premier point que  $P^\C_t\underset{[\C,+]}{\bar{\otimes}}-$ est isomorphe à $\omega_t\delta_t$. Comme $\delta_t$ est un endofoncteur exact de $\St(\C,\kk)$ (mais généralement pas de $(\C,\kk)$ !), cela donne formellement les deux dernières assertions.
 
 La troisième assertion s'obtient à partir de l'isomorphisme canonique $P^\C_t\underset{[\C,+]}{\bar{\otimes}}X\simeq\omega_t\delta_t(X)$ et de la fidélité des foncteurs $\omega_t$.
\end{proof}

\begin{rem}\label{rq-fcptt}
 Il n'y a pas de fonctorialité claire en $t$ de $\omega_t\tau_t$, car $\tau_t$ est fonctoriel {\em covariant} en $t$ tandis que $\omega_t$ est fonctoriel {\em contravariant} en $t$.
\end{rem}

\paragraph*{Fonctorialité en $\C$ de $\underset{[\C,+]}{\otimes}$}

Soient $\Phi : \C\to\D$ un morphisme de CMHS et $F$, $G$ des foncteurs dans $(\D,\kk)$. On dispose d'une application naturelle
\begin{equation}\label{eqf-pt1}
 \Phi^*(F)\underset{[\C,+]}{\otimes}\Phi^*(G)\to\Phi^*( F\underset{[\D,+]}{\otimes}G)
\end{equation}
dans $(\C,\kk)$, adjointe au morphisme
$$\Phi^*(F)\boxtimes\Phi^*(G)\simeq (\Phi\times\Phi)^*(F\boxtimes G)\to (\Phi\times\Phi)^*\big((F\underset{[\D,+]}{\otimes}G)\circ +\big)\simeq\Phi^*(F\underset{[\D,+]}{\otimes}G)\circ +$$
obtenu en appliquant $(\Phi\times\Phi)^*$ à l'unité $F\boxtimes G\to (F\underset{[\D,+]}{\otimes}G)\circ +$. Ce morphisme naturel fait commuter le diagramme
$$\xymatrix{\Phi^*(F)\underset{[\C,+]}{\otimes}\Phi^*(G)\ar[r]\ar[d] & \Phi^*(F\underset{[\D,+]}{\otimes}G)\ar[d]\\
\Phi^*(F)\otimes\Phi^*(G)\ar[r]^{\simeq} & \Phi^*(F\otimes G)
}$$
dont les flèches verticales sont données par~(\ref{ecpt}).

Par conséquent, (\ref{eqf-pt1}) induit un morphisme naturel
\begin{equation}\label{eqf-pt2}
 \Phi^*(F)\underset{[\C,+]}{\bar{\otimes}}\Phi^*(G)\to\Phi^*(F\underset{[\D,+]}{\bar{\otimes}}G)\;;
\end{equation}
bien sûr, les morphismes naturels~(\ref{eqf-pt1}) et~(\ref{eqf-pt2}) induisent des analogues dans $\St(\C,\kk)$, $\mathbf{D}(\C,\kk)$ et $\mathbf{D}_{st}(\C,\kk)$.

\paragraph*{Une compatibilité entre les produits $\otimes$ et $\underset{[\C,+]}{\otimes}$}

On dispose d'un morphisme 
\begin{equation}\label{2pt1}
 (T\otimes U)\underset{[\C,+]}{\otimes}(A\otimes B)\to (T\underset{[\C,+]}{\otimes}A)\otimes (U\underset{[\C,+]}{\otimes}B)
\end{equation}
naturel en les objets $A$, $B$, $T$ et $U$ de $(\C,\kk)$ (ou $\St(\C,\kk)$), adjoint au morphisme
$$(T\otimes U)\boxtimes (A\otimes B)\simeq (T\boxtimes A)\otimes (U\boxtimes B)\to +^*(T\underset{[\C,+]}{\otimes}A)\otimes +^*(U\underset{[\C,+]}{\otimes}B)$$
$$\simeq+^*\big((T\underset{[\C,+]}{\otimes}A)\otimes (U\underset{[\C,+]}{\otimes}B)\big)$$
où la flèche centrale est le produit tensoriel (usuel) des unités $T\boxtimes A\to +^*(T\underset{[\C,+]}{\otimes}A)$ et $U\boxtimes B\to +^*(U\underset{[\C,+]}{\otimes}B)$.

Ce morphisme satisfait à des conditions d'associativité et de commutativité qu'on peut traduire en disant qu'il fournit une structure monoïdale symétrique sur le foncteur $\otimes : (\C,\kk)\times (\C,\kk)\to (\C,\kk)$, où $(\C,\kk)$ est munie de la structure monoïdale $\underset{[\C,+]}{\otimes}$ (et $(\C,\kk)\times (\C,\kk)$ de la structure monoïdale produit) ; on vérifie également qu'il s'insère dans un diagramme commutatif
$$\xymatrix{(T\underset{[\C,+]}{\otimes}U)\underset{[\C,+]}{\otimes}(A\underset{[\C,+]}{\otimes}B)\ar[r]^\simeq\ar[d] & (T\underset{[\C,+]}{\otimes}A)\underset{[\C,+]}{\otimes}(U\underset{[\C,+]}{\otimes}B)\ar[d] \\
(T\otimes U)\underset{[\C,+]}{\otimes}(A\otimes B)\ar[r]\ar[d] & (T\underset{[\C,+]}{\otimes}A)\otimes (U\underset{[\C,+]}{\otimes}B)\ar[d] \\(T\otimes U)\otimes (A\otimes B)\ar[r]^\simeq & (T\otimes A)\otimes (U\otimes B)
}$$
dont les flèches verticales sont induites par la transformation naturelle $\underset{[\C,+]}{\otimes}\to\otimes$ introduite en~(\ref{ecpt}).

Bien sûr, tout cela passe au niveau dérivé : dans $\dst(\C,\kk)$, on dispose d'un morphisme naturel
\begin{equation}\label{2pt2}
 (T\overset{\mathbf{L}}{\underset{\kk}{\otimes}}U)\overset{\mathbf{L}}{\underset{[\C,+]}{\otimes}}(A\overset{\mathbf{L}}{\underset{\kk}{\otimes}}B)\to (T\overset{\mathbf{L}}{\underset{[\C,+]}{\otimes}}A)\overset{\mathbf{L}}{\underset{\kk}{\otimes}}(U\overset{\mathbf{L}}{\underset{[\C,+]}{\otimes}}B)
\end{equation}
qui vérifie des propriétés de compatibilité analogues.

\paragraph*{Structure (co)multiplicative sur l'homologie stable des groupes de type congruence}

Les différents produits tensoriels nous permettent de discuter de structures (co)multiplicatives à disposition sur $\Hh_\bullet(\Gamma_\Phi)$, où $\Phi : \C\to\D$ est un morphisme de CMHS stablement plein et faiblement surjectif.

On note déjà que si $\kk$ est un corps (ou, plus généralement, sous des hypothèses de platitude appropriées), le coproduit externe en homologie des groupes fournit dans $(\D,\kk)$ un morphisme gradué
$$\Hh_\bullet(\Gamma_\Phi;X\otimes Y)\to\Hh_\bullet(\Gamma_\Phi;X)\otimes\Hh_\bullet(\Gamma_\Phi;Y)$$
naturel et comonoïdal symétrique en les objets $X$ et $Y$ de $\St(\C,\kk)$.

Par ailleurs, le foncteur $+ : \C\times\C\to\C$ est un morphisme de CMHS. Comme on dispose d'un isomorphisme canonique $\Gamma_{\Phi\times\Phi}\simeq\Gamma_\Phi\times\Gamma_\Phi$ de foncteurs : $\C\times\C\to\mathbf{Grp}$, on obtient, pour deux objets $X$ et $Y$ de $\St(\C,\kk)$, un morphisme naturel gradué de $\St(\D\times\D,\kk)$
$$\Hh_\bullet(\Gamma_\Phi;X)\boxtimes\Hh_\bullet(\Gamma_\Phi;Y)\to\Hh_\bullet(\Gamma_\Phi;X\underset{[\C,+]}{\otimes}Y))\circ +$$
en composant
$$\Hh_\bullet(\Gamma_\Phi;X)\boxtimes\Hh_\bullet(\Gamma_\Phi;Y)\to\Hh_\bullet(\Gamma_{\Phi}\times\Gamma_\Phi;X\boxtimes Y)\simeq\Hh_\bullet(\Gamma_{\Phi\times\Phi};X\boxtimes Y)$$
(induit par le produit externe en homologie des groupes) avec
$$\Hh_\bullet(\Gamma_{\Phi\times\Phi};X\boxtimes Y)\to\Hh_\bullet\big(\Gamma_{\Phi\times\Phi};(X\underset{[\C,+]}{\otimes}Y)\circ +)\big)\to\Hh_\bullet(\Gamma_\Phi;X\underset{[\C,+]}{\otimes}Y))\circ +$$
(morphisme induit par l'unité $X\boxtimes Y\to(X\underset{[\C,+]}{\otimes}Y)\circ +$ suivi du morphisme induit par $+$). Le morphisme gradué de $\St(\D,\kk)$
$$\Hh_\bullet(\Gamma_\Phi;X)\underset{[\D,+]}{\otimes}\Hh_\bullet(\Gamma_\Phi;Y)\to\Hh_\bullet(\Gamma_\Phi;X\underset{[\C,+]}{\otimes}Y)$$
adjoint au précédent définit un morphisme naturel et monoïdal symétrique en les objets $X$ et $Y$ de $\St(\C,\kk)$. En particulier, $\Hh_\bullet(\Gamma_\Phi;\kk)$ est une algèbre graduée commutative pour le produit $\underset{[\C,+]}{\otimes}$.

Lorsque $\kk$ est un corps, les structures multiplicative et comultiplicative sur le foncteur $\Hh_\bullet(\Gamma_\Phi)$ vérifient la propriété de compatibilité suivante : si $X$, $Y$, $A$ et $B$ sont des objets de $\St(\C,\kk)$, le diagramme
$$\xymatrix{\Hh_\bullet(\Gamma_\Phi;X\otimes A)\underset{[\D,+]}{\otimes}\Hh_\bullet(\Gamma_\Phi;Y\otimes B)\ar[r]^-{(c)}\ar[d]_-{(m)} &
\big(\Hh_\bullet(\Gamma_\Phi;X)\otimes\Hh_\bullet(\Gamma_\Phi;A)\big)\underset{[\D,+]}{\otimes}\big(\Hh_\bullet(\Gamma_\Phi;Y)\otimes\Hh_\bullet(\Gamma_\Phi;B)\big)\ar[d]^-{(*)} \\
\Hh_\bullet\big(\Gamma_\Phi;(X\otimes A)\underset{[\C,+]}{\otimes}(Y\otimes B)\big)\ar[d]_-{(*)} & \big(\Hh_\bullet(\Gamma_\Phi;X)\underset{[\D,+]}{\otimes}\Hh_\bullet(\Gamma_\Phi;Y)\big)\otimes\big(\Hh_\bullet(\Gamma_\Phi;A)\underset{[\D,+]}{\otimes}\Hh_\bullet(\Gamma_\Phi;B)\big)\ar[d]^-{(m)}
\\
\Hh_\bullet\big(\Gamma_\Phi;(X\underset{[\C,+]}{\otimes}Y)\otimes (A\underset{[\C,+]}{\otimes}B)\big)\ar[r]^-{(c)} & \Hh_\bullet(\Gamma_\Phi;X\underset{[\C,+]}{\otimes}Y)\otimes\Hh_\bullet(\Gamma_\Phi;A\underset{[\C,+]}{\otimes}B)
}$$
commute, où les flèches désignées par $(m)$ (resp. $(c)$) se déduisent de la structure multiplicative (resp. comultiplicative) tandis que celles désignées par $(*)$ proviennent du morphisme naturel~(\ref{2pt1}).

Si $\Phi$ est $\kk$-excisif pour l'homologie stable, la flèche naturelle $-\underset{[\D,+]}{\otimes}\Hh_\bullet(\Gamma_\Phi;\kk)\to -\otimes\Hh_\bullet(\Gamma_\Phi;\kk)$ est un isomorphisme, de sorte que cela signifie (on suppose encore ici que $\kk$ est un corps) que $\Hh_\bullet(\Gamma_\Phi;\kk)$ est une algèbre de Hopf graduée connexe commutative et cocommutative, phénomène classique en $K$-théorie algébrique.

\subsection{Suite spectrale fondamentale}\label{sssf}

Nos conventions en matière de suites spectrales sont celles qu'on peut trouver par exemple dans \cite[chap.~5]{Weib}.

\begin{lm}\label{prel-thp1}
 Soient $\C$ une CMHS et $G : \C\to\mathbf{Grp}$ un sous-foncteur monoïdal de ${\rm Aut}_\C$ vérifiant les conditions de la proposition~\ref{cmhs-quot}.
 \begin{enumerate}
  \item Il existe dans $(\C/G,\kk)$ un isomorphisme gradué
  $$H_\bullet(G;P^\C_t)\simeq\omega^{\C/G}_{\Pi_G(t)}(H_\bullet(G;\kk))$$
  naturel en $t\in {\rm Ob}\,\C$ ;
  \item Il existe dans  $(\C/G,\kk)$ une suite spectrale
  $$E^2_{p,q}=\mathbf{L}_p\big((-\underset{[\C/G,+]}{\otimes}H_q(G;\kk))\circ (\Pi_G)_*\big)(X)\Rightarrow H_{p+q}(G;X)$$
naturelle en l'objet $X$ de $(\C,\kk)$.
 \end{enumerate}

 De plus, le morphisme de coin $H_\bullet(G;X)\to\mathbf{L}_\bullet((\Pi_G)_*)(X)$ est le morphisme canonique de la proposition~\ref{prh1}, tandis que le morphisme de coin
 $$(\Pi_G)_*(X)\underset{[\C/G,+]}{\otimes}H_\bullet(G;\kk)\to H_\bullet(G;X)$$
 est la composée
$$(\Pi_G)_*(X)\underset{[\C/G,+]}{\otimes}H_\bullet(G;\kk)\simeq H_0(G;X)\underset{[\C/G,+]}{\otimes}H_\bullet(G;\kk)\to H_\bullet(G;X)$$
dont le deuxième morphisme est la multiplication.
 \end{lm}

\begin{proof}
 Soient $t$, $x$ des objets de $\C$ et $f\in\C(t,x)$. Le stabilisateur de $f$ sous l'action de $G(x)$ est $G(x\underset{f}{-}t)$, par le troisime axiome d'une CMHS et la deuxième condition de la proposition~\ref{cmhs-quot}. On en déduit une bijection
 $$\C(t,x)\simeq\underset{\bar{f}\in\C(t,x)/G(x)}{\bigsqcup}G(x)/G(x\underset{f}{-}t)$$
$G(x)$-équivariante naturelle en $x\in {\rm Ob}\,\C$ (où $f$ désigne un relèvement de $\bar{f}$ dans $\C(t,x)$). En linéarisant et en appliquant le lemme de Shapiro, on en déduit la première assertion. La deuxième en découle pour des raisons formelles, en résolvant $X$ par des sommes directes de foncteurs du type $P^\C_t$ (noter que tous les foncteurs qu'on applique commutent aux sommes directes) et en utilisant l'isomorphisme canonique $\omega^{\C/G}_{\Pi_G(t)}\simeq(\Pi_G)_*(P^\C_t)\underset{[\C/G,+]}{\otimes}-$.

Il suffit de vérifier l'assertion relative aux morphismes de coin lorsque $X$ est un foncteur du type $P^\C_t$. Le cas de $H_\bullet(G;X)\to\mathbf{L}_\bullet((\Pi_G)_*)(X)$ est alors clair car seul subsiste le degré nul, où l'on obtient un isomorphisme canonique. Pour l'autre morphisme de coin, il s'agit de voir que le morphisme
$$\omega^{\C/G}_{\Pi_G(t)}(H_\bullet(G;\kk))\xrightarrow{\simeq}P^{\C/G}_{\Pi_G(t)}\underset{[\C/G,+]}{\otimes}H_\bullet(G;\kk)\xrightarrow{\simeq}H_0(G;P^\C_t)\underset{[\C/G,+]}{\otimes}H_\bullet(G;\kk)\to H_\bullet(G;P^\C_t)$$
composé des isomorphismes canoniques et du produit coïncide avec celui qu'on a construit plus haut, ce qu'on vérifie directement à partir des définitions.
\end{proof}

Le théorème suivant constitue le résultat principal de toute la section~\ref{ls1}.

\begin{thm}\label{th-spf}
 Soit $\Phi : \C\to\D$ un morphisme de CMHS faiblement surjectif et stablement plein. Il existe une suite spectrale
 $$E^2_{p,q}(X;\Phi)=\mathbf{L}_p\big((-\underset{[\D,+]}{\otimes}\Hh_q(\Gamma_\Phi;\kk))\circ\Phi_*\big)(X)\Rightarrow\Hh_{p+q}(\Gamma_\Phi;X)$$
naturelle en l'objet $X$ de $\St(\C,\kk)$.

De plus, le morphisme de coin $\Hh_\bullet(\Gamma_\Phi;X)\to\mathbf{L}_\bullet(\Phi_*)(X)$ est le morphisme canonique de la proposition~\ref{prdfh}, tandis que le morphisme de coin
$$\Phi_*(X)\underset{[\D,+]}{\otimes}\Hh_\bullet(\Gamma_\Phi;\kk)\to\Hh_\bullet(\Gamma_\Phi;X)$$
est la composée
$$\Phi_*(X)\underset{[\D,+]}{\otimes}\Hh_\bullet(\Gamma_\Phi;\kk)\simeq\Hh_0(\Gamma_\Phi;X)\underset{[\D,+]}{\otimes}\Hh_\bullet(\Gamma_\Phi;\kk)\to\Hh_\bullet(\Gamma_\Phi;X)$$
dont le deuxième morphisme est la multiplication.

La deuxième page de cette suite spectrale peut se comprendre grâce à la suite spectrale de foncteurs composés
$$E'^2_{i,j}(X,T;\Phi)={\rm Tor}_i^{[\D,+]}(\mathbf{L}_j(\Phi_*)(X),T)\Rightarrow\mathbf{L}_{i+j}\big((-\underset{[\D,+]}{\otimes}T)\circ\Phi_*\big)(X).$$
naturelle en les objets $X$ et $T$ de $\St(\C,\kk)$.
\end{thm}

Les notations $E(X;\Phi)$ et $E'(X,T;\Phi)$ pour ces suites spectrales seront conservées dans toute la suite de cet article ; la mention à $\Phi$ pourra être omise lorsque cela ne prête pas à confusion.

\begin{proof}
 On peut supposer $\D=\C/\Gamma_\Phi$, au vu de la proposition~\ref{prdfh}. Dans ce cas, la première suite spectrale est l'image par le foncteur canonique $(\D,\kk)\to\St(\D,\kk)$ de celle du lemme~\ref{prel-thp1}. La deuxième suite spectrale s'obtient par des arguments formels.
\end{proof}

Les suites spectrales du théorème~\ref{th-spf} possèdent la fonctorialité suivante. Supposons donné un diagramme commutatif de CMHS du type (\ref{dc-cmhs}), avec $\Phi_1$ et $\Phi_2$ stablement pleins et faiblement surjectifs. Commençons par la deuxième suite spectrale : si $T$ est un objet de $\St(\D_2,\kk)$ et $X$ un objet de $\St(\C_2,\kk)$, on dispose d'un morphisme naturel de suites spectrales
$$E'(\Psi^*X,\Xi^*T;\Phi_1)\to\Xi^* E'(X,T;\Phi_2)$$
qui, au niveau des deuxièmes pages, est le morphisme naturel
$${\rm Tor}_i^{[\D_1,+]}(\mathbf{L}_j(\Phi_1)_*(\Psi^* X),\Xi^* T)\to\Xi^* {\rm Tor}_i^{[\D_2,+]}(\mathbf{L}_j(\Phi_2)_*(X),T)$$
obtenu en utilisant le morphisme naturel $\mathbf{L}_j(\Phi_1)_*(\Psi^* X)\to\Xi^*\mathbf{L}_j(\Phi_2)_*(X)$ déduit de la transformation naturelle $\mathbf{L}(\Phi_1)_*\circ\Psi^*\to\Xi^*\circ\mathbf{L}(\Phi_2)_*$ de foncteurs $\dst(\D_2,\kk)\to\dst(\D_1,\kk)$ adjointe à la transformation naturelle
$$\Psi^*\to\Psi^*\Phi_2^*\mathbf{L}(\Phi_2)_*\simeq\Phi_1^*\Xi^*\mathbf{L}(\Phi_2)_*$$
induite par l'unité ${\rm Id}\to\Phi_2^*\mathbf{L}(\Phi_2)_*$, qui fournit
$${\rm Tor}_i^{[\D_1,+]}(\mathbf{L}_j(\Phi_1)_*(\Psi^* X),\Xi^* T)\to {\rm Tor}_i^{[\D_1,+]}(\Xi^*\mathbf{L}_j(\Phi_2)_*(X),\Xi^* T)$$
qu'on compose ensuite avec le morphisme naturel
$${\rm Tor}_i^{[\D_1,+]}(\Xi^*\mathbf{L}_j(\Phi_2)_*(X),\Xi^* T)\to\Xi^* {\rm Tor}_i^{[\D_2,+]}(\mathbf{L}_j(\Phi_2)_*(X),T)$$
obtenu en dérivant la transformation naturelle~(\ref{eqf-pt1}).

On dispose d'une transformation naturelle
\begin{equation}\label{eqtn1}
\mathbf{L}_\bullet\big((-\underset{[\D_1,+]}{\otimes}\Xi^* T)\circ (\Phi_1)_*\big)\circ\Psi^*\to\Xi^*\mathbf{L}_\bullet\big((-\underset{[\D_2,+]}{\otimes}T)\circ (\Phi_2)_*\big)
\end{equation}
définie de façon similaire qui décrit notre morphisme de suites spectrales au niveau des aboutissements.

On dispose également d'un morphisme naturel de suites spectrales
$$E(\Psi^*X;\Phi_1)\to\Xi^* E(X;\Phi_2)$$
qui, entre les aboutissements, est le morphisme $\Hh_\bullet(\Gamma_{\Phi_1};\Psi^* X)\to\Xi^*\Hh_\bullet(\Gamma_{\Phi_2};X)$ défini en (\ref{foncf}), et sur les deuxièmes pages est le morphisme
$$\mathbf{L}_p\big((-\underset{[\D_1,+]}{\otimes}\Hh_q(\Gamma_{\Phi_1};\kk))\circ (\Phi_1)_*\big)(\Psi^* X)\to\Xi^*\mathbf{L}_p\big((-\underset{[\D_2,+]}{\otimes}\Hh_q(\Gamma_{\Phi_2};\kk))\circ (\Phi_2)_*\big)(X)$$
obtenu en composant la flèche
$$\mathbf{L}_p\big((-\underset{[\D_1,+]}{\otimes}\Hh_q(\Gamma_{\Phi_1};\kk))\circ (\Phi_1)_*\big)(\Psi^* X)\to\mathbf{L}_p\big((-\underset{[\D_1,+]}{\otimes}\Xi^*\Hh_q(\Gamma_{\Phi_2};\kk))\circ (\Phi_1)_*\big)(\Psi^* X)$$
induit par $\Hh_\bullet(\Gamma_{\Phi_1};\kk)\to\Xi^*\Hh_\bullet(\Gamma_{\Phi_2};\kk)$ et  
$$\mathbf{L}_p\big((-\underset{[\D_1,+]}{\otimes}\Xi^*\Hh_q(\Gamma_{\Phi_2};\kk))\circ (\Phi_1)_*\big)(\Psi^* X)\to\Xi^*\mathbf{L}_p\big((-\underset{[\D_2,+]}{\otimes}\Hh_q(\Gamma_{\Phi_2};\kk))\circ (\Phi_2)_*\big)(X)$$
donnée par~(\ref{eqtn1}).

Le lecteur courageux pourra écrire les propriétés de compatibilité à la structure multiplicative (et à la structure comultiplicative, lorsque $\kk$ est un corps) sur $\Hh_\bullet(\Gamma_\Phi)$ que vérifient les suites spectrales du théorème~\ref{th-spf}.

Un cas particulièrement important est celui de la suite spectrale $E(\Phi^* F;\Phi)$, où $F$ est un objet de $\St(\D,\kk)$. Comme $\Gamma_\Phi$ opère trivialement sur $\Phi^*F$, l'aboutissement $\Hh_\bullet(\Gamma_\Phi;\Phi^*F)$ de cette suite spectrale est également celui d'une suite spectrale de coefficients universels que nous noterons
$$CU^2_{p,q}(F;\Phi)={\rm Tor}^\kk_p(F,\Hh_q(\Gamma_\Phi;\kk))\Rightarrow\Hh_{p+q}(\Gamma_\Phi;\Phi^*F)$$
(on pourra noter simplement $CU(F)$ cette suite spectrale si $\Phi$ est clair dans le contexte).

On dispose d'un morphisme de suites spectrales $E(\Phi^* F;\Phi)\to CU(F;\Phi)$ naturel en $F$, qui est l'identité sur les aboutissements, et qu'on peut se décrire sur les deuxièmes pages comme le morphisme
$$E^2_{p,q}(\Phi^*F;\Phi)=\mathbf{L}_p\big((-\underset{[\D,+]}{\otimes}\Hh_q(\Gamma_\Phi;\kk))\circ\Phi_*\big)(\Phi^*F)\to {\rm Tor}^{[\D,+]}_p(F,\Hh_q(\Gamma_\Phi;\kk))\to\cdots$$
$${\rm Tor}^\kk_p(F,\Hh_q(\Gamma_\Phi;\kk))=CU^2_{p,q}(F;\Phi)$$
composé du morphisme de coin
$$\mathbf{L}_p\big((-\underset{[\D,+]}{\otimes}\Hh_q(\Gamma_\Phi;\kk))\circ\Phi_*\big)(\Phi^*F)\to E'^2_{p,0}(\Phi^*F,\Hh_q(\Gamma_\Phi;\kk);\Phi)$$
(modulo l'identification de $\Phi_*\Phi^*(F)$ à $F$ via la coünité de l'adjonction) et de l'évaluation de la transformation naturelle 
${\rm Tor}^{[\D,+]}_\bullet\to {\rm Tor}^\kk_\bullet$ obtenue en dérivant~(\ref{ecpt}).

Ce morphisme naturel de suites spectrales s'obtient comme suit : on peut tout d'abord se ramener au cas où $\Phi$ est le foncteur de projection $\C\to\C/\Gamma_\Phi$ (cf. démonstration du théorème~\ref{th-spf}) et raisonner donc avec des suites spectrales vivant dans $(\D,\kk)$ plutôt que $\St(\D,\kk)$. Ceci posé, notons $\R^\C(X)$ une résolution projective d'un foncteur $X$ de $(\C,\kk)$ (qu'on peut choisir fonctorielle en $X$) : évaluée sur un objet $x$ de $\C/\Gamma_\Phi$, c'est-à-dire un objet de $\C$, la suite spectrale $E(X;\Phi)$ est la suite spectrale d'hyperhomologie associée à $\kk\overset{\mathbf{L}}{\underset{\kk[\Gamma_\Phi(x)]}{\otimes}}\R^\C(X)(x)$. Si $F$ est un foncteur de $(\D,\kk)$, la suite spectrale $CU(F;\Phi)$ (qu'on relève aussi dans $(\D,\kk)$), évaluée sur un objet $x$ de $\C$ est la suite spectrale d'hyperhomologie associée à $\kk\overset{\mathbf{L}}{\underset{\kk[\Gamma_\Phi(x)]}{\otimes}}\Phi^*\R^\D(F)(x)$ car $\R^\D(F)(\Phi x)$ est un complexe de $\kk[\Gamma_\Phi(x)]$-modules triviaux et $\kk$-projectifs. Le quasi-isomorphisme naturel de complexes $\R^\C(\Phi^* F)\to\Phi^*\R^\D(F)$ (qui est unique à homotopie près) induit entre les suites spectrales d'hyperhomologie correspondantes notre morphisme $E(\Phi^* F;\Phi)\to CU(F;\Phi)$.

L'énoncé qui suit caractérise l'excisivité en homologie stable en termes de foncteurs dérivés ; on pourrait utiliser la comparaison entre les suites spectrales $E(\Phi^* F;\Phi)$ et $CU(F;\Phi)$ qu'on vient d'expliciter pour l'établir, mais il est plus simple de se ramener par un argument forme à une situation de valeurs $\kk$-projectives.

\begin{pr}\label{cnex}
Soit $\Phi : \C\to\D$ un morphisme de CMHS faiblement surjectif et stablement plein.

 Le foncteur $\Phi$ est $\kk$-excisif pour l'homologie stable jusqu'en degré $d$ si et seulement si les endofoncteurs $\mathbf{L}_i(\Phi_*)\Phi^*$ de $\St(\D,\kk)$ sont nuls pour $0<i\leq d$. De plus, on dispose alors d'un isomorphisme
 $$F\underset{[\D,+]}{\overline{\otimes}}\Hh_{d+1}(\Gamma_\Phi;\kk)\simeq\mathbf{L}_{d+1}(\Phi_*)(\Phi^* F)$$
 naturel en l'objet $F$ de $\St(\D,\kk)$.
\end{pr}

\begin{proof}
Raisonnant par récurrence sur $d$, on peut supposer que l'assertion est vraie jusqu'en degré $d-1$. L'isomorphisme $F\underset{[\D,+]}{\overline{\otimes}}\Hh_d(\Gamma_\Phi)\simeq\mathbf{L}_d(\Phi_*)\Phi^* F$ qu'elle procure montre, en utilisant l'assertion~\ref{p3ptt} de la proposition~\ref{pg-ptt}, que $\Hh_d(\Gamma_\Phi;\kk)$ est (stablement) constant si et seulement si $\mathbf{L}_d(\Phi_*)\Phi^*=0$, de sorte qu'il suffit de vérifier l'isomorphisme naturel $F\underset{[\D,+]}{\overline{\otimes}}\Hh_{d+1}(\Gamma_\Phi;\kk)\simeq\mathbf{L}_{d+1}(\Phi_*)(\Phi^* F)$. 

La nullité de $\mathbf{L}_d(\Phi_*)\Phi^*$ (ou, pour $d=0$, le fait que la coünité $\Phi_*\Phi^*\to {\rm Id}$ est un isomorphisme) montre que $\mathbf{L}_{d+1}(\Phi_*)\Phi^*$ est un foncteur exact à droite, de sorte qu'il suffit d'exhiber un tel isomorphisme naturel pour $F$ {\em projectif}. En particulier, $F$ est à valeurs $\kk$-plates, ce qui garantit, puisque $\Gamma_\Phi$ opère trivialement sur $\Phi^*(F)$, un isomorphisme naturel $\Hh_\bullet(\Gamma_\Phi;\Phi^* F)\simeq F\otimes\Hh_\bullet(\Gamma_\Phi;\kk)$.
 
  Examinons les suites spectrales $E(\Phi^* F)$ et $E'(\Phi^* F,\Hh_q(\Gamma_\Phi;\kk))$ du théorème~\ref{th-spf} : on a $E'^2_{i,j}=0$ pour $0<j\leq d$, $E'^2_{0,0}\simeq F\underset{[\D,+]}{\otimes}\Hh_q(\Gamma_\Phi;\kk)$, $E'^2_{i,0}=0$ si $i>0$ puisque $F$ est projectif, ce qui implique que, pour $p+q=d+1$, $E^2_{p,q}$ est nul sauf pour $p=0$, auquel cas il s'identifie à $F\otimes\Hh_{d+1}(\Gamma_\Phi;\kk)$, ou pour $q=0$, auquel cas il s'identifie à $\mathbf{L}_{d+1}\Phi_*(\Phi^* F)$. De plus, l'hypothèse de récurrence montre que $E^2_{p,q}=0$ dès lors qu'on a à la fois $0<p\leq d$ et $q\leq d$.
  
  On en déduit que les morphismes de coin fournissent une suite exacte
  $$E^2_{0,d+1}\simeq F\underset{[\D,+]}{\otimes}\Hh_{d+1}(\Gamma_\Phi;\kk)\to F\otimes\Hh_{d+1}(\Gamma_\Phi;\kk)\to E^2_{d+1,0}\simeq\mathbf{L}_{d+1}\Phi_*(\Phi^* F)\to\cdots$$
  $$E^2_{0,d}\simeq F\underset{[\D,+]}{\otimes}\Hh_d(\Gamma_\Phi;\kk)\to F\otimes\Hh_d(\Gamma_\Phi;\kk)$$
 mais cette dernière flèche est un isomorphisme par l'hypothèse de récurrence, de sorte que cette suite exacte se réduit à la suivante
 $$E^2_{0,d+1}\simeq F\underset{[\D,+]}{\otimes}\Hh_{d+1}(\Gamma_\Phi;\kk)\to F\otimes\Hh_{d+1}(\Gamma_\Phi;\kk)\to E^2_{d+1,0}\simeq\mathbf{L}_{d+1}\Phi_*(\Phi^* F)\to 0,$$
 qui fournit exactement l'isomorphisme naturel souhaité et achève la démonstration.
\end{proof}

On obtient en particulier le résultat suivant, qui généralise \cite[§\,2.3]{DV}.

\begin{cor}\label{crclexc}
 Supposons que $\Phi$ est $\kk$-excisif pour l'homologie stable et que $F$ est un foncteur de $(\C,\kk)$.
 
 Si $\kk$ est un corps, on dispose d'un isomorphisme naturel gradué
$$\Hh_\bullet(\Gamma_\Phi;F)\simeq\Hh_\bullet(\Gamma_\Phi;\kk)\underset{\kk}{\otimes}\mathbf{L}_\bullet\Phi_*(F).$$

Si $\kk$ est de dimension homologique au plus $1$, on a une suite exacte naturelle (scindée, de façon a priori non naturelle) de coefficients universels
$$0\to\bigoplus_{i+j=n}\Hh_i(\Gamma_\Phi;\kk)\underset{\kk}{\otimes}\mathbf{L}_j\Phi_*(F)\to\Hh_n(\Gamma_\Phi;F)\to\bigoplus_{i+j=n-1}{\rm Tor}^\kk_1(\Hh_i(\Gamma_\Phi;\kk),\mathbf{L}_j\Phi_*(F))\to 0$$
pour tout $n\in\mathbb{N}$ ; il suffit que $\Phi$ soit $\kk$-excisif pour l'homologie stable jusqu'en degré $n$ pour que cela soit valable.
\end{cor}

\begin{proof}
 Grâce aux suites spectrales $E'(F,\Hh_q(\Gamma_\Phi;\kk);\Phi)$, à  la proposition~\ref{cnex} et à la formule de Künneth, il suffit de montrer que, si $\kk$ est de dimension homologique au plus $1$, la suite spectrale $E(F;\Phi)$ s'arrête à la deuxième page et que le gradué associé est trivial, si $\Phi$ est $\kk$-excisif pour l'homologie stable (avec une adaptation tronquée évidente en degrés $\leq n$ pour la dernière assertion). Cela provient de ce que le complexe de $\dst(\C,\kk)$ induit par $x\mapsto\kk\overset{\mathbf{L}}{\underset{\kk[\Gamma_\Phi(x)]}{\otimes}}\kk$ est {\em formel} (jusqu'en degré $n$ dans le cas tronqué). En effet, son homologie est par hypothèse (stablement) constante (jusqu'en degré $n$ dans le cas tronqué), de sorte qu'il suffit d'appliquer la théorie de l'obstruction pour les complexes de chaînes de Dold \cite{Dold} : en effet, si $A$ et $B$ sont deux $\kk$-modules, on a ${\rm Ext}^r_{\St(\C,\kk)}(A,B)\simeq {\rm Ext}^r_{(\C,\kk)}(A,B)\simeq {\rm Ext}^r_\kk(A,B)$ (par abus, on note encore $A$ et $B$ pour les foncteurs constants sur $\C$ associés, ainsi que leurs images dans $\St(\C,\kk)$), où le premier isomorphisme découle du lemme~2.17 de \cite{DV-pol} et le second de ce que $\C$ a un objet initial, or ce $\kk$-module d'extensions est nul pour $r>1$ par hypothèse sur $\kk$. Cette propriété de formalité pour $\kk\overset{\mathbf{L}}{\underset{\kk[\Gamma_\Phi]}{\otimes}}\kk$ s'étend aussitôt à $\kk\overset{\mathbf{L}}{\underset{\kk[\Gamma_\Phi]}{\otimes}}P$ où $P$ est projectif, ce dont on déduit le résultat souhaité en considérant une résolution projective de $F$. 
\end{proof}

\section{Étude dans le cas des groupes de congruence}\label{sf}

On commence par rappeler dans cette section le cadre catégorique général qui permet de définir des groupes de congruence de type linéaire, hermitien, symplectique, et comment lui associer naturellement des catégories de modèles pour l'homologie stable. On généralise ensuite un résultat important de \cite{Dja-JKT} (comparaison homologique à coefficients polyonomiaux de catégories d'objets hermitiens et d'objets hermitiens dégénérés), dans un contexte de catégories dérivées. Ce résultat est utilisé de façon essentielle pour étudier, sur des foncteurs polynomiaux, d'abord les extensions de Kan à gauche dérivées associées aux foncteurs induits entre catégories d'objets hermitiens par des foncteurs additifs, puis le produit tensoriel le long de la somme hermitienne et ses dérivés (cela nécessitera l'introduction d'un autre produit tensoriel auxiliaire). Cela permet, à l'aide des suites spectrales du théorème~\ref{th-spf}, de démontrer nos résultats principaux sur l'homologie des groupes de congruence. 

\subsection{Catégories additives à dualité et objets hermitiens}\label{sherm}

\begin{nota}
 \begin{enumerate}
  \item On note $\mathbf{Add}$ la catégorie des petites catégories additives dont les idempotents se scindent, les morphismes étant les foncteurs additifs.
  \item On note $\mathbf{Add}^{ps}$ la sous-catégorie de $\mathbf{Add}$ ayant les mêmes objets et dont les morphismes sont les foncteurs pleins et faiblement surjectifs (i.e. essentiellement surjectifs à facteur direct près).
  \item Si $\A$ est une catégorie additive, on note $\mathbf{S}(\A)$ ayant les mêmes objets et dont les morphismes $A\to B$ sont les couples $(u,v)\in\A(A,B)\times\A(B,A)$ tels que $vu={\rm Id}_\A$. Si $\A$ appartient à $\mathbf{Add}$, on vérifie aussitôt que $\mathbf{S}(\A)$ est une CMHS (pour la structure monoïdale symétrique induite par la somme directe) ; on définit ainsi un foncteur $\mathbf{S} : \mathbf{Add}\to\mathbf{CMHS}$ (l'effet sur les morphismes étant évident).
 \end{enumerate}
\end{nota}

\begin{rem}
\begin{enumerate}
 \item La structure de foncteur d'automorphismes sur $\mathbf{S}(\A)$ est donnée comme suit : à un morphisme $(u,v) : A\to B$ de $\mathbf{S}(\A)$, on associe le morphisme de groupes
 $${\rm Aut}_\A(A)(\simeq {\rm Aut}_{\mathbf{S}(\A)}(A))\to {\rm Aut}_\A(B)\quad f\mapsto ufv+1_B-uv.$$
 \item L'hypothèse de scindement des idempotents, simplement destinée à garantir que $\mathbf{S}(\A)$ est une CMHS, n'est guère restrictive (on peut remplacer la catégorie additive $\A$ par son enveloppe karoubienne). Elle pourrait être supprimée en généralisant un petit peu le cadre formel des CMHS.
\end{enumerate}
\end{rem}

Dans toute la suite, on se fixe $\varepsilon\in\{1,-1\}$ ; les structures hermitiennes que l'on considère (on renvoie par exemple au début de la section~4 de \cite{Dja-JKT} pour plus de précisions à ce sujet) pourraient être appelées $\varepsilon$-hermitiennes, le cas $\varepsilon=-1$ correspondant au cas symplectique et $\varepsilon=1$ au cas hermitien usuel.

\begin{nota}\label{nh}
 \begin{enumerate}
  \item On désigne par $\mathbf{Add}_D$ la catégorie des objets $\A$ de $\mathbf{Add}$ munis d'un foncteur de dualité $D_\A : \A\to\A^{op}$ (qui sera souvent noté simplement $D$) ; les morphismes sont les foncteurs additifs commutant à la dualité.
  \item On note $\mathbf{Add}_D^{ps}$ la sous-catégorie de $\mathbf{Add}_D$ ayant les mêmes objets et dont les morphismes sont les foncteurs pleins et faiblement surjectifs.
  \item Si $\A$ est un objet de $\mathbf{Add}_D$, on note $T_\A$ (ou simplement $T$) le foncteur quadratique $\A^{op}\to\mathbf{Ab}$ (on désignera également de la même façon sa composée avec le foncteur d'oubli $\mathbf{Ab}\to\mathbf{Ens}$) donné par $T_\A(x)=\A(x,Dx)_{\mathfrak{S}_2}$, où l'action de $\mathfrak{S}_2$ est donnée par l'auto-dualité du foncteur contravariant $D$, tordue par $\varepsilon$. Ainsi, $T(x)$ n'est autre que l'ensemble des structures ($\varepsilon$-)hermitiennes {\em éventuellement dégénérées} sur l'objet $x$ de $\A$. On notera $\mathbf{HD}(\A)$ la catégorie d'éléments $\A^T$ (c'est la catégorie des objets de $\A$ munis d'une structure hermitienne éventuellement dégénérée) ; le foncteur d'oubli $\pi^T : \mathbf{HD}(\A)\to\A$ sera noté $p_\A$ (ou simplement $p$). La catégorie $\mathbf{HD}(\A)$ possède une structure monoïdale symétrique naturelle donnée par la somme hermitienne, notée $\overset{\perp}{\oplus}$.
  \item\label{nh4} On note $\mathbf{H}(\A)$ la sous-catégorie pleine de $\mathbf{HD}(\A)$ constituée des objets $x$ de $\A$ munis d'une structure hermitienne (non dégénérée), c'est-à-dire tels que l'image par la norme
  $$T_\A(x)\to\A(x,Dx)\qquad [f]\mapsto f+\varepsilon.\bar{f}$$
  (où $[f]$ désigne la classe dans $T_\A(x)$ de $f\in\A(x,Dx)$, et $\bar{f}$ l'effet sur $f$ du foncteur de dualité $D$, modulo l'isomorphisme canonique $DDx\simeq x$) de la structure herminitienne sur $x$ soit un isomorphisme. La catégorie $\mathbf{H}(\A)$ est une sous-catégorie monoïdale de $\mathbf{HD}(\A)$ qui possède une structure canonique de CMHS. Tout morphisme $\Phi : \A\to\B$ de $\mathbf{Add}_D$ induit des foncteurs $\mathbf{HD}(\Phi) : \mathbf{HD}(\A)\to\mathbf{HD}(\B)$ et $\mathbf{HD}(\Phi) : \mathbf{HD}(\A)\to\mathbf{HD}(\B)$ ; en particulier, on dispose d'un foncteur $\mathbf{H} : \mathbf{Add}_D\to\mathbf{CMHS}$.
  
  On note $i_\A : \mathbf{H}(\A)\to\mathbf{HD}(\A)$ le foncteur d'oubli.
 \end{enumerate}
\end{nota}

On dispose d'un isomorphisme naturel d'adjonction
\begin{equation}\label{eq-omegh}
 {\rm Tor}_\bullet^\A(\omega_\A(X),F)\simeq {\rm Tor}_\bullet^{\mathbf{HD}(\A)}(X,p_\A^*(F))
\end{equation}
pour $X$ dans $(\mathbf{HD}(\A)^{op},\kk)$ et $F$ dans $(\A,\kk)$, où le foncteur exact $\omega_\A : (\mathbf{HD}(\A)^{op},\kk)\to (\A^{op},\kk)$ est donné sur les objets par
$$\omega_\A(X)(V)=\bigoplus_{\xi\in T(V)}X(V,\xi)$$
(ceci est un cas particulier des considérations générales d'extensions de Kan dans des catégories d'éléments --- cf. (\ref{eqom})).

\smallskip

Les exemples les plus importants d'objets de $\mathbf{Add}$ et de $\mathbf{Add}_D$ sont donnés comme suit.

\begin{nota}
Soit $A$ un anneau (associatif et unitaire). On note $\mathbf{P}(A)$ un squelette de la catégorie  des $A$-modules à gauche projectifs de type fini : c'est un objet de $\mathbf{Add}$. Si $A$ est muni d'une involution, le foncteur de dualité ${\rm Hom}_A(-,A)$ fait de $\mathbf{P}(A)$ un objet de $\mathbf{Add}_D$. On obtient ainsi un foncteur $\mathbf{P}$ de la catégorie des anneaux (resp. des anneaux à involution) vers $\mathbf{Add}$ (resp. $\mathbf{Add}_D$), l'effet sur les morphismes étant donné par l'extension des scalaires.

On notera $\mathbf{S}(A)$ (resp. $\mathbf{H}(A)$) pour $\mathbf{S}(\mathbf{P}(A))$ (resp. $\mathbf{H}(\mathbf{P}(A))$), où $A$ est un anneau (resp. un anneau à involution).
\end{nota}

\paragraph*{Liens entre les foncteurs $\mathbf{S}$ et $\mathbf{H}$}

On note d'abord que, à côté du foncteur d'oubli $\mathbf{Add}_D\to\mathbf{Add}$, on dispose d'un foncteur $\mathbf{Add}\to\mathbf{Add}_D$ associant la catégorie additive $\A^{\rm e}:=\A^{op}\times\A$ munie de la dualité échangeant les deux facteurs à la catégorie additive $\A$. La catégorie $\mathbf{H}(\A^{\rm e})$ est canoniquement équivalente à $\mathbf{S}(\A)$, équivalence par l'intermédiaire de laquelle la catégorie $\mathbf{HD}(\A^{\rm e})$ s'identifie à la catégorie des factorisations  $\mathbf{F}(\A)$ (cf. Quillen \cite{QK}).

\begin{nota}\label{nhs} Soit $\A$ un objet de $\mathbf{Add}_D$.
 \begin{enumerate}
  \item On désigne par $o_\A : \mathbf{H}(\A)\to\mathbf{S}(\A)$ le foncteur qui est l'oubli sur les objets et associe à une flèche $f : A\to B$ de $\mathbf{H}(\A)$ la flèche $(f,\xi_A^{-1}\bar{f}\xi_B)$, où $\bar{f}:= D(f) : DB\to DA$ et où $\xi_A : A\to DA$ désigne l'isomorphisme de $\A$ donné par la structure hermitienne (cf. notation~\ref{nh}.\ref{nh4}), et de même pour $B$.
  \item On note $h_\A : \mathbf{S}(\A)\simeq\mathbf{H}(\A^{\rm e})\to\mathbf{H}(\A)$ le foncteur induit par le morphisme $\A^{\rm e}=\A^{op}\times\A\to\A$ de $\mathbf{Add}_D$ dont la composante $\A^{op}\to\A$ (resp. $\A\to\A$) est le foncteur de dualité (resp. l'identité).
  \item On désigne par ${\rm H}_\A$ l'endofoncteur $h_\A o_\A$ de $\mathbf{H}(\A)$.
 \end{enumerate}
 
 Dans ces notations, les références à $\A$ en indice pourront être omises quand cela n'engendre pas de confusion.
\end{nota}

\begin{rem}
 Explicitement, $h_\A(V)$ est l'objet $V\oplus DV$ muni de la structure hermitienne image dans $T_\A(V\oplus DV)$ du morphisme $V\oplus DV\to D(V\oplus DV)\simeq DV\oplus V$ donné par la projection canonique sur $V$. Les objets de l'image essentielle de $h_\A$ sont par définition les objets hermitiens {\em hyperboliques} de $\mathbf{H}(\A)$. Comme on dispose toujours d'un morphisme $E\to {\rm H}_\A(E)$, pour $E$ dans $\mathbf{H}(\A)$, le foncteur $h_\A$ est faiblement surjectif.
\end{rem}

\paragraph*{Groupes de congruence associés}

Si $\Phi : \A\to\B$ est un morphisme de $\mathbf{Add}_D$, on dispose d'un foncteur $\Gamma_{\mathbf{H}(\Phi)} : \mathbf{H}(\A)\to\mathbf{Grp}$ ; on obtient donc, sous les hypothèses du lemme classique suivant (laissé au lecteur), un foncteur homologique $\Hh_\bullet(\Gamma_{\mathbf{H}(\Phi)};-) : \St(\mathbf{H}(\A),\kk)\to\St(\mathbf{H}(\B),\kk)$, dont l'étude constitue l'objet de la suite de cet article.

\begin{lm}\label{lm-cadreh}
 Soit $\Phi : \A\to\B$ un morphisme de $\mathbf{Add}_D^{ps}$. Le morphisme de CMHS $\mathbf{H}(\Phi) : \mathbf{H}(\A)\to\mathbf{H}(\B)$ est faiblement surjectif et stablement plein.
\end{lm}

Les groupes de congruence hermitiens usuels s'obtiennent lorsque $\A$ (resp. $\B)$ est de la forme $\mathbf{P}(A)$
(resp. $\mathbf{P}(B)$), $\Phi$ étant induit par un morphisme d'anneaux à involution $A\to B$. Les hypothèses du lemme sont vérifiées si ce morphisme est surjectif.

En appliquant ce qui précède à des catégories à dualité du type $\A^{{\rm e}}$, on voit que si $\Phi : \A\to\B$ est un morphisme de $\mathbf{Add}$, on obtient un foncteur $\Gamma_{\mathbf{S}(\Phi)} : \mathbf{S}(\A)\to\mathbf{Grp}$, et un foncteur homologique $\Hh_\bullet(\Gamma_{\mathbf{S}(\Phi)};-) : \St(\mathbf{S}(\A),\kk)\to\St(\mathbf{S}(\B),\kk)$ lorsque $\Phi$ appartient à $\mathbf{Add}^{ps}$. Cela s'applique à tout foncteur $\mathbf{P}(A)\to\mathbf{P}(B)$ induit par un morphisme surjectif d'anneaux $A\to B$ ; les $\Gamma_{\mathbf{S}(\Phi)}$ correspondent alors aux groupes de congruence usuels (dans les groupes linéaires).

\medskip

Pour certaines de nos considérations ultérieures, l'hypothèse supplémentaire suivante (qu'on ne fait pas de façon systématique) sur la catégorie additive à dualité $\A$ sera utile.

\begin{hyp}\label{h2inv}
 Le foncteur quadratique réduit $T : \A^{op}\to\mathbf{Ab}$ est {\em pseudo-diagonalisable} au sens de la définition~\ref{df-fqpd}.
\end{hyp}

Cette hypothèse est vérifiée dans deux cas particuliers fondamentaux :
\begin{enumerate}
 \item $\A$ appartient à l'image essentielle du foncteur $\mathbf{Add} \to\mathbf{Add}_D\quad\A\mapsto\A^{{\rm e}}$ ;
 \item $T$ est à valeurs dans les $\mathbb{Z}[1/2]$-modules.
\end{enumerate}

L'hypothèse~\ref{h2inv} sera utilisée par l'intermédiaire de la proposition suivante.

\begin{pr}\label{pr-h2inv}
 Soit $A : \A\to\kk\md$ un foncteur additif. On a
 $${\rm Tor}^{\kk[\A]}_i(\kk[T],A)=0\quad\text{et}\quad{\rm Ext}^i_{(\A,\kk)}(A,\kk^T)=0\quad\text{pour }i\leq 1.$$
 Ces annulations s'étendent à tous les degrés (co)homologiques $i$ si l'hypothèse~\ref{h2inv} est satisfaite ou que $2$ est inversible dans $\kk$.
\end{pr}

\begin{proof}
 Le lemme d'annulation de Pirashvili \cite{P-add} implique qu'on a
 $${\rm Tor}^{\kk[\A]}_\bullet(\kk[S],A)=0\quad\text{et}\quad{\rm Ext}^\bullet_{(\A,\kk)}(A,\kk^S)=0$$ si $S$ est un foncteur diagonalisable, ce qui montre d'emblée la deuxième assertion (si $2$ est inversible dans $\kk$, $\kk[T]$ est facteur direct de $\kk[S]$ pour le foncteur $S$ ci-après). Pour la première, on note que $T$ est le conoyau d'un endomorphisme (donné par $f\mapsto\bar{f}-\epsilon f$) du foncteur diagonalisable $S : x\mapsto\A(x,Dx)$. La conclusion découle donc du même lemme de Pirashvili ainsi que de la première assertion du lemme~\ref{lm-troesch} ci-dessous.
\end{proof}

La propriété suivante (qui nous sera de nouveau utile un peu plus tard) constitue une variante du résultat de Troesch~\cite[Théorème 4]{Troe}.

\begin{lm}\label{lm-troesch}
 Soient $\A$ une petite catégorie additive et $A : \A^{op}\to\mathbf{Ab}$ et $B : \A\to\kk\md$ des foncteurs additifs.
 \begin{enumerate}
  \item Le foncteur gradué
  $$X\mapsto {\rm Tor}_\bullet^{\kk[\A]}(\kk[X],B)$$
  défini sur $(\A^{op},\mathbf{Ab})$ est homologique.
  \item Il existe un isomorphisme naturel gradué
 $${\rm Tor}_\bullet^{\kk[\A]}(\kk[A],B)\simeq {\rm Tor}_\bullet^\A(A,B)$$
 (où l'on note par abus, dans le terme de droite, $B$ pour l'image de ce foncteur dans $\A\md$, c'est-à-dire sa post-composition par le foncteur d'oubli $\kk\md\to\mathbf{Ab}$).
 \end{enumerate}
\end{lm}

(Le membre de gauche de l'isomorphisme se calcule dans la catégorie de tous les foncteurs vers $\kk\md$, tandis que celui de droite se calcule dans la catégorie des foncteurs additifs vers $\mathbf{Ab}$.)

\begin{proof}
Si $0\to X\to Y\to Z\to 0$ est une suite exacte de foncteurs $\A^{op}\to\mathbf{Ab}$, une forme de la résolution barre fournit une suite exacte de foncteurs
$$\cdots\to\bar{\kk}[X]^{\otimes (n+1)}\otimes\kk[Y]\to\bar{\kk}[X]^{\otimes n}\otimes\kk[Y]\to\cdots\to\bar{\kk}[X]\otimes\kk[Y]\to\kk[Y]\to\kk[Z]\to 0$$
 (où $\bar{\kk}[-]$ désigne la partie réduite du foncteur $\kk[-]$) dont le lemme d'annulation de Pirashvili permet de tirer le premier point. Le deuxième s'en déduit puisqu'il se réduit au lemme de Yoneda lorsque $A$ est projectif (dans $\mdd\A$).
\end{proof}

Donnons maintenant quelques précisions sur les processus de stabilisation (au sens du §\,\ref{pps}) en jeu. Comme toute catégorie de modèles pour l'homologie stable, $\mathbf{H}(\A)$ (pour $\A$ dans $\mathbf{Add}_D$) est munie de l'auto-action donnée par sa structure monoïdale, la somme hermitienne. Quant à $\mathbf{HD}(\A)$, on la munit de l'action de sa sous-catégorie monoïdale $\mathbf{H}(\A)$ par somme hermitienne. Cela permet d'utiliser la proposition-définition~\ref{fonct-st} pour obtenir des adjonctions
$$i_* : \St(\mathbf{H}(\A),\M)\rightleftarrows\St(\mathbf{HD}(\A),\M) : i^*$$
et
\begin{equation}\label{adj-hhd}
 \mathbf{L}i_* : \dst(\mathbf{H}(\A),\M)\rightleftarrows\dst(\mathbf{HD}(\A),\M) : i^*.
\end{equation}

\begin{rem}
 On pourrait aussi munir $\mathbf{HD}(\A)$ de son auto-action canonique (i.e. se placer exactement dans le cadre de \cite{DV-pol}). Il n'est pas difficile de voir que cette autre action donne lieu à la même catégorie $\St$ (car tout objet hermitien dégénéré se plonge dans un objet hermitien non dégénéré) ; on peut voir {\em a posteriori} que les notions d'objet polynomial sont stablement les mêmes (ce n'est revanche pas vrai pour les foncteurs {\em fortement} polynomiaux), mais l'auteur n'en connaît pas de démonstration simple (on peut voir ce résultat comme une conséquence des considérations du paragraphe suivant). Quoi qu'il en soit, la notion d'objet polynomial associée à l'action de $\mathbf{H}(\A)$ est la plus adaptée dans notre contexte.
\end{rem}

\subsection{Comparaison de $\dst(\mathbf{H}(\A),\kk)$ et  $\dst(\mathbf{HD}(\A),\kk)$ sur des objets polynomiaux}\label{shd}

Dans tout ce paragraphe, $\A$ désigne un objet de $\mathbf{Add}_D$ et $\M$ une catégorie de Grothendieck $\kk$-linéaire.

Avant de démontrer nos résultats principaux, nous avons besoin de quelques préliminaires sur les foncteurs polynomiaux.

\begin{nota}
 Soit $E$ un ensemble fini. On note ${\rm T}_E$ l'endofoncteur ${\rm H}^{\overset{\perp}{\oplus}E}\overset{\perp}{\oplus}{\rm Id}$ de $\mathbf{H}(\A)$. (On rappelle que ${\rm H}$ est défini dans la notation~\ref{nhs}.)
 
 Si $I$ est une partie de $E$, on dispose d'une transformation naturelle $\alpha_I : {\rm Id}\to {\rm T}_E$ qui sur l'objet $V$ est le morphisme $V\to {\rm H}(V)^{\overset{\perp}{\oplus}E}\overset{\perp}{\oplus}V$ dont le morphisme de $\A$ sous-jacent a pour composantes l'identité $V\to V$ et soit l'inclusion canonique $V\hookrightarrow V\oplus DV={\rm H}(V)$, soit $0$,  $V\to{\rm H}(V)$ selon que le facteur correspondant est étiqueté par un élément de $E$ appartenant ou non à $I$.
 
 On définit l'{\em effet croisé hermitien} associé à $E$ comme la transformation naturelle
$$cr_E:=\underset{I\in\Pp(E)}{\sum}(-1)^{|I|}\alpha_I^* : {\rm Id}\to {\rm T}_E^*$$ d'endofoncteurs de $(\mathbf{H}(\A),\M)$, où $|I|$ désigne le cardinal de $I$.

On dispose d'une version duale de cette transformation naturelle, encore notée $cr_E : {\rm T}_E^*\to {\rm Id}$, entre endofoncteurs de $(\mathbf{H}(\A)^{op},\M)$.
\end{nota}

La notion d'effet croisé hermitien, et la notation $cr_E$ en ce sens, sera utilisée seulement dans ce §\,\ref{shd} ; ultérieurement interviendront les effets croisés usuels.

Si $F$ (resp. $X$) est un foncteur contravariant (resp. covariant) sur $\mathbf{H}(\A)$, on dispose d'un diagramme commutatif
\begin{equation}\label{diag-sco}
 \xymatrix{{\rm Tor}^{\kk[\mathbf{H}(\A)]}_\bullet({\rm T}_E^*F,X)\ar[rr]^{cr_E(X)_*}\ar[rrd]_{cr_E(F)_*} &  & {\rm Tor}^{\kk[\mathbf{H}(\A)]}_\bullet({\rm T}_E^*F,{\rm T}_E^*X)\ar[d]\\
& & {\rm Tor}^{\kk[\mathbf{H}(\A)]}_\bullet(F,X)
}
\end{equation}
dont la flèche verticale est induite par l'endofoncteur ${\rm T}_E$ de $\mathbf{H}(\A)$.

La remarque suivante est immédiate et cruciale : ${\rm H}$ ne s'étend pas en un endofoncteur de $\mathbf{HD}(\A)$, mais ${\rm H}(V)$ est un objet de $\mathbf{HD}(\A)$ bien défini si $V$ est un objet de $\mathbf{HD}(\A)$, et donc ${\rm T}_E(V)$ également. De plus, la formule ci-avant s'étend pour définir un morphisme $\alpha_I(V) : V\to {\rm H}(V)$ de $\mathbf{HD}(\A)$. On dispose aussi d'une forme faible de fonctorialité de ${\rm H}$ et de ${\rm T}_E$ sur $\mathbf{HD}(\A)$ (outre la fonctorialité évidente relative aux isomorphismes) : tout morphisme $U\to V$ de $\mathbf{HD}(\A)$ dont la source $U$ appartient à $\mathbf{H}(\A)$ induit un morphisme $U\to V$ dans $\mathbf{S}(\A)$, donc des morphismes ${\rm H}(U)\to {\rm H}(V)$ et ${\rm T}_E(U)\to {\rm T}_E(V)$ dans $\mathbf{H}(\A)$, et ce d'une façon compatible à la composition et aux applications $\alpha_I$ en un sens évident.

Cela nous permet d'énoncer la proposition suivante (qui s'inspire, de même que les énoncés qu'on en déduit ensuite, de Scorichenko \cite{Sco} ; voir aussi \cite[§\,5.1]{Dja-JKT} pour une utilisation dans un cadre hermitien du critère de Scorichenko), dont la vérification est immédiate.

\begin{pr}\label{sco1}
 Soit $F$ un foncteur de $\pol_d(\A,\M)$. Si $E$ est un ensemble fini de cardinal au moins $d+1$, alors, pour tout objet $V$ de $\mathbf{HD}(\A)$, le morphisme $cr_E(p^*F)(V)$ est nul.
\end{pr}

\begin{nota}\label{nota-ekdaux}
 Soient $X$ un objet de $\mathbf{D}(\mathbf{HD}(\A),\M)$ et $V$ un objet de de $\mathbf{HD}(\A)$. On note $R_V(X)$ le cône (dans $\mathbf{D}(\kk\md)$) du morphisme canonique
 $$i^* P^{\mathbf{HD}(\A)^{op}}_V\overset{\mathbf{L}}{\underset{\kk[\mathbf{H}(\A)]}{\otimes}}i^*X\to P^{\mathbf{HD}(\A)^{op}}_V\overset{\mathbf{L}}{\underset{\kk[\mathbf{HD}(\A)]}{\otimes}}X\simeq X(V)$$ induit par le foncteur $i : \mathbf{H}(\A)\to\mathbf{HD}(\A)$.
\end{nota}

\begin{pr}\label{scotech} Soient $d\in\mathbb{N}$, $V$ un objet de $\mathbf{HD}(\A)$, $E$ un ensemble fini de cardinal $d+1$ et $X$ un foncteur de $(\mathbf{HD}(\A),\M)$ tel que $cr_E(i^*X)=0$ et $cr_E(X)(V)=0$.
 
Alors le morphisme
$$R_V(X)\to R_{{\rm T}_E(V)}(X)$$
 induit par le morphisme canonique $V\to{\rm T}_E(V)={\rm H}(V)^{\overset{\perp}{\oplus}E}\overset{\perp}{\oplus}V$ (qui n'est autre que $\alpha_\emptyset$) induit $0$ en homologie.
\end{pr}

\begin{proof}
 Cela vient formellement des faits suivants :
 \begin{itemize}
  \item comme $cr_E(i^*X)=0$, le diagramme~(\ref{diag-sco}) montre que $cr_E(i^*P^{\mathbf{HD}(\A)^{op}}_V)$ induit $0$ quand on applique ${\rm Tor}^{\kk[\mathbf{H}(\A)]}_\bullet(-,i^*X)$ ;
  \item l'observation de naturalité faible de ${\rm T}_E$ fournit un morphisme
  $$i^*P^{\mathbf{HD}(\A)^{op}}_V\to{\rm T}_E^*i^*P^{\mathbf{HD}(\A)^{op}}_{{\rm T}_E(V)}$$
  et la composée de celui-ci avec $cr_E : {\rm T}_E^*i^*P^{\mathbf{HD}(\A)^{op}}_{{\rm T}_E(V)}\to i^*P^{\mathbf{HD}(\A)^{op}}_{{\rm T}_E(V)}$ coïncide avec le morphisme induit par $cr_E\in\kk[\mathbf{HD}(\A)(V,{\rm T}_E(V))]$ ;
  \item si $I\in\Pp(E)$ est non vide, le morphisme $\alpha_I(V)\in\mathbf{HD}(\A)(V,{\rm T}_E(V))$ se factorise par un objet de $\mathbf{H}(\A)$ (ce fait élémentaire mais crucial se trouve par exemple dans le dernier point du lemme 5.2 de \cite{Dja-JKT}), de sorte que le morphisme qu'il induit $R_V(X)\to R_{{\rm T}_E(V)}(X)$ induit $0$ en homologie.
 \end{itemize}
\end{proof}

Comme le foncteur $i : \mathbf{H}(\A)\to\mathbf{HD}(\A)$ est pleinement fidèle, il en est de même pour $i_*$ et $\mathbf{L}i_*$ --- ce qui équivaut (dans ce dernier cas) à dire que l'unité ${\rm Id}\to i^*\mathbf{L}i_*(X)$ est un isomorphisme. On rappelle que $\dsto(\mathbf{HD}(\A),\M)$ désigne la sous-catégorie de $\dst(\mathbf{HD}(\A),\M)$ des complexes dont l'homologie est analytique.

\begin{thm}\label{th-hhd}
 Pour tout objet $X$ de $\dsto(\mathbf{HD}(\A),\M)$, la coünité $\mathbf{L}i_*(i^* X)\to X$ de l'adjonction (\ref{adj-hhd}) est un isomorphisme.
\end{thm}

\begin{proof}
Il suffit de vérifier l'assertion lorsque $X$ appartient à l'image essentielle de $p^* : \dsto(\A,\M)\to\dsto(\mathbf{HD}(\A),\M)$  --- disons $X=p^*(F)$ ; on peut même se contenter de le faire lorsque l'homologie de $F$ appartient à $\pol_d(\A,\M)$ pour un certain $d\in\mathbb{N}$. Cela provient formellement de la pleine fidélité de $\mathbf{L}i_*$ et de ce que le corollaire~6.20 de \cite{DV-pol} montre que la sous-catégorie triangulée de $\dst(\mathbf{H}(\A),\M)$ stable par sommes directes engendrée par l'image des $\pol_d(\A,\M)$ dans $\dst(\mathbf{H}(\A),\M)$ est exactement $\dsto(\mathbf{H}(\A),\M)$, et de ce que tous les foncteurs en jeu sont triangulés et commutent aux sommes directes.

La conclusion se déduit alors des propositions~\ref{sco1} et~\ref{scotech}.
\end{proof}

L'important résultat suivant généralise \cite[Théorème~5.3]{Dja-JKT}.

\begin{cor}\label{cor-hhd}
 \begin{enumerate}
  \item L'adjonction (\ref{adj-hhd}) induit une équivalence entre les catégories $\dsto(\mathbf{H}(\A),\M)$ et $\dsto(\mathbf{HD}(\A),\M)$.
  \item Pour tout $d\in\mathbb{N}$, le foncteur $i^* : \pol_d(\mathbf{HD}(\A),\M)\to\pol_d(\mathbf{H}(\A),\M)$ est une équivalence de catégories.
 \end{enumerate}
\end{cor}

\begin{thm}[Cf. \cite{DV-pol}]\label{th-dd2}
 Pour tout $d\in\mathbb{N}$, les foncteurs $p^*$ et $i^*$ induisent des équivalences de catégories
 $$\pol_d(\A,\M)/\pol_{d-2}(\A,\M)\xrightarrow{\simeq}\pol_d(\mathbf{HD}(\A),\M)/\pol_{d-2}(\mathbf{HD}(\A),\M)\xrightarrow{\simeq}$$
 $$\pol_d(\mathbf{H}(\A),\M)/\pol_{d-2}(\mathbf{H}(\A),\M).$$
\end{thm}

\begin{proof}
Le fait que 
$p^*$ et $i^*$ induisent des équivalences de catégories
 $$\pol_d(\A,\M)/\pol_{d-1}(\A,\M)\xrightarrow{\simeq}\pol_d(\mathbf{HD}(\A),\M)/\pol_{d-1}(\mathbf{HD}(\A),\M)\xrightarrow{\simeq}$$
 $$\pol_d(\mathbf{H}(\A),\M)/\pol_{d-1}(\mathbf{H}(\A),\M)$$
 découle directement des résultats précédents et du théorème~6.19 de \cite{DV-pol}.
 
 Pour obtenir le renforcement annoncé, on utilise notre théorème~\ref{th-tilde} et la proposition~6.17 de \cite{DV-pol}, qui montrent qu'il suffit de vérifier l'annulation de ${\rm Ext}^i_{(\mathbf{H}(\A),\M)}(A,C)$ pour $i\leq 1$, $A$ additif (factorisant par $pi : \mathbf{H}(\A)\to\A$) et $C$ constant. Le théorème~\ref{th-hhd} montre qu'il suffit de vérifier l'annulation de ${\rm Ext}^i_{(\mathbf{HD}(\A),\M)}(A,C)$ pour $i\leq 1$, $A$ additif (factorisant par $p : \mathbf{HD}(\A)\to\A$) et $C$ constant. On obtient cette annulation par adjonction (dans le cadre général des catégories d'éléments --- cf. (\ref{eqom})) : ${\rm Ext}^\bullet_{(\mathbf{HD}(\A),\M)}(A,C)\simeq {\rm Ext}^\bullet_{(\A,\M)}(A,C^T)$ (avec des abus de notation évidents), de sorte que la proposition~\ref{pr-h2inv} permet de conclure.
\end{proof}

\subsection{Extensions de Kan dérivées associées à un foncteur additif}\label{pkdf}

Dans toute la suite de cet article, les effets croisés $cr$ ont la signification usuelle issue d'Eilenberg-Mac Lane \cite[chapitre~II]{EML} (et non celle du §\,\ref{shd}) : si $\A$ est une petite catégorie additive, pour tout $d\in\mathbb{N}$, on dispose d'un foncteur d'effet croisé $cr_d^\A : (\A,\kk)\to (\A^d,\kk)$ (noté parfois simplement $cr_d$), qui est facteur direct de la précomposition par la somme $d$-itérée $\A^d\to\A$.

L'énoncé qui suit, dont l'intérêt dépasse le cadre du présent travail, montre que les extensions de Kan dérivées le long d'un foncteur additif se comportent de façon très raisonnable et peuvent donner lieu à des calculs explicites dans de nombreux cas favorables. Il est sans doute connu des spécialistes du sujet.

\begin{pr}\label{pr-ekad}
 Soient $\A$, $\B$ des petites catégories additives et $\Phi : \A\to\B$ un foncteur additif.
 \begin{enumerate}
  \item Pour tous $i, d\in\mathbb{N}$, $\mathbf{L}_i(\Phi_*) : (\A,\kk)\to (\B,\kk)$ envoie $\pol_d(\A,\kk)$ dans $\pol_d(\B,\kk)$.
  \item Pour tout entier $d\geq 0$, on dispose d'un isomorphisme naturel gradué
  $$cr^\B_d\circ\mathbf{L}_\bullet(\Phi_*)\simeq\mathbf{L}_\bullet(\Phi^{\times d}_*)\circ cr^\A_d$$
  de foncteurs $(\A,\kk)\to (\B^d,\kk)$, où $\Phi^{\times d} : \A^d\to\B^d$ désigne le foncteur additif $(a_1,\dots,a_d)\mapsto (\Phi(a_1),\dots,\Phi(a_d))$.
  \item Si $\kk$ est un corps, on dispose de la formule de Künneth suivante :
  $$\mathbf{L}_n(\Phi_*)(F\otimes G)\simeq\bigoplus_{i+j=n}\mathbf{L}_i(\Phi_*)(F)\otimes\mathbf{L}_j(\Phi_*)(G)$$
  (isomorphisme monoïdal symétrique naturel en les foncteurs $F$ et $G$ de $(\A,\kk)$).
  \item Supposons que le foncteur $\Phi$ appartient à $\mathbf{Add}^{ps}$. Alors, pour tout $e\in\mathbb{N}\cup\{+\infty\}$, les assertions suivantes sont équivalentes :
  \begin{enumerate}
   \item\label{lp1} les foncteurs $\mathbf{L}_i(\Phi_*)\circ\Phi^* : (\B,\kk)\to (\B,\kk)$ sont nuls pour $0<i<e$ ;
   \item\label{lp2} les foncteurs $\mathbf{L}_i(\Phi_*)\circ\Phi^* : (\B,\kk)\to (\B,\kk)$ sont nuls sur les foncteurs additifs pour $0<i<e$ ;
    \item\label{lp3} les foncteurs $\mathbf{L}_i(\Phi_*)\circ\Phi^* : (\B,\kk)\to (\B,\kk)$ sont nuls sur les foncteurs du type $\kk[A]$ pour $A\in {\rm Ob}\,\B\md$ (i.e. $A$ additif) et $0<i<e$ ;
   \item\label{lp4} si $A$ et $B$ sont des objets de $\mdd\B$ et $\B\md_\kk$ respectivement, on a ${\rm Tor}^\A_i(\Phi^*(A),\Phi^*(B))=0$ pour $0<i<e$ --- autrement dit, si $\mathbf{L}_\bullet^{add}(\Phi_*)$ désigne le dérivé à gauche du foncteur qu'induit $\Phi_*$ de $\A\md$ dans $\B\md$, c'est-à-dire du foncteur adjoint à gauche au foncteur qu'induit $\Phi^*$ de $\B\md$ dans $\A\md$, le foncteur $\mathbf{L}^{add}_i(\Phi_*)\circ\Phi^*$ est nul sur $\B\md_\kk$ (vu comme sous-catégorie de $\B\md$ par le même abus qu'au lemme~\ref{lm-troesch}) pour $0<i<e$.
  \end{enumerate}
  De plus, de manière générale, le morphisme canonique $\mathbf{L}^{add}_\bullet(\Phi_*)(A)\to\mathbf{L}_\bullet(\Phi_*)(A)$ est un isomorphisme pour $A$ dans $\A\md_\kk$.
 \end{enumerate}
\end{pr}

La notation $\mathbf{L}_\bullet^{add}$ sera conservée dans la suite.

\begin{proof}
 Les trois premières propriétés sont des conséquences aisées de l'isomorphisme canonique
 $$\mathbf{L}_\bullet(\Phi_*)(F)(V)\simeq {\rm Tor}^{\kk[\A]}_\bullet(\Phi^*P^{\B^{op}}_V,F)$$
 et du fait que le foncteur $\Phi^*P^{\B^{op}}_V$ est exponentiel (i.e. transforme somme directe en produit tensoriel).
 
 L'équivalence entre~(\ref{lp1}) et (\ref{lp3}) provient d'un argument formel de comparaison de foncteurs homologiques, étant donné que la coünité $\Phi_*\Phi^*\to {\rm Id}_{(\B,\kk)}$ est un isomorphisme en raison de la plénitude et de la faible surjectivité de $\Phi$.
 
 L'équivalence entre~(\ref{lp2}) et (\ref{lp4}) et la comparaison entre $\mathbf{L}^{add}_\bullet(\Phi_*)$ et $\mathbf{L}_\bullet(\Phi_*)$ sur les foncteurs additifs se déduisent du lemme~\ref{lm-troesch}.
 
 Le fait que (\ref{lp4}) implique (\ref{lp3}) s'obtient par un argument formel à partir de l'isomorphisme naturel 
 $$ \kk[A]\underset{\kk[\A]}{\otimes}\kk[B]\simeq\kk[A\underset{\A}{\otimes}B]$$
 (pour $A$ et $B$ dans $\mdd\A$ et $\A\md$ respectivement)
 en utilisant la correspondance de Dold-Kan et une résolution projective de $A$ ou de $B$ (cf. la construction des foncteurs dérivés à la Dold-Puppe \cite{DP}).
\end{proof}

Nous nous penchons maintenant sur le comportement des extensions de Kan à gauche dérivées long de foncteurs induits par des foncteurs additifs entre catégories d'objets hermitiens, en commençant par le cas de catégories d'objets hermitiens éventuellement dégénérés, plus abordable puisqu'il s'agit de catégories d'éléments.

\begin{pr}\label{pr-ekhd}
 Soit $\Phi : \A\to\B$ un morphisme de $\mathbf{Add}_D^{ps}$. Il existe un foncteur homologique $\bar{\mathbf{L}}^H_\bullet(\Phi) : (\A,\kk)\to (\mathbf{HD}(\B),\kk)$ et des morphismes de foncteurs homologiques $\bar{\mathbf{L}}^H_\bullet(\Phi)\to\mathbf{L}_\bullet(\mathbf{HD}(\Phi)_*)\circ p_\A^*\to p_\B^*\circ\mathbf{L}_\bullet(\Phi_*)$ vérifiant les propriétés suivantes.
 \begin{enumerate}
  \item Ces morphismes s'insèrent dans une suite exacte longue
  $$\cdots\to\bar{\mathbf{L}}^H_i(\Phi)\to\mathbf{L}_i(\mathbf{HD}(\Phi)_*)\circ p^*\to p^*\circ\mathbf{L}_i(\Phi_*)\to \bar{\mathbf{L}}^H_{i-1}(\Phi)\to\cdots$$
  $$\to\bar{\mathbf{L}}^H_0(\Phi)\to\mathbf{HD}(\Phi)_*\circ p^*\to p^*\circ\Phi_*\to 0$$
  de foncteurs $(\A,\kk)\to (\mathbf{HD}(\B),\kk)$.
  \item Pour tous $i, d\in\mathbb{N}$, la composée de $\bar{\mathbf{L}}^H_i(\Phi) : (\A,\kk)\to (\mathbf{HD}(\B),\kk)$ et du foncteur canonique $(\mathbf{HD}(\B),\kk)\to\St(\mathbf{HD}(\B),\kk)$ envoie $\pol_d(\A,\kk)$ dans $\pol_{d-1}(\mathbf{HD}(\B),\kk)$.
  \item Si $2$ est inversible dans $\kk$, ou que l'hypothèse~\ref{h2inv} est satisfaite dans $\A$ et dans $\B$, ce foncteur envoie même $\pol_d(\A,\kk)$ dans $\pol_{d-2}(\mathbf{HD}(\B),\kk)$.
 \end{enumerate}
\end{pr}

\begin{proof}
 Grâce à l'isomorphisme~(\ref{eq-omegh}), le foncteur $\mathbf{L}_\bullet(\mathbf{HD}(\Phi)_*)\circ p^*$ est décrit par un isomorphisme canonique
 $$\mathbf{L}_\bullet(\mathbf{HD}(\Phi)_*)(p^*F)(E)\simeq {\rm Tor}^{\kk[\A]}_\bullet\big(\omega_\A\big(\mathbf{HD}(\Phi)^* P^{\mathbf{HD}(\B)^{op}}_E\big),F\big).$$
 On constate que le foncteur $\omega_\A\big(\mathbf{HD}(\Phi)^* P^{\mathbf{HD}(\B)^{op}}_E\big)$ est canoniquement isomorphe à $\kk[X_E]$, où $X_E : \A^{op}\to\mathbf{Ab}$ désigne le foncteur (de degré $2$)
 $$\Phi^*\B(-,E)\underset{\Phi^* T_\B}{\times}T_\A,$$
 le produit fibré étant relatif au morphisme $\Phi^*\B(-,E)\to\Phi^* T_\B$ induit par $\B(-,E)\to T_\B$ (obtenu en tirant en arrière la structure hermitienne sur $E$) et $T_\A\to\Phi^* T_\B$ induit par $\Phi$.
 
 Comme $\Phi$ est plein, le morphisme $T_\A\to\Phi^* T_\B$ qu'il induit est surjectif, de sorte que, si l'on note $T_\Phi$ son noyau, $X_E$ s'insère dans une suite exacte naturelle
 \begin{equation}\label{esef}
  0\to T_\Phi\to X_E\to\Phi^*\B(-,E)\to 0.
 \end{equation}
 Le foncteur homologique $\bar{\mathbf{L}}^H_\bullet(\Phi)$ est défini par
 $$\bar{\mathbf{L}}^H_\bullet(\Phi)(F)(E):={\rm Tor}^{\kk[\A]}_\bullet\big({\rm Ker}\,\big(\kk[X_E]\twoheadrightarrow\Phi^* P^{\B^{op}}_{p(E)}\big),F\big),$$
 ce qui donne aussitôt des transformations naturelles vérifiant la première propriété.
 
 La résolution barre montre que le foncteur ${\rm Ker}\,\big(\kk[X_E]\twoheadrightarrow\Phi^* P^{\B^{op}}_{p(E)}\big)$ est quasi-isomorphe (et ce naturellement en $E$) à un complexe de chaînes de foncteurs $(\kk[X_E]\otimes\bar{\kk}[T_\Phi]^{\otimes n})_{n\geq 1}$ (cf. la démonstation du lemme~\ref{lm-troesch}). En utilisant derechef la résolution barre, et le fait que $E\mapsto X_E$ définit un foncteur $\mathbf{HD}(B)\to (\mathbf{HD}(\A)^{op},\kk)$ fortement polynomial de degré $1$ (en raison de la suite exacte (\ref{esef})), on obtient que l'application de $\delta_{H_1}\dots\delta_{H_r}$ à $E\mapsto\kk[X_E]$ est quasi-isomorphe, naturellement et stablement en $E$, à un complexe dont chaque terme est un produit tensoriel de $r$ foncteurs {\em réduits} de $(\mathbf{HD}(\A)^{op},\kk)$, ce qui permet facilement de conclure en utilisant le lemme d'annulation de Pirashvili et la proposition~\ref{pr-h2inv} (ainsi que la première partie du lemme~\ref{lm-troesch}, qui montre que l'annulation de la proposition~\ref{pr-h2inv} vaut en remplaçant $T$ par $T_\Phi$).
\end{proof}

\begin{rem}
 Si le foncteur $\Phi : \A\to\B$ possède une section, la démonstration précédente se simplifie ; $\bar{\mathbf{L}}^H_\bullet(\Phi)$ est alors facteur direct de $\mathbf{L}_\bullet(\mathbf{HD}(\Phi)_*)\circ p^*$ et se factorise par $p^* : (\B,\kk)\to (\mathbf{HD}(\B),\kk)$.
\end{rem}

Le corollaire~\ref{crclexc}, la proposition~\ref{pr-ekad} et le théorème ci-après donnent des renseignements précis (donnant la possibilité de mener des calculs sur des foncteurs raisonnables) sur l'homologie stable des groupes de congruence à coefficients polynomiaux dans le cas excisif (qui sera caractérisé dans le théorème~\ref{sus-gal}), à partir de l'homologie stable des mêmes groupes à coefficients constants.

\begin{thm}\label{th-ekherm}
 Soient $d\in\mathbb{N}$ et $\Phi : \A\to\B$ un morphisme de $\mathbf{Add}_D^{ps}$.
 \begin{enumerate}
  \item Le foncteur homologique $\mathbf{L}_\bullet\mathbf{H}(\Phi)_* : \St(\mathbf{H}(\A),\kk)\to\St(\mathbf{H}(\B),\kk)$ envoie $\pol_d(\mathbf{H}(\A),\kk)$ dans $\pol_d(\mathbf{H}(\B),\kk)$.
  \item Le diagramme
  $$\xymatrix{\pol_d(\A,\kk)/\pol_{d-1}(\A,\kk)\ar[r]\ar[d]_\simeq & \pol_d(\B,\kk)/\pol_{d-1}(\B,\kk)\ar[d]^\simeq \\
\pol_d(\mathbf{H}(\A),\kk)/\pol_{d-1}(\mathbf{H}(\A),\kk)\ar[r] & \pol_d(\mathbf{H}(\B),\kk)/\pol_{d-1}(\mathbf{H}(\B),\kk)
}$$
 commute (à isomorphisme canonique près), où les flèches verticales sont les équivalences induites par $pi$ d'après le théorème~\ref{th-dd2} et les flèches horizontales sont induites par $\mathbf{L}_\bullet(\Phi_*)$ et $\mathbf{L}_\bullet\mathbf{H}(\Phi)_*$.
 \item Si $2$ est inversible dans $\kk$ ou que l'hypothèse~\ref{h2inv} est satisfaite dans $\A$ et dans $\B$, on a même la commutation du diagramme analogue
 $$\xymatrix{\pol_d(\A,\kk)/\pol_{d-2}(\A,\kk)\ar[r]\ar[d]_\simeq & \pol_d(\B,\kk)/\pol_{d-2}(\B,\kk)\ar[d]^\simeq \\
\pol_d(\mathbf{H}(\A),\kk)/\pol_{d-2}(\mathbf{H}(\A),\kk)\ar[r] & \pol_d(\mathbf{H}(\B),\kk)/\pol_{d-2}(\mathbf{H}(\B),\kk)
}.$$
 \end{enumerate}

\end{thm}

\begin{proof}
 Combiner les théorèmes~\ref{th-hhd}, \ref{th-dd2} et la proposition~\ref{pr-ekhd}.
\end{proof}

\subsection{Étude des produits tensoriels sur les foncteurs polynomiaux}\label{spt}

Soit $\A$ une petite catégorie additive à dualité. Si $F$ est un foncteur défini sur $\A$, on notera $F^\vee$ la précomposition par le foncteur de dualité $D : \A^{op}\to\A$.

\begin{nota}\label{not-ptm2}
 Soient $F$ et $G$ des foncteurs dans $(\A,\kk)$. On note $F\underset{(\A,\tau)}{\otimes}G$ le foncteur de $(\A,\kk)$ défini par
 $$(F\underset{(\A,\tau)}{\otimes}G)(x):=\tau_x(F)^\vee\underset{\kk[\A]}{\otimes}\tau_x(G)$$
 et, plus généralement, pour $i\in\mathbb{N}$, on définit un foncteur ${\rm Tor}_i^{(\A,\tau)}(F,G)$ par
 $${\rm Tor}_i^{(\A,\tau)}(F,G)(x):={\rm Tor}_i^\A(\tau_x(F)^\vee,\tau_x(G)).$$
 
 Cette construction se remonte aux catégories dérivées, fournissant un foncteur
 $$-\overset{\mathbf{L}}{\underset{(\A,\tau)}{\otimes}} - : \mathbf{D}(\A,\kk)\times\mathbf{D}(\A,\kk)\to\mathbf{D}(\A,\kk).$$
\end{nota}

Le foncteur $\underset{(\A,\tau)}{\otimes}$ définit une structure monoïdale symétrique sur $(\A,\kk)$ (par auto-adjonction du foncteur $D$) et les produits tensoriels (dérivés) possèdent les mêmes propriétés formelles (commutation aux colimites filtrantes, bifoncteur homologique en chaque variable...) que les produits tensoriels usuels (dérivés), du fait que les foncteurs de translation préservent les objets projectifs de $(\A,\kk)$.

Les projections naturellement scindées $\tau_x(F)^\vee\twoheadrightarrow F(x)$ et $\tau_x(G)\twoheadrightarrow G(x)$ (où les buts sont vus comme foncteurs constants) induisent des transformations naturelles
$$\underset{(\A,\tau)}{\otimes}\to\otimes,\qquad {\rm Tor}_i^{(\A,\tau)}\to{\rm Tor}_i^\kk$$
de foncteurs $(\A,\kk)\times (\A,\kk)\to (\A,\kk)$
et
$$\overset{\mathbf{L}}{\underset{(\A,\tau)}{\otimes}}\to\overset{\mathbf{L}}{\underset{\kk}{\otimes}}$$
de foncteurs $\mathbf{D}(\A,\kk)\times\mathbf{D}(\A,\kk)\to\mathbf{D}(\A,\kk)$ qui sont des épimorphismes scindés
dont les noyaux seront notés respectivement
$$\underset{(\A,\tau)}{\bar{\otimes}},\qquad \overline{{\rm Tor}}_i^{(\A,\tau)},\quad\text{et}\quad\overset{\mathbf{L}}{\underset{(\A,\tau)}{\bar{\otimes}}}.$$

Le résultat suivant est immédiat.

\begin{lm}\label{lm-deg}
 Soient $i, a, b\in\mathbb{N}$ et $F$, $G$ des foncteurs de $\pol_a(\A,\kk)$ et $\pol_b(\A,\kk)$ respectivement. Alors ${\rm Tor}_i^{(\A,\tau)}(F,G)$ appartient à $\pol_{a+b}(\A,\kk)$ tandis que $\overline{{\rm Tor}}_i^{(\A,\tau)}(F,G)$ appartient à $\pol_{a+b-2}(\A,\kk)$.
\end{lm}

\begin{nota}
 Si $x$ est un objet de $\A$, on note ${\rm d}_x$ l'endofoncteur $-\underset{(\A,\tau)}{\bar{\otimes}}(\kk\otimes\A(Dx,-))$ de $(\A,\kk)$.
\end{nota}

On dispose donc d'isomorphismes canoniques
\begin{equation}\label{ediv}
 {\rm d}_x(F)(t)\simeq (\kk\otimes\A(-,x))\underset{\kk[\A]}{\otimes}\tau_t(F)\simeq  (\kk\otimes\A(-,x))\underset{\kk[\A]}{\otimes}(F\oplus cr_2(F)(t,-)).
\end{equation}

On notera $\bd x$ l'endofoncteur $F\mapsto \big(t\mapsto (\kk\otimes\A(-,x))\underset{\kk[\A]}{\otimes}cr_2(F)(t,-)\big)$ de $(\A,\kk)$ ; ainsi $\bd x(F)$ s'identifie à la partie réduite de ${\rm d}_x(F)$, tandis que son terme constant s'identifie à l'évaluation en $x$ du plus grand quotient additif de $F$.

Cette construction possède une fonctorialité en $\A$ : si $\Xi : \A\to\B$ est un morphisme de $\mathbf{Add}_D$, $F$ un objet de $(\B,\kk)$ et $x$ un objet de $\A$, on dispose d'un morphisme naturel
\begin{equation}\label{fdiv}
{\rm d}_x(\Xi^* F)\to\Xi^*{\rm d}_{\Xi(x)}(F)
\end{equation}
dans $(\A,\kk)$.

\begin{pr}\label{pr-bd}
\begin{enumerate}
 \item Pour tout objet $x$ de $\A$, il existe une suite exacte
 $$(\delta_x)^2\to\delta_x\to {\rm d}_x\to 0$$
 d'endofoncteurs de $(\A,\kk)$.
 \item Soient $F$ un foncteur polynomial de $(\A,\kk)$ polynomial et $d\in\mathbb{N}$. Supposons que ${\rm d}_x(F)$ est de degré au plus $d-1$ pour tout $x\in {\rm Ob}\,\A$. Alors $F$ appartient à $\pol_d(\A,\kk)$.
 \item Si $F$ appartient à $\pol_2(\A,\kk)$, alors on dispose d'un isomorphisme naturel
 $$\bd x(F)(t)\simeq cr_2(F)(x,t).$$
 \item\label{ider} Pour tout objet $x$ de $\A$, le foncteur ${\rm d}_x$ est une dérivation au sens où l'on dispose d'un isomorphisme
 $${\rm d}_x(F\underset{\kk}{\otimes}G)\simeq\big({\rm d}_x(F)\underset{\kk}{\otimes}G\big)\oplus\big(F\underset{\kk}{\otimes}{\rm d}_x(G)\big)$$
 naturel et monoïdal symétrique en les objets $F$ et $G$ de $(\A,\kk)$.
\end{enumerate}
\end{pr}

\begin{proof}
Le premier point résulte de (\ref{ediv}) et de la suite exacte classique
\begin{equation}\label{ppa}
\bar{\kk}[\A(-,x)]^{\otimes 2}\to\bar{\kk}[\A(-,x)]\to\kk\otimes\A(-,x)\to 0
\end{equation}
de $(\A^{op},\kk)$. Le deuxième s'en déduit aussitôt. Le troisième provient directement de (\ref{ediv}) et de ce que $\kk\otimes\A(-,x)$ est le plus grand quotient additif de $\kk[\A(-,x)]$ (ce qu'indique aussi (\ref{ppa})).

L'assertion~\ref{ider} se déduit directement de la définition de ${\rm d}_x$ et du lemme d'annulation de Pirashvili.
\end{proof}

\begin{rem}\label{rdd}
 Les foncteurs ${\rm d}_x$ sont des variantes des foncteurs considérés dans \cite[§\,3]{D-div}.
\end{rem}

Avant d'utiliser le produit tensoriel $ - {\underset{(\A,\tau)}{\otimes}} -$ (dont les foncteurs ${\rm d}_x$ constituent le cas particulier le plus important), nous aurons besoin du résultat général suivant.

\begin{lm}\label{lm-rcos}
 Soient $\A$ un objet de $\mathbf{Add}_D$, $\C$ la catégorie des factorisations associée à $\mathbf{H}(\A)$ (ainsi, les objets de $\C$ sont les triplets $(x,t,f)$ constitués d'objets $x$ et $t$ de $\mathbf{H}(\A)$ et d'un morphisme $f\in\mathbf{H}(\A)(t,x)$), $\mathfrak{c} : \C\to\A$ le foncteur composé de $\C\to\mathbf{H}(\A)\quad (x,t,f)\mapsto x\underset{f}{-}t$ et du foncteur d'oubli $\mathbf{H}(\A)\to\A$,  $\mathfrak{i} : \C\to\A$ le foncteur composé de $(x,t,f)\mapsto x$ et du même foncteur d'oubli, et $\mathfrak{r} : \C\to\A$ le foncteur composé de $\C\to\mathbf{H}(\A)^{op}\quad (x,t,f)\mapsto t$, de l'oubli $\mathbf{H}(\A)^{op}\to\A^{op}$ et du foncteur de dualité $D : \A^{op}\to\A$. On dispose ainsi d'une transformation naturelle canonique $\mathfrak{c}\to\mathfrak{i}$. Si $F$ est un foncteur dans $(\A,\kk)$, le foncteur $\mathfrak{c}^*F$ de $(\C,\kk)$ est naturellement quasi-isomorphe à un complexe de cochaînes
 $$\big((\mathfrak{r},\dots,\mathfrak{r})^* cr_n(F)\oplus (\mathfrak{i},\mathfrak{r},\dots,\mathfrak{r})^* cr_{n+1}(F)\big)_{n\geq 0}$$
 et le conoyau du morphisme canonique $\mathfrak{c}^*F\to\mathfrak{i}^*F$ est naturellement quasi-isomorphe à un complexe de cochaînes
 $$\big((\mathfrak{r},\dots,\mathfrak{r})^* cr_{i+1}(F)\oplus (\mathfrak{i},\mathfrak{r},\dots,\mathfrak{r})^* cr_{i+2}(F)\big)_{i\geq 0}.$$
\end{lm}

\begin{proof}
 Évaluée sur un objet $(x,t,f)$, la transformation naturelle canonique $\mathfrak{c}\to\mathfrak{i}$ est le monomorphisme scindé $x\underset{f}{-}t\to x$ de $\A$ isomorphe à $x\underset{f}{-}t\to t\oplus (x\underset{f}{-}t)$. En utilisant les décompositions canoniques $F(x)\simeq cr_0(F)\oplus cr_1(F)(x)$ et $F(x)\simeq F\big(t\oplus (x\underset{f}{-}t)\big)\simeq cr_0(F)\oplus cr_1(F)(x\underset{f}{-}t)\oplus cr_1(F)(t)\oplus cr_2(F)(x\underset{f}{-}t,t)$, on obtient une suite exacte
 $$0\to F(x\underset{f}{-}t)\to cr_0(F)\oplus cr_1(F)(x)\to cr_1(F)(t)\oplus cr_2(F)(x\underset{f}{-}t,t)\to 0.$$
 En utilisant l'isomorphisme (de $\A$) $t\simeq Dt$ donné par la structure hermitienne sur $t$ et le morphisme canonique $x\underset{f}{-}t\to x$, on déduit de cette suite exacte courte une suite exacte
 $$0\to F(x\underset{f}{-}t)\to cr_0(F)\oplus cr_1(F)(x)\to cr_1(F)(Dt)\oplus cr_2(F)(x,Dt)$$
 où le conoyau de la dernière flèche s'identifie à
 $$cr_2(F)(x,Dt)/cr_2(F)(x\underset{f}{-}t,Dt)\simeq cr_2(F)(t,Dt)\oplus cr_3(F)(x\underset{f}{-}t,t,Dt)\;;$$
 une récurrence immédiate permet de poursuivre pour obtenir une suite exacte
 $$0\to F(x\underset{f}{-}t)\to$$
 $$cr_0(F)\oplus cr_1(F)(x)\to cr_1(F)(Dt)\oplus cr_2(F)(x,Dt)\to cr_2(F)(Dt,Dt)\oplus cr_3(F)(x,Dt,Dt)\to\cdots $$
 dont la deuxième ligne est exactement l'évaluation en $(x,t,f)$ du complexe de cochaînes de l'énoncé (la naturalité est immédiate). De plus, le monomorphisme canonique $F(x\underset{f}{-}t)\to F(x)$ s'identifie via ce quasi-isomorphisme à la troncature au-delà du terme de degré $0$, $cr_0(F)\oplus cr_1(F)(x)\simeq F(x)$, ce qui permet d'en déduire le deuxième quasi-isomorphisme annoncé.
\end{proof}

On en déduit aussitôt :
\begin{lm}\label{lm-o2}
 Soient $\A$ un objet de $\mathbf{Add}_D$, $F$ un foncteur dans $(\A,\kk)$. Le $\kk$-module $\omega_t^{\mathbf{H}(\A)}((pi)^* F)(x)$ (resp. le conoyau du morphisme naturel $\omega_t^{\mathbf{H}(\A)}((pi)^* F)(x)\to\kk[\mathbf{H}(\A)(t,x)]\otimes ((pi)^* F)(x)$) est quasi-isomorphe à un complexe de cochaînes
 $$\big(\kk[\mathbf{H}(\A)(t,x)]\otimes (cr_n(F)(Dt,\dots,Dt)\oplus cr_{n+1}(F)(x,Dt,\dots,Dt))\big)_{n\geq 0}$$
 $$\text{(resp. }\big(\kk[\mathbf{H}(\A)(t,x)]\otimes (cr_{i+1}(F)(Dt,\dots,Dt)\oplus cr_{i+2}(F)(x,Dt,\dots,Dt))\big)_{i\geq 0}\text{)}$$
 naturellement en $t\in {\rm Ob}\,\mathbf{H}(\A)^{op}$ et $x\in {\rm Ob}\,\mathbf{H}(\A)$.
\end{lm}

Nous sommes maintenant en mesure d'énoncer et démontrer le résultat principal de ce paragraphe, qui donne accès à un contrôle qualitatif précis des foncteurs ${\rm Tor}_i^{[\mathbf{H}(\A),\overset{\perp}{\oplus}]}$ évalués sur des foncteurs polynomiaux.

\begin{thm}\label{thp-ptens}
 Soient $\A$ un objet de $\mathbf{Add}_D$ et $a$, $b$, $i$ des entiers naturels.
 \begin{enumerate}
  \item Le foncteur ${\rm Tor}_i^{[\mathbf{H}(\A),\overset{\perp}{\oplus}]}$ envoie $\pol_a(\mathbf{H}(\A),\kk)\times\pol_b(\mathbf{H}(\A),\kk)$ dans $\pol_{a+b}(\mathbf{H}(\A),\kk)$.
  \item Le foncteur $\overline{{\rm Tor}}_i^{[\mathbf{H}(\A),\overset{\perp}{\oplus}]}$ envoie $\pol_a(\mathbf{H}(\A),\kk)\times\pol_b(\mathbf{H}(\A),\kk)$ dans $\pol_{a+b-2}(\mathbf{H}(\A),\kk)$.
  \item Si $F$ et $G$ sont des foncteurs de $\pol_a(\A,\kk)$ et $\pol_b(\A,\kk)$ respectivement, alors il existe un morphisme naturel
$$\overline{{\rm Tor}}_i^{[\mathbf{H}(\A),\overset{\perp}{\oplus}]}((pi)^* F, (pi)^* G)\to (pi)^*\overline{{\rm Tor}}_i^{(\A,\tau)}(F,G)$$
 dont noyau et conoyau appartiennent à $\pol_{a+b-4}(\mathbf{H}(\A),\kk)$.
 \end{enumerate}
\end{thm}

\begin{proof}
 Le théorème~\ref{th-dd2} montre qu'il suffit de montrer la dernière assertion. Pour ce faire, on part de l'isomorphisme canonique
 $$\overline{{\rm Tor}}_\bullet^{[\mathbf{H}(\A),\overset{\perp}{\oplus}]}((pi)^* F, (pi)^* G)(x)\simeq$$
 $${\rm Tor}_\bullet^{\kk[\mathbf{H}(\A)]}\big(t\mapsto{\rm Coker}\,\big(\omega_t^{\mathbf{H}(\A)}((pi)^* F)(x)\to\kk[\mathbf{H}(\A)(t,x)]\otimes ((pi)^* F)(x)\big), (pi)^* G\big)$$
 qui est isomorphe par le lemme~\ref{lm-o2} à l'homologie du complexe
 $$C_\bullet(F,x,-)\overset{\mathbf{L}}{\underset{\kk[\mathbf{H}(\A)]}{\otimes}}(pi)^* G$$
 où $C_\bullet(F,x,t)$ est un complexe de {\em chaînes} tel que
 $$C_i(F,x,t)=\kk[\mathbf{H}(\A)(t,x)]\otimes (cr_{-i+1}(F)(Dt,\dots,Dt)\oplus cr_{-i+2}(F)(x,Dt,\dots,Dt))$$
 concentré en degrés $i\leq 0$. Par conséquent, on dispose d'une suite spectrale d'hyperhomologie
 $$E^1_{r,s}={\rm Tor}_s^{\kk[\mathbf{H}(\A)]}(C_r(F,x,-),(pi)^*G)\Rightarrow\overline{{\rm Tor}}_{r+s}^{[\mathbf{H}(\A),\overset{\perp}{\oplus}]}((pi)^* F, (pi)^* G)(x)$$
 (où $E^1_{r,s}$ est nul sauf si $\leq 0$ ; la différentielle $d^n$ est de bidegré $(-n,n-1)$). Comme $F$ est polynomial, le complexe $C_\bullet(F,-,-)$ est borné, de sorte qu'il n'y a pas de problème de convergence, contrairement à ce qui arriverait sans cette hypothèse de polynomialité.
 
 On note que $C_r(F,x,-)$ s'identifie à la restriction à $\mathbf{H}(\A)^{op}$ du foncteur de $(\mathbf{HD}(\A)^{op},\kk)$ donné par
 $$t\mapsto\kk[\mathbf{HD}(\A)(t,x)]\otimes (cr_{-r+1}(F)(Dt,\dots,Dt)\oplus cr_{-r+2}(F)(x,Dt,\dots,Dt)),$$
 de sorte que la polynomialité de $G$ et le théorème~\ref{th-hhd}, combinés à l'isomorphisme d'adjonction~(\ref{eq-omegh}), donnent un isomorphisme naturel
 $$E^1_{r,s}\simeq {\rm Tor}_s^{\kk[\A]}(t\mapsto\kk[\A(t,x)]\otimes (cr_{-r+1}(F)(Dt,\dots,Dt)\oplus cr_{-r+2}(F)(x,Dt,\dots,Dt)),G)$$
 {\em stablement en $x$} ; le terme de droite s'identifie à
 $${\rm Tor}_s^{\kk[\A]}(t\mapsto(cr_{-r+1}(F)(Dt,\dots,Dt)\oplus cr_{-r+2}(F)(x,Dt,\dots,Dt)),\tau_x(G))$$
 qui pour $r=0$ se réduit à $\overline{{\rm Tor}}_s^{(\A,\tau)}(F,G)(x)$.
 
 Par adjonction somme/diagonale, on a en général
 $${\rm Tor}_\bullet^{\kk[\A]}(t\mapsto(cr_n(F)(Dt,\dots,Dt)\oplus cr_{n+1}(F)(x,Dt,\dots,Dt)),\tau_x(G))\simeq$$
 $${\rm Tor}_\bullet^{\kk[\A^n]}(cr_n(F^\vee)\oplus cr_{n+1}(F)(x,-)^\vee,(t_1,\dots,t_n)\mapsto G(x\oplus t_1\oplus\dots\oplus t_n))$$
 $$\simeq {\rm Tor}_\bullet^{\kk[\A^n]}(cr_n(F^\vee)\oplus cr_{n+1}(F)(x,-)^\vee,cr_n(G)\oplus cr_{n+1}(G)(x,-))$$
 qui est polynomial de degré au plus $a+b-4$ en $x$ lorsque $n\geq 2$, d'où la conclusion.
\end{proof}

\begin{rem}\label{rqi-til}
\begin{enumerate}
 \item Ce théorème nous sera surtout utile lorsque l'un des arguments $F$ et $G$ des groupes de torsion est additif, auquel cas la démonstration se simplifie (la suite spectrale se réduit à une suite exacte longue).
 \item En utilisant la comparaison des foncteurs polynomiaux sur $\mathbf{H}(\A)$ et $\A$ \cite[théorème~6.17]{DV-pol}, on voit facilement qu'on dispose stablement d'un isomorphisme naturel
$$\widetilde{F\underset{[\mathbf{H}(\A),\overset{\perp}{\oplus}]}{\otimes}G}\simeq\tilde{F}\otimes\tilde{G}$$
lorsque $F$ et $G$ sont des foncteurs polynomiaux (ou analytiques) sur $\mathbf{H}(\A)$ (noter que cet isomorphisme vit en fait dans $(\A,\kk)$).
Cependant, il paraît difficile d'utiliser la même voie pour étudier précisément sur les ${\rm Tor}^{[\mathbf{H}(\A),+]}$ de degré supérieur.
\end{enumerate}
\end{rem}

\subsection{Homologie des groupes de congruence}\label{sconc}

On commence par montrer que l'homologie (en chaque degré) des groupes de congruence définit des foncteurs faiblement polynomiaux. On utilise à cet effet la notion de {\em donnée triangulaire} introduite dans la définition~\ref{df-triang1}. Signalons que les résultats antérieurs de polynomialité de cette homologie reposant sur des méthodes de stabilité homologique \cite{Pu,CEFN,CE-cong,CMNR} conduisent à de la polynomialité {\em forte} (\cite{CMNR} donne une meilleure borne sur le degré faible --- qui y est appelé {\em degré stable} --- que celle donnée par le degré fort, mais en général non optimale, et sans s'affranchir d'une hypothèse de rang stable de Bass fini qui garantit la polynomialité forte).

\begin{lm}\label{lm-triad}
 Soient $\A$ une petite catégorie additive et $V$ un objet de $\A$. Alors les formules $T^{\rm s}_V(A)=\A(A,V)\simeq\left(\begin{array}{cc}
                                                            1 & *\\
                                                            0 & 1
                                                            \end{array}
\right)\subset {\rm Aut}_\A(V\oplus A)\simeq\tau_V {\rm Aut}_{\mathbf{S}(\A)}(A)$ et $T^{\rm i}_V(A)=\A(V,A)\simeq \left(\begin{array}{cc}
                                                            1 & 0\\
                                                            * & 1
                                                            \end{array}
\right)\subset {\rm Aut}_\A(V\oplus A)$ définissent des sous-foncteurs $T^{\rm s}_V$ et $T^{\rm i}_V$ de $\tau_V {\rm Aut}_{\mathbf{S}(\A)} : \mathbf{S}(\A)\to\mathbf{Grp}$ (où les fonctorialités en $V$ s'entendent via les foncteurs canoniques $\mathbf{S}(\A)\to\A^{op}$ et $\mathbf{S}(\A)\to\A$).
 
De plus, l'ensemble $\{T^{\rm s}_V,T^{\rm i}_V\}$ est une donnée triangulaire en $V$ sur $\mathbf{S}(\A)$.
\end{lm}

\begin{proof}
 La première assertion est évidente.
 
 La seconde résulte du calcul matriciel :
$$\left(\begin{array}{cc}
                                                            0 & 1\\
                                                            1 & 0
                                                            \end{array}\right)=\left(\begin{array}{cc}
                                                            -1 & 0\\
                                                            0 & 1
                                                            \end{array}\right)\left(\begin{array}{cc}
                                                            1 & -1\\
                                                            0 & 1
                                                            \end{array}\right)\left(\begin{array}{cc}
                                                            1 & 0\\
                                                            1 & 1
                                                            \end{array}\right)\left(\begin{array}{cc}
                                                            1 & -1\\
                                                            0 & 1
                                                            \end{array}\right)$$
dans ${\rm Aut}_\A(V\oplus V)$.
\end{proof}

\begin{thm}\label{thcong1}
 Soit $\Phi : \A\to\B$ un morphisme de $\mathbf{Add}^{ps}$. Pour tout entier $n$, l'objet $\Hh_n(\Gamma_\Phi;\kk)$ (où l'on note $\Gamma_\Phi$ pour $\Gamma_{\mathbf{S}(\Phi)}$) de $\St(\mathbf{S}(\B),\kk)$ appartient à $\pol_{2n}(\mathbf{S}(\B),\kk)$. Plus généralement, pour $X$ dans $\pol_d(\mathbf{S}(\A),\kk)$,  $\Hh_n(\Gamma_\Phi;X)$ appartient à $\pol_{2n+d}(\mathbf{S}(\B),\kk)$.
\end{thm}

\begin{proof}
 On raisonne par récurrence sur $n$ selon le schéma suivant : notant $(H_n)$ l'assertion selon laquelle $\Hh_m(\Gamma_\Phi;\kk)$ appartient à $\pol_{2m}(\mathbf{S}(\B),\kk)$ pour tout $m<n$ et $(H'_m)$ celle que $\Hh_m(\Gamma_\Phi;X)$ appartient à $\pol_{2m+d}(\mathbf{S}(\B),\kk)$ pour tout $m<n$, tout $d\in\mathbb{N}$ et tout $X$ de $\pol_d(\mathbf{S}(\A),\kk)$, on établit les implications
 $$(H_n)\Rightarrow (H'_n)\Rightarrow (H_{n+1}).$$
 
 Le fait que $(H_n)$ entraîne $(H'_n)$ découle directement des théorèmes~\ref{th-spf}, \ref{th-ekherm} et~\ref{thp-ptens}.
 
 Supposons $(H'_n)$ vérifiée et montrons que $\Hh_n(\Gamma_\Phi;\kk)$ est de degré au plus $2n$. On utilise pour cela la proposition~\ref{pr-extri} avec la donnée triangulaire du lemme~\ref{lm-triad} et $G=\Gamma_\Phi$. On définit $T^{\rm s}_{V,\Phi}(A):={\rm Ker}\,\big(\A(A,V)\to\B(\Phi(A),\Phi(V))\big)$ (morphisme induit par $\Phi$ ; $T^{\rm s}_{V,\Phi}$ est une abréviation pour $\big(T^{\rm s}_V\big)_{\Gamma_\Phi}$) et de même $T^{\rm i}_{V,\Phi}(A):={\rm Ker}\,\big(\A(V,A)\to\B(\Phi(V),\Phi(A))\big)$. On vérifie aussitôt les trois premières hypothèses sur ces foncteurs requises dans la proposition~\ref{pr-extri}, tandis que la quatrième, pour $d=2n-1$ provient de ce que $H_j(T^{\rm s}_{V,\Phi};\kk)$ et $H_j(T^{\rm i}_{V,\Phi};\kk)$ sont des foncteurs polynomiaux de degré au plus $j$, puisque $T^{\rm s}_{V,\Phi}$ et $T^{\rm i}_{V,\Phi}$ prennent leurs valeurs dans les groupes abéliens. Il suffit donc d'appliquer la proposition~\ref{pr-swp} pour conclure.
\end{proof}

On déduit maintenant de ce résultat la généralisation hermitienne suivante\,\footnote{Il serait naturel de démontrer directement ce résultat, mais les notions de donnée et d'extension triangulaires que nous avons introduites en~\ref{df-triang1} ne semblent pas le permettre si facilement.} :
\begin{thm}\label{thcong2}
 Soit $\Phi : \A\to\B$ un morphisme de $\mathbf{Add}_D^{ps}$. Pour tout entier $n$, l'objet $\Hh_n(\Gamma_{\mathbf{H}(\Phi)};\kk)$ de $\St(\mathbf{H}(\B),\kk)$ appartient à $\pol_{2n}(\mathbf{H}(\B),\kk)$. Plus généralement, pour $X$ dans $\pol_d(\mathbf{H}(\A),\kk)$,  $\Hh_n(\Gamma_\Phi;X)$ appartient à $\pol_{2n+d}(\mathbf{H}(\B),\kk)$.
\end{thm}

\begin{proof}
Il suffit d'établir l'assertion à coefficients dans $\kk$, dont la seconde se déduit exactement comme dans la démonstration précédente. Comme le foncteur monoïdal $h_\B : \mathbf{S}(\B)\to\mathbf{H}(\B)$ est faiblement surjectif, il suffit de montrer que $h_\B^*\Hh_n(\Gamma_{\mathbf{H}(\Phi)};\kk)$ appartient à $\pol_{2n}(\mathbf{S}(\B),\kk)$.

On dispose d'un isomorphisme $h_\A o_\A h_\A\simeq h_\A\overset{\perp}{\oplus}h_\A$ de foncteurs $\mathbf{S}(\A)\to\mathbf{H}(\A)$, ce qui permet de former un diagramme commutatif
$$\xymatrix{ & \Gamma_{\mathbf{S}(\Phi)}\circ (o_\A h_\A)\ar[rd] & \\
\Gamma_{\mathbf{H}(\Phi)}\circ h_\A\ar[rr]\ar[ru]\ar[rd] && \Gamma_{\mathbf{H}(\Phi)}\circ(h_\A\overset{\perp}{\oplus}h_\A) \\
 & (\Gamma_{\mathbf{H}(\Phi)}\circ h_\A) \times (\Gamma_{\mathbf{H}(\Phi)}\circ h_\A)\ar[ru] &
}$$
de foncteurs $\mathbf{S}(\A)\to\mathbf{Grp}$ dans lequel la flèche en haut à gauche est l'inclusion induite par le foncteur $o_\A$, celle en haut à droite est induite par $h_\A$ (via l'isomorphisme qu'on vient de rappeler), celle du milieu par l'endofoncteur $E\mapsto E\overset{\perp}{\oplus} E$ de $\mathbf{H}(\A)$, celle en bas à gauche est l'inclusion diagonale et celle en bas à droite est donnée par la structure monoïdale sur le foncteur $\Gamma_{\mathbf{H}(\Phi)}$.

Raisonnant par récurrence sur le degré homologique $n$, on peut supposer $\Hh_i(\Gamma_{\mathbf{H}(\Phi)})$ de degré au plus $2i$ pour $i<n$, et le théorème~\ref{thcong1} montre que l'image du morphisme $F:=h_\B^*\Hh_n(\Gamma_{\mathbf{H}(\Phi)};\kk)\to h_\B^*\Hh_n(\Gamma_{\mathbf{H}(\Phi)};\kk)\circ (E\mapsto E\oplus E)$ induit par le diagramme précédent a une image dans $\pol_{2n}(\mathbf{S}(\B),\kk)$. En utilisant la partie inférieure du diagramme, l'hypothèse de récurrence et la formule de Künneth, on en déduit que l'image de $F\xrightarrow{i^1_*+i^2_*}F\circ (E\mapsto E\oplus E)$, où $i^1$ et $i^2$ sont les deux inclusions canoniques $E\to E\overset{\perp}{\oplus}E$ a une image dans $\pol_{2n}(\mathbf{S}(\B),\kk)$. On obtient le résultat souhaité en appliquant la proposition~\ref{swa} au noyau de la projection de $F$ sur ladite image.
 \end{proof}

Le résultat suivant constitue la généralisation hermitienne catégorique de Suslin \cite{S-exc}.
 
 \begin{thm}\label{sus-gal}
  Soit $\Phi : \A\to\B$ un morphisme de $\mathbf{Add}_D^{ps}$. Notons $e$ la borne inférieure dans $\mathbb{N}^*\cup\{+\infty\}$ des entiers $i>0$ tels que $\mathbf{L}_i^{add}(\Phi_*)\circ\Phi^*$ soit non nul sur $\B\md_\kk$, avec les notations de la proposition~\ref{pr-ekad}. Alors $\Gamma_{\mathbf{H}(\Phi)}$ est $\kk$-excisif en degrés $<e$.
  
  Si $e$ est fini, $\Gamma_{\mathbf{H}(\Phi)}$ n'est pas $\kk$-excisif en degré $e$ ; de plus, $\Hh_e(\Gamma_{\mathbf{H}(\Phi)};\kk)$ est polynomial de degré $2$, et son image dans $\pol_2(\mathbf{H}(\B),\kk)/\pol_1(\mathbf{H}(\B),\kk)\simeq\pol_2(\B,\kk)/\pol_1(\B,\kk)$ --- et même dans $\pol_2(\mathbf{H}(\B),\kk)/\pol_0(\mathbf{H}(\B),\kk)\simeq\pol_2(\B,\kk)/\pol_0(\B,\kk)$ si $2$ est inversible dans $\kk$ ou que $\A$ et $\B$ vérifient l'hypothèse~\ref{h2inv} --- (cf. théorème~\ref{th-dd2}) est isomorphe à 
 $$V\mapsto {\rm Tor}^\A_e\big(\Phi^*\B(-,V),\kk\otimes\Phi^*\B(DV,-)\big)_{\mathfrak{S}_2},$$
 où l'action de $\mathfrak{S}_2$ est induite par l'auto-adjonction du foncteur de dualité $D$.
 \end{thm}

 \begin{proof}
  Ce résultat se déduit des propositions~\ref{cnex},~\ref{pr-ekad},~\ref{pr-ekhd} et~\ref{pr-bd} (et de la formule~(\ref{ediv})) ainsi que des théorèmes~\ref{thp-ptens} et~\ref{thcong2}.
  
  Précisons la description à travers les isomorphismes naturels
  $$cr_2(H)(x,t)\simeq {\rm d}_t(H)(x)\simeq\mathbf{L}_e^{add}(\Phi_*)(\kk\otimes\Phi^*\B(Dt,-))(x)\simeq {\rm Tor}^\A_e\big(\Phi^*\B(-,x),\kk\otimes\Phi^*\B(Dt,-)\big),$$
  où l'on note $H$ pour $\Hh_e(\Gamma_{\mathbf{H}(\Phi)};\kk)$ (ou plutôt, par abus, un foncteur défini sur $\B$ ayant une image isomorphe dans la catégorie quotient via l'équivalence du théorème~\ref{th-dd2}), de l'isomorphisme canonique $cr_2(H)(x,t)\simeq cr_2(H)(t,x)$. On note tout d'abord qu'on dispose d'un isomorphisme canonique
  $${\rm Tor}^\A_e\big(\Phi^*\B(-,x),\kk\otimes\Phi^*\B(Dt,-)\big)\simeq {\rm Tor}^\A_e\big(\kk\otimes\Phi^*\B(-,x),\Phi^*\B(Dt,-)\big)$$
  (qui se réduit à la composée des isomorphismes canoniques $${\rm Tor}^\A_e\big(\Phi^*\B(-,x),\kk\otimes\Phi^*\B(Dt,-)\big)\simeq\kk\otimes{\rm Tor}^\A_e\big(\Phi^*\B(-,x),\Phi^*\B(Dt,-)\big)$$
  et
  $$\kk\otimes{\rm Tor}^\A_e\big(\Phi^*\B(-,x),\Phi^*\B(Dt,-)\big)\simeq{\rm Tor}^\A_e\big(\kk\otimes\Phi^*\B(-,x),\Phi^*\B(Dt,-)\big)$$
  si le groupe abélien sous-jacent à $\kk$ est sans torsion) en comparant chaque membre à l'homologie en degré $e$ de
  $$\Phi^*\B(-,x)\overset{\mathbf{L}}{\underset{\A}{\otimes}}\big(\kk\overset{\mathbf{L}}{\underset{\mathbb{Z}}{\otimes}}\Phi^*\B(Dt,-)\big)\simeq\big(\kk\overset{\mathbf{L}}{\underset{\mathbb{Z}}{\otimes}}\Phi^*\B(-,x)\big)\overset{\mathbf{L}}{\underset{\A}{\otimes}}\Phi^*\B(Dt,-),$$
  en utilisant que ${\rm Tor}^\A_i\big(\Phi^*\B(-,x),{\rm Tor}_1^\mathbb{Z}(\kk,\Phi^*\B(Dt,-))\big)$ et ${\rm Tor}^\A_i\big({\rm Tor}_1^\mathbb{Z}(\kk,\Phi^*\B(-,x)),\Phi^*\B(Dt,-)\big)$ sont nuls pour $0<i<e$ (grâce à l'hypothèse de minimalité faite sur $e$)
  et canoniquement isomorphes (à ${\rm Tor}_1^\mathbb{Z}(\kk,\B(Dt,x))$) pour $i=0$.
  
  On considère alors le diagramme commutatif
  $$\xymatrix{cr_2(H)(x,t)\ar[r]\ar[d] & {\rm Tor}^\A_e\big(\Phi^*\B(-,x),\kk\otimes\Phi^*\B(Dt,-)\big)\ar[d] \\
  cr_2(H')(Dx,Dt)\ar[r]\ar[dd] & {\rm Tor}^{\A^{op}}_e\big(\Phi^*\B(Dx,-),\kk\otimes\Phi^*\B(-,t)\big)\ar[d] \\
  & {\rm Tor}^{\A}_e\big(\kk\otimes\Phi^*\B(-,t),\Phi^*\B(Dx,-)\big)\ar[d] \\
  cr_2(H)(t,x)\ar[r] & {\rm Tor}^\A_e\big(\Phi^*\B(-,t),\kk\otimes\Phi^*\B(Dx,-)\big)
  }$$
  dans lequel $H'$ désigne l'analogue de $H$ en remplaçant $\Phi : \A\to\B$ par $\Phi^{op} : \A^{op}\to\B^{op}$, les flèches horizontales sont les isomorphismes précédents, les flèches verticales supérieures sont induites par la dualité, la flèche verticale en bas à droite est l'isomorphisme canonique spécifié précédemment (et celle au-dessus est l'isomorphisme d'échange des facteurs dans le groupe de torsion). On vérifie aussitôt que les composées verticales décrivent notre isomorphisme d'échange (en utilisant la façon dont l'isomorphisme $cr_2(H)(x,t)\simeq {\rm d}_t(H)(x)$, valable parce que $H$ est quadratique sans quotient additif, est construit), ce dont on déduit le résultat souhaité.
  
  Un autre point mérite quelques précisions : le fait qu'on obtienne un isomorphisme dans $\pol_2(\mathbf{H}(\B),\kk)/\pol_0(\mathbf{H}(\B),\kk)\simeq\pol_2(\B,\kk)/\pol_0(\B,\kk)$ si $2$ est inversible dans $\kk$ ou que $\A$ et $\B$ vérifient l'hypothèse~\ref{h2inv}. Le fait que ${\rm d}_t(H)$ soit un foncteur réduit, pour tout $t$, montre que le foncteur quadratique $H$ n'a pas de quotient additif non nul ; par conséquent, modulo les foncteurs constants, le morphisme canonique ${\rm cr}_2(H)\to H$ est un {\em épimorphisme} (où ${\rm cr}_2$ a la signification de l'appendice~\ref{a2}). Il suffit donc de noter que la partie réduite de $H$ est {\em pseudo-diagonalisable} au sens de la proposition~\ref{df-fqpd}, ce qui est gratuit si $2$ est inversible dans $\kk$ ou se déduit via les corollaires~\ref{pbcr} et~\ref{pdcr} de ce que $\B$ vérifie l'hypothèse~\ref{h2inv}.
 \end{proof}

 \begin{rem}
 \begin{enumerate}
  \item Dans l'énoncé précédent, on peut remplacer les coïnvariants (sous l'action de $\mathfrak{S}_2$) par les invariants, qui leur sont canoniquement isomorphes dans la catégorie quotient considérée (par la proposition~\ref{df-fqpd} si cette catégorie quotient est $\pol_2/\pol_0$).
  \item On peut décrire concrètement l'isomorphisme du théorème de la façon suivante. Soient $\tilde{x}$ un objet de $\mathbf{H}(\A)$ et $x:=\mathbf{H}(\Phi)(x)$. On dispose d'un morphisme naturel
  $$\Hh_e(\Gamma_{\mathbf{H}(\Phi)};\kk)(x)=H_e(\Gamma_{\mathbf{H}(\Phi)}(\tilde{x});\kk)={\rm Tor}^{\mathbb{Z}[\Gamma_{\mathbf{H}(\Phi)}(\tilde{x})]}_e(\mathbb{Z},\kk)\to\cdots$$
  $${\rm Tor}^{{\rm End}_\A(\tilde{x})}_e(\Phi^*\B(-,x)(\tilde{x}),\kk\otimes\Phi^*\B(Dx,-)(\tilde{x}))\to {\rm Tor}^\A_e(\Phi^*\B(-,x),\kk\otimes\Phi^*\B(Dx,-))$$
 où la première flèche est induite par le morphisme de monoïdes (multiplicatifs) $\mathbb{Z}[\Gamma_{\mathbf{H}(\Phi)}(\tilde{x})]\twoheadrightarrow\Gamma_{\mathbf{H}(\Phi)}(\tilde{x})\hookrightarrow {\rm End}_\A(\tilde{x})$ et sur les coefficients par les morphismes $\Gamma_{\mathbf{H}(\Phi)}(\tilde{x})$-équivariants $\mathbb{Z}\to\B(x,x)=\Phi^*\B(-,x)(\tilde{x})$ donné par l'identité de $x$ et $\kk\to\kk\otimes\B(Dx,x)=\kk\otimes\Phi^*\B(Dx,-)(\tilde{x})$ par l'isomorphisme $Dx\to x$ inverse de celui induit par la structure hermitienne sur $x$, tandis que la deuxième flèche est induite par le foncteur additif pleinement fidèle canonique de la catégorie préadditive à un objet associée à l'anneau ${\rm End}_\A(\tilde{x})$ dans $\A$. Il suffit alors de projeter sur les coïnvariants sous l'action de l'involution décrite dans l'énoncé (ou de constater qu'on tombe dans les invariants --- cf. remarque précédente) et de passer à la catégorie quotient idoine pour obtenir l'isomorphisme du théorème~\ref{sus-gal}.
 
 Cette construction est une généralisation directe de Suslin \cite[§\,4]{S-exc}.
 \end{enumerate}
 \end{rem}

 En particulier :
 \begin{thm}\label{sus-gal2}
  Soit $\Phi : \A\to\B$ un morphisme de $\mathbf{Add}^{ps}$. Notons $e$ la borne inférieure dans $\mathbf{N}^*\cup\{+\infty\}$ des entiers $i>0$ tels que $\mathbf{L}_i^{add}(\Phi_*)\circ\Phi^*$ soit non nul sur $\B\md_\kk$ (avec les notations de la proposition~\ref{pr-ekad}). Alors $\Gamma_{\mathbf{S}(\Phi)}$ est $\kk$-excisif en degrés $<e$.
  
 Si $e$ est fini, $\Gamma_{\mathbf{S}(\Phi)}$ n'est pas $\kk$-excisif en degré $e$ ; de plus, $\Hh_e(\Gamma_\Phi;\kk)$ est polynomial de degré $2$, et son image dans $\pol_2(\mathbf{S}(\B),\kk)/\pol_0(\mathbf{S}(\B),\kk)\simeq\pol_2(\B^{op}\times\B,\kk)/\pol_0(\B^{op}\times\B,\kk)$ (cf. théorème~\ref{th-dd2}) est isomorphe à 
 $$(U,V)\mapsto {\rm Tor}^\A_e\big(\Phi^*\B(-,V),\kk\otimes\Phi^*\B(U,-)\big).$$
 \end{thm}
 
 Dans l'énoncé qui suit, si $A$ est un anneau, $GL(A)$ désigne le foncteur d'automorphismes de la catégorie $\mathbf{S}(A)$ et, si $M$ est un $\kk$-$A$-bimodule, $\mathfrak{gl}(M)$ désigne le foncteur $V\mapsto {\rm End}_A(V)\otimes_A M$ de $(\mathbf{S}(A),\kk)$.
 
 \begin{cor}[Suslin \cite{S-exc}]\label{csus}
  Soient $A$ un anneau et $I$ un idéal bilatère de $A$. Notons $e$ la borne inférieure dans $\mathbf{N}^*\cup\{+\infty\}$ des entiers $i>0$ tels que ${\rm Tor}^A_i(A/I,\kk\otimes A/I)$ soit non nul. Alors l'anneau sans unité $I$ est $\kk$-excisif pour la $K$-théorie algébrique jusqu'en degrés $<e$ ; de plus, on dispose d'un isomorphisme
  $$H_e({\rm Ker}\,(GL(A)\to GL(A/I));\kk)\simeq\mathfrak{gl}({\rm Tor}^A_e(A/I,\kk\otimes A/I))$$
  dans $\pol_2(\mathbf{S}(A/I),\kk)/\pol_0(\mathbf{S}(A/I),\kk)$.
 \end{cor}

 Nous pouvons maintenant examiner l'homologie stable de nos groupes de congruence au-delà du premier degré non excisif (sous une hypothèse raisonnable sur $\kk$). 
 \begin{pr}\label{hne}
 Supposons que $\kk$ est de dimension homologique au plus~$1$.
 
 Soient $\Phi : \A\to\B$ un morphisme de $\mathbf{Add}_D^{ps}$ et $e\in\mathbb{N}^*$. On suppose que $\mathbf{L}^{add}_i\Phi_*\Phi^*$ est nul sur $\B\md_\kk$  pour $0<i<e$. 
 
 Pour tout entier $n>0$, $\Hh_n(\Gamma_\Phi;\kk)$ est de degré au plus $2\big[\frac{n}{e}\big]$ (où les crochets indiquent la partie entière). De plus, si $n=me$ et que $F$ est un objet de $\mathcal{P}ol_d(\mathbf{H}(\B),\kk)$, on a un isomorphisme naturel 
$$F\underset{[\mathbf{H}(\B),\overset{\perp}{\oplus}]}{\overline{\otimes}}\Hh_{me}(\Gamma_{\mathbf{H}(\Phi)};\kk)\simeq\mathbf{L}_e(\mathbf{H}(\Phi)_*)\mathbf{H}(\Phi)^*(F)\underset{\kk}{\otimes}\Hh_{(m-1)e}(\Gamma_{\mathbf{H}(\Phi)};\kk)$$
dans $\mathcal{P}ol_{d+2m-2}(\mathbf{H}(\B),\kk)/\mathcal{P}ol_{d+2m-4}(\mathbf{H}(\B),\kk)$. 
\end{pr}

\begin{proof} On note que l'hypothèse faite sur $\kk$ implique que la suite spectrale $CU(F)$ s'arrête à la deuxième page.

Posons $m=\big[\frac{n}{e}\big]$ et ${\rm H}_q=\Hh_q(\Gamma_{\mathbf{H}(\Phi)};\kk)$. Par récurrence, on peut supposer que ${\rm H}_N$ est de degré au plus $2\big[\frac{N}{e}\big]$ pour $N<n$. On examine alors le morphisme de suites spectrales $E(\mathbf{H}(\Phi)^*F;\mathbf{H}(\Phi))\to CU(F;\mathbf{H}(\Phi))$ construit au §\,\ref{sssf}, en degré total $n$, d'abord dans la catégorie quotient $\St(\mathbf{H}(\B),\kk)/\pol_{2m+d-2}(\mathbf{H}(\B),\kk)$, puis dans $\St(\mathbf{H}(\B),\kk)/\pol_{2m+d-4}(\mathbf{H}(\B),\kk)$ lorsque $n$ est multiple de $e$.

L'hypothèse d'excisivité en degrés $<e$ et la suite spectrale $E'(\mathbf{H}(\Phi)^*F,{\rm H}_q)$ montrent qu'on dispose d'isomorphismes $E^2_{p,q}(\mathbf{H}(\Phi)^*F)\simeq {\rm Tor}^{[\mathbf{H}(\B),\overset{\perp}{\oplus}]}_p(F,{\rm H}_q)$ pour $p<e$.  Si l'on a également $q<n$, l'hypothèse de récurrence et le théorème~\ref{thp-ptens} montrent que le noyau et le conoyau du morphisme $E^2_{p,q}(\mathbf{H}(\Phi)^*F)\to CU^2_{p,q}(F)={\rm Tor}^\kk_p(F,{\rm H}_q)$ sont de degré au plus $2\big[\frac{q}{e}\big]+d-2$, qui est toujours inférieur à $2m+d-2$, et même à $2m+d-4$ si $n$ est multiple de $e$. Ainsi, $E^r_{p,q}\to CU^r_{p,q}$ est un isomorphisme dans $\St(\mathbf{H}(\B),\kk)/\pol_{2m+d-2}(\mathbf{H}(\B),\kk)$ (et dans $\St(\mathbf{H}(\B),\kk)/\pol_{2m+d-4}(\mathbf{H}(\B),\kk)$ si $n=me$) pour $r=2$ (on suppose toujours $p<e$ et $q<n$). Cette propriété est en fait valable pour tous les $r\geq 2$ --- et donc aussi pour $r=\infty$ ---  si $p<e$, $q<n$ et $p+q\leq n$, comme on le voit par récurrence (sur $r$) en considérant le diagramme commutatif
$$\xymatrix{E^r_{p+r,q-r+1}\ar[r]\ar[d]_{d^r} & CU^r_{p+r,q-r+1}\ar[d]^{d^r=0}\\
E^r_{p,q}\ar[r]^\simeq\ar[d]_{d^r} & CU^r_{p,q}\ar[d]^{d^r=0}\\
E^r_{p-r,q+r-1}\ar[r]^\simeq & CU^r_{p-r,q+r-1}
}$$
où les deux isomorphismes horizontaux (dans la catégorie quotient) proviennent de l'hypothèse de récurrence (qui s'applique en bas car $q+r-1\leq n-p+r-1<n$ si $p-r\geq 0$).

 Le morphisme $E^2_{0,n}(\mathbf{H}(\Phi)^*F)\simeq F\underset{[\mathbf{H}(\B),\overset{\perp}{\oplus}]}{\otimes}{\rm H}_n\to CU^2_{0,n}(F)=F\otimes {\rm H}_n$ a pour conoyau $F\underset{[\mathbf{H}(\B),\overset{\perp}{\oplus}]}{\overline{\otimes}}{\rm H}_n$ ; ce conoyau persiste à l'infini car $CU^2_{r,s}(F)=0$ pour $r\geq 2$ en raison de l'hypothèse de dimension homologique faite sur $\kk$.

Quant aux termes $CU^2_{p,q}(F)$ et $E^2_{p,q}(\mathbf{H}(\Phi)^*F)$ pour $p\geq e$ et $p+q=n$, l'hypothèse de récurrence (sur $n$) et les théorèmes~\ref{th-ekherm} et~\ref{thp-ptens} (ainsi que la suite spectrale $E'(\mathbf{H}(\Phi)^*F,{\rm H}_q)$) montrent qu'ils sont de degré au plus $2\big[\frac{q}{e}\big]+d\leq 2\big[\frac{n-e}{e}\big]+d=2m+d-2$. 

Par conséquent, dans la catégorie quotient $\St(\mathbf{H}(\B),\kk)/\pol_{2m+d-2}(\mathbf{H}(\B),\kk)$, ne subsistent à l'infini en degré total $n$ que les termes :
\begin{enumerate}
 \item $E^\infty_{p,n-p}\to CU^\infty_{p,n-p}$ pour $0<p<e$, qui sont des isomorphismes ; 
 \item $E^\infty_{0,n}\to CU^\infty_{0,n}$, dont le conoyau est isomorphe à $F\underset{[\mathbf{H}(\B),\overset{\perp}{\oplus}]}{\overline{\otimes}}{\rm H}_n$.
\end{enumerate}

Comme le morphisme de suites spectrales $E\to CU$ induit un isomorphisme entre les aboutissements, on en déduit que, pour tout $F$ dans $\pol_d(\mathbf{H}(\B),\kk)$, $F\underset{[\mathbf{H}(\B),\overset{\perp}{\oplus}]}{\overline{\otimes}}{\rm H}_n$ est nul dans la catégorie quotient $\St(\mathbf{H}(\B),\kk)/\pol_{2m+d-2}(\mathbf{H}(\B),\kk)$, c'est-à-dire qu'il appartient à $\pol_{2m+d-2}(\mathbf{H}(\B),\kk)$. La proposition~\ref{pr-bd} et les théorèmes~\ref{thp-ptens} et~\ref{thcong2} garantissent donc que ${\rm H}_n$ appartient à $\pol_{2m}(\mathbf{H}(\B),\kk)$.

Supposons maintenant $n=me$ et examinons notre morphisme de suites spectrales dans $\St(\mathbf{H}(\B),\kk)/\pol_{2m+d-4}(\mathbf{H}(\B),\kk)$. On reprend le raisonnement précédent, mais il faut accorder une attention particulière aux termes de bidegré $(e,(m-1)e)$. L'hypothèse d'excisivité jusqu'en degré $e-1$ montre que $E'^2_{i,j}(\mathbf{H}(\Phi)^*F,{\rm H}_{(m-1)e})$ est nul pour $0<i<e$, d'où une suite exacte
$$E'^2_{e+1,0}\to E'^2_{0,e}\to E^2_{e,(m-1)e}\to E'^2_{e,0}\to 0\;;$$
par ailleurs le morphisme naturel
$$E'^2_{i,j}={\rm Tor}^{[\mathbf{H}(\B),\overset{\perp}{\oplus}]}_i(\mathbf{L}_j(\mathbf{H}(\Phi)_*)(\mathbf{H}(\Phi)^*F),{\rm H}_{(m-1)e})\to {\rm Tor}^\kk_i(\mathbf{L}_j(\mathbf{H}(\Phi)_*)(\mathbf{H}(\Phi)^*F),{\rm H}_{(m-1)e})$$
est un isomorphisme dans $\St(\mathbf{H}(\B),\kk)/\pol_{2m+d-2}(\mathbf{H}(\B),\kk)$ (par le théorème~\ref{thp-ptens} et l'hypothèse de récurrence) ; comme son but est nul pour $i>1$ par hypothèse sur $\kk$, on en déduit que la suite exacte précédente se réduit, dans la catégorie quotient, à une suite exacte courte
$$0\to {\rm H}_{(m-1)e}\otimes\mathbf{L}_e(\mathbf{H}(\Phi)_*)\mathbf{H}(\Phi)^*(F)\to E^2_{e,(m-1)e}\to CU^2_{e,(m-1)e}\to 0.$$

En raisonnant comme précédemment, mais cette fois dans la catégorie quotient $\St(\mathbf{H}(\B),\kk)/\pol_{2m+d-4}(\mathbf{H}(\B),\kk)$, on voit que n'y subsistent à l'infini en degré total $n=me$ que les termes :
\begin{enumerate}
 \item $E^\infty_{p,me-p}\to CU^\infty_{p,me-p}$ pour $0<p<e$, qui sont des isomorphismes ; 
 \item $E^\infty_{e,(m-1)e}\to CU^\infty_{e,(m-1)}$, qui est un épimorphisme de noyau isomorphe à $\mathbf{L}_e(\mathbf{H}(\Phi)_*)\mathbf{H}(\Phi)^*(F)\otimes {\rm H}_{(m-1)e}$ (la propriété établie juste avant sur les deuxièmes pages s'étend par récurrence aux pages $r\geq 2$, car les termes de bidegré $(e+r,(m-1)e-r+1)$ susceptibles d'y entrer sont nuls dans la catégorie quotient puisque $\big[\frac{(m-1)e-r+1}{e}\big]\leq m-2$, et les termes de bidegré $(e-r,(m-1)e+r-1)$ qui en sortent sont isomorphes, avec une différentielle nulle pour $CU$) ;
 \item $E^\infty_{0,me}\to CU^\infty_{0,me}$, dont le conoyau est isomorphe à $F\underset{[\mathbf{H}(\B),\overset{\perp}{\oplus}]}{\overline{\otimes}}{\rm H}_{me}$.
\end{enumerate}

Le fait que le morphisme de suites spectrales $E\to CU$ induit un isomorphisme entre les aboutissements permet d'en déduire l'isomorphisme naturel $F\underset{[\mathbf{H}(\B),\overset{\perp}{\oplus}]}{\overline{\otimes}}{\rm H}_{me}\simeq\mathbf{L}_e(\mathbf{H}(\Phi)_*)\mathbf{H}(\Phi)^*(F)\otimes {\rm H}_{(m-1)e}$ recherché (dans la catégorie quotient).
\end{proof}

\begin{rem}
 \begin{enumerate}
 \item Lorsque $\kk$ est un corps ou que $e$ égale $1$, la démonstration se simplifie légèrement.
 \item Lorsque $n$ n'est pas multiple de $e$, l'auteur ignore si $\Hh_n(\Gamma_\Phi;\kk)$ peut être de degré strictement inférieur à  $2\big[\frac{n}{e}\big]$. En effet, l'examen des suites spectrales fait apparaître en degré homologique $n$ plusieurs termes de degré polynomial maximal qui pourraient parfois se compenser.
 \item Quand $\kk$ n'est pas de dimension homologique au plus $1$, il n'est pas clair que la proposition subsiste. Néanmoins, ce n'est pas tellement restrictif en pratique, l'examen de $\Hh_\bullet(\Gamma_\Phi;\kk)$ étant surtout intéressant lorsque $\kk$ est un corps, un localisé ou un complété de l'anneau des entiers.
 \end{enumerate}
\end{rem}

Sous les hypothèses de la proposition~\ref{hne}, on voit que, pour $n=me$, ${\rm Tor}^\kk_1(\Hh_i(\Gamma_{\mathbf{H}(\Phi)};\kk),\Hh_j(\Gamma_{\mathbf{H}(\Phi)};\kk))$ est de degré au plus $2n-2$ pour $i+j=n-1$, de même que $\Hh_i(\Gamma_{\mathbf{H}(\Phi)};\kk)\otimes\Hh_j(\Gamma_{\mathbf{H}(\Phi)};\kk)$ lorsque $i+j=n$ et que $i$ et $j$ ne sont pas multiples de $e$. Par conséquent, la formule de Künneth fournit un isomorphisme
$$\Hh_{me}(\Gamma_{\mathbf{H}(\Phi)}\times\Gamma_{\mathbf{H}(\Phi)};\kk)\simeq\bigoplus_{a+b=m}\Hh_{ae}(\Gamma_{\mathbf{H}(\Phi)};\kk)\otimes\Hh_{be}(\Gamma_{\mathbf{H}(\Phi)};\kk)$$
dans la catégorie 
 $\pol_{2m}(\mathbf{H}(\B),\kk)/\pol_{2m-2}(\mathbf{H}(\B),\kk)$, de sorte que l'on y dispose d'un coproduit
 $$\Hh_{me}(\Gamma_{\mathbf{H}(\Phi)};\kk)\to\bigoplus_{a+b=m}\Hh_{ae}(\Gamma_{\mathbf{H}(\Phi)};\kk)\otimes\Hh_{be}(\Gamma_{\mathbf{H}(\Phi)};\kk).$$
 La discusion de la fin du §\,\ref{sprt} montre que la restriction aux degrés multiples de $e$ de $\Hh_\bullet(\Gamma_{\mathbf{H}(\Phi)};\kk)$ possède une structure de type Hopf, où l'on regarde $\Hh_{me}(\Gamma_{\mathbf{H}(\Phi)};\kk)$ dans la catégorie 
 $\pol_{2m}(\mathbf{H}(\B),\kk)/\pol_{2m-2}(\mathbf{H}(\B),\kk)$ (noter que le passage à ces catégories quotients garantit que les produits tensoriels usuels et $\underset{[\mathbf{H}(\B),\overset{\perp}{\oplus}]}{\otimes}$ coïncident, d'après le théorème~\ref{thp-ptens}).
 
 \begin{thm}\label{thcong3}
  Plaçons-nous dans la situation du théorème~\ref{sus-gal} ; supposons $e$ fini et $\kk$ de dimension homologique au plus $1$. Alors $\Hh_n(\Gamma_{\mathbf{H}(\Phi)};\kk)$ est de degré au plus $2\big[\frac{n}{e}]$, avec égalité lorsque $n$ est multiple de $e$. De plus, la structure multiplicative sur $\Hh_\bullet(\Gamma_{\mathbf{H}(\Phi)};\kk)$ induit pour tout $m\in\mathbb{N}$ un isomorphisme
  $$\big(\Hh_e(\Gamma_{\mathbf{H}(\Phi)};\kk)^{\otimes m}\big)_{\mathfrak{S}_m}\simeq\Hh_{me}(\Gamma_{\mathbf{H}(\Phi)};\kk),$$
 l'action de $\mathfrak{S}_m$ par permutation des facteurs du produit tensoriel (qui est pris sur $\kk$) étant tordue par la signature si $e$ est impair,
  dans $\pol_{2m}(\mathbf{H}(\B),\kk)/\pol_{2m-1}(\mathbf{H}(\B),\kk)$ et même dans $\pol_{2m}(\mathbf{H}(\B),\kk)/\pol_{2m-2}(\mathbf{H}(\B),\kk)$ si $2$ est inversible dans $\kk$ ou que $\A$ et $\B$ vérifient l'hypothèse~\ref{h2inv}.
  
  Via cet isomorphisme, la structure comultiplicative spécifiée ci-dessus est la structure cocommutative (au sens gradué) libre, modulo l'isomorphisme canonique
  $$\big(\Hh_e(\Gamma_{\mathbf{H}(\Phi)};\kk)^{\otimes m}\big)_{\mathfrak{S}_m}\simeq\big(\Hh_e(\Gamma_{\mathbf{H}(\Phi)};\kk)^{\otimes m}\big)^{\mathfrak{S}_m}$$
  (dans la catégorie quotient).
 \end{thm}

 \begin{proof}
 On commence par établir l'assertion relative à la catégorie quotient $\pol_{2m}(\mathbf{H}(\B),\kk)/\pol_{2m-1}(\mathbf{H}(\B),\kk)$ ; l'assertion plus forte, valable dans $\pol_{2m}(\mathbf{H}(\B),\kk)/\pol_{2m-2}(\mathbf{H}(\B),\kk)$, quand $2$ est inversible dans $\kk$ ou que $\A$ et $\B$ vérifient l'hypothèse~\ref{h2inv}, se montre essentiellement de la même façon, modulo quelques raffinements que nous expliciterons à la fin.
 
 On note tout d'abord que, d'après le premier point de la proposition~\ref{pr-bd}, pour tout objet $x$ de $\B$ et tout naturel $i$, le foncteur ${\rm d}_x$ induit un foncteur (encore noté de la même façon) de
 $$\pol_i(\B,\kk)/\pol_{i-1}(\B,\kk)\simeq\pol_i(\mathbf{H}(\B),\kk)/\pol_{i-1}(\mathbf{H}(\B),\kk)$$
 vers
 $$\pol_{i-1}(\B,\kk)/\pol_{i-2}(\B,\kk)\simeq\pol_{i-1}(\mathbf{H}(\B),\kk)/\pol_{i-2}(\mathbf{H}(\B),\kk)$$
 isomorphe au foncteur induit par le foncteur différence $\delta_x$, qui est exact.
 
 La discussion précédant l'énoncé explique pourquoi la structure multiplicative  sur $\Hh_\bullet(\Gamma_{\mathbf{H}(\Phi)};\kk)$ induit une flèche
  \begin{equation}\label{mpi}
\big(\Hh_e(\Gamma_{\mathbf{H}(\Phi)};\kk)^{\otimes m}\big)_{\mathfrak{S}_m}\to\Hh_{me}(\Gamma_{\mathbf{H}(\Phi)};\kk)  
  \end{equation}
 dans la catégorie quotient.
  
  On note également que la proposition~\ref{hne} et les théorèmes~\ref{th-ekherm} et~\ref{thp-ptens} fournissent un isomorphisme
  \begin{equation}\label{eqib}
  {\rm d}_x(\Hh_{me}(\Gamma_{\mathbf{H}(\Phi)};\kk))\simeq\Hh_{(m-1)e}(\Gamma_{\mathbf{H}(\Phi)};\kk)\otimes\mathbf{L}^{add}_e(\Phi_*)(\kk\otimes\Phi^*\B(Dx,-))
  \end{equation}
  (dans la catégorie quotient) naturel en l'objet $x$ de $\B$. Cet isomorphisme possède également une fonctorialité en la flèche $\Phi : \A\to\B$ de $\mathbf{Add}_D^{ps}$ en le sens suivant.
  
  Supposons que
  $$\xymatrix{\A\ar[r]^\Phi\ar[d]_\Psi & \B\ar[d]^\Xi \\
 \A'\ar[r]^{\Phi'} & \B' 
  }$$
  est un diagramme commutatif (à isomorphisme près) de $\mathbf{Add}_D$ dans lequel $\Phi$ et $\Phi'$ appartiennent à $\mathbf{Add}_D^{ps}$ ; supposons également que le plus petit entier $i>0$ tel que $\mathbf{L}^{add}_i(\Phi'_*)\circ\Phi'^*$ soit non nul sur $\B'\md_\kk$ soit $e$ (comme pour $\Phi$). Alors on dispose d'un diagramme commutatif
  $$\xymatrix{{\rm d}_x(\Hh_{me}(\Gamma_{\mathbf{H}(\Phi)};\kk))\ar[r]^-\simeq\ar[d] & \Hh_{(m-1)e}(\Gamma_{\mathbf{H}(\Phi)};\kk)\otimes\mathbf{L}^{add}_e(\Phi_*)(\kk\otimes\Phi^*\B(Dx,-))\ar[d] \\
  \Xi^*{\rm d}_{\Xi(x)}(\Hh_{me}(\Gamma_{\mathbf{H}(\Phi')};\kk))\ar[r]^-\simeq & \Xi^*\Hh_{(m-1)e}(\Gamma_{\mathbf{H}(\Phi')};\kk)\otimes\Xi^*\mathbf{L}^{add}_e(\Phi'_*)(\kk\otimes\Phi'^*\B'(D\Xi(x),-))
  }$$
  dans $\pol_{2m-1}(\mathbf{H}(\B),\kk)/\pol_{2m-2}(\mathbf{H}(\B),\kk)$, où la flèche verticale de gauche est induite par le morphisme canonique $\Hh_\bullet(\Gamma_{\mathbf{H}(\Phi)};\kk)\to\Xi^*\Hh_\bullet(\Gamma_{\mathbf{H}(\Phi')};\kk)$ et (\ref{fdiv}) et celle de droite est analogue à ce qui est explicité après le théorème~\ref{th-spf} (c'est d'ailleurs la fonctorialité des suites spectrales de ce théorème --- et des autres isomorphismes utilisés dans nos raisonnements précédents --- qui atteste de la commutativité du diagramme précédent).
  
  Comme la structure multiplicative sur $\Hh_\bullet(\Gamma_{\mathbf{H}(\Phi)};\kk)$ est induite par le diagramme commutatif (à isomorphisme près)
  $$\xymatrix{\A\times\A\ar[r]^{\Phi\times\Phi}\ar[d] & \B\times\B\ar[d] \\
 \A\ar[r]^{\Phi} & \B 
  }$$
  de $\mathbf{Add}_D$ dont les flèches verticales sont données par la somme directe, 
  et la structure comultiplicative par  le diagramme commutatif (à isomorphisme près)
  $$\xymatrix{\A\ar[r]^{\Phi}\ar[d] & \B\ar[d] \\
 \A\times\A\ar[r]^{\Phi\times\Phi} & \B\times\B 
  }$$
  dont les flèches verticales sont les diagonales, on en déduit que l'isomorphisme (\ref{eqib}) est compatible aux structures multiplicative et comultiplicative en le sens suivant : notant $H_m$ pour $\Hh_{me}(\Gamma_{\mathbf{H}(\Phi)};\kk)$ et $L_x$ pour $\mathbf{L}^{add}_e(\Phi_*)(\kk\otimes\Phi^*\B(Dx,-))$, le diagramme suivant commute
  \begin{equation}\label{edcm}
   \xymatrix{{\rm d}_x(H_m\otimes H_n)\ar[d]_-\simeq\ar[r] &  {\rm d}_x(H_{m+n})\ar[dd]^-\simeq \\ 
  ({\rm d}_x(H_m)\otimes H_n)\oplus (H_m\otimes {\rm d}_x(H_n))\ar[d]_-\simeq & \\
 (H_{m-1}\otimes H_n\otimes L_x)\oplus (H_m\otimes H_{n-1}\otimes L_x)\ar[r] & H_{m+n-1}\otimes L_x
  }
  \end{equation}
  où l'isomorphisme en haut à gauche est induit par la propriété de dérivation de ${\rm d}_x$ (proposition~\ref{pr-bd}.\ref{ider}), les deux autres isomorphismes par (\ref{eqib}), et les flèches horizontales par la structure multiplicative, de même qu'un diagramme analogue (avec les flèches horizontales renversées) pour la structure comultiplicative. La propriété de dérivation de ${\rm d}_x$ donne des isomorphismes ${\rm d}_x(T^{\otimes m})\simeq {\rm d}_x(T)\otimes T^{\otimes (m-1)}\uparrow_{\mathfrak{S}_{m-1}}^{\mathfrak{S}_m}$ ($\mathfrak{S}_m$-équivariant) puis ${\rm d}_x\big((T^{\otimes m})_{\mathfrak{S}_m}\big)\simeq {\rm d}_x(T)\otimes (T^{\otimes (m-1)})_{\mathfrak{S}_{m-1}}$ pour tout foncteur $T$ de $(\B,\kk)$, ce qui montre par récurrence sur $m$ que le morphisme (\ref{mpi}) induit après application de ${\rm d}_x\simeq\delta_x$ un isomorphisme dans $\pol_{2m-1}(\mathbf{H}(\B),\kk)/\pol_{2m-2}(\mathbf{H}(\B),\kk)$, pour tout objet $x$ de $\B$. Comme $(H_1^{\otimes m})^{\mathfrak{S}_m}$ est isomorphe à $(H_1^{\otimes m})_{\mathfrak{S}_m}$ dans $\pol_{2m}(\mathbf{H}(\B),\kk)/\pol_{2m-1}(\mathbf{H}(\B),\kk)$, on obtient finalement que les flèches
  \begin{equation}\label{eqcm}
  (H_1^{\otimes m})_{\mathfrak{S}_m}\to H_m\to (H_1^{\otimes m})^{\mathfrak{S}_m}
  \end{equation}
  données par le produit et le coproduit sont des isomorphismes dans la catégorie quotient, la composée étant l'isomorphisme induit par la norme (cette assertion sur la composée est évidente en degré $m=1$ et s'étend aux degrés supérieurs par récurrence en appliquant ${\rm d}_x$, grâce au diagramme (\ref{edcm}) et à son analogue comultiplicatif). Cela fournit l'assertion relative à la catégorie quotient $\pol_{2m}(\mathbf{H}(\B),\kk)/\pol_{2m-1}(\mathbf{H}(\B),\kk)$.

  Lorsque $2$ est inversible dans $\kk$ ou que $\A$ et $\B$ vérifient l'hypothèse~\ref{h2inv}, on peut reprendre le raisonnement précédent dans $\pol_{2m}(\mathbf{H}(\B),\kk)/\pol_{2m-2}(\mathbf{H}(\B),\kk)$. La seule difficulté supplémentaire est que le foncteur de
 $$\pol_i(\B,\kk)/\pol_{i-2}(\B,\kk)\simeq\pol_i(\mathbf{H}(\B),\kk)/\pol_{i-2}(\mathbf{H}(\B),\kk)$$
 vers
 $$\pol_{i-1}(\B,\kk)/\pol_{i-3}(\B,\kk)\simeq\pol_{i-1}(\mathbf{H}(\B),\kk)/\pol_{i-3}(\mathbf{H}(\B),\kk)$$
 induit par ${\rm d}_x$ n'est plus isomorphe à celui qu'induit $\delta_x$ ; ce foncteur induit par ${\rm d}_x$ n'est donc a priori exact qu'à droite. On la contourne grâce à l'observation suivante (qui se déduit de la proposition~\ref{pr-bd}) : si $F$ et $G$ sont des foncteurs de $\pol_i(\B,\kk)$ tels que $\mathbf{L}_1({\rm d}_x)(G)$ appartienne à $\pol_{i-3}(\B,\kk)$ pour tout objet $x$ de $\B$, alors tout morphisme $F\to G$ de $\pol_i(\mathbf{H}(\B),\kk)/\pol_{i-2}(\mathbf{H}(\B),\kk)$ qui induit un isomorphisme dans $\pol_{i-1}(\mathbf{H}(\B),\kk)/\pol_{i-3}(\mathbf{H}(\B),\kk)$ après application de chaque foncteur ${\rm d}_x$ est un isomorphisme. On l'applique aux morphismes de (\ref{eqcm}), dont la composée est un isomorphisme dans $\pol_{2m}(\mathbf{H}(\B),\kk)/\pol_{2m-2}(\mathbf{H}(\B),\kk)$, en utilisant la proposition~\ref{prd1q}, qui montre que $\mathbf{L}_1({\rm d}_x)(H_1^{\otimes m})$ est nul (on a vu à la fin de la démonstration du théorème~\ref{sus-gal} que $H_1$ est pseudo-diagonalisable), ce qui permet de conclure puisque $(H_1^{\otimes m})^{\mathfrak{S}_m}$ en est un facteur direct dans $\pol_{2m}(\mathbf{H}(\B),\kk)/\pol_{2m-2}(\mathbf{H}(\B),\kk)$ (puisque la norme $(H_1^{\otimes m})_{\mathfrak{S}_m}\to (H_1^{\otimes m})^{\mathfrak{S}_m}$ y devient un isomorphisme).
 \end{proof}

 En particulier, on obtient la généralisation suivante du corollaire~\ref{csus} :
 
 \begin{cor}\label{cs-ds}
  Soient $I$ un idéal bilatère d'un anneau $A$ et $d\in\mathbb{N}^*$. Alors  l'objet $\Hh_d({\rm Ker}\,(GL(A)\to GL(A/I));\mathbb{Z})$ de $\St(\mathbf{S}(A/I),\mathbb{Z})$ est polynomial de degré au plus $2d$ ; il est isomorphe à $\Lambda^d(\mathfrak{gl}(I/I^2))$ dans la catégorie $\pol_{2d}(\mathbf{S}(A/I),\mathbb{Z})/\pol_{2d-2}(\mathbf{S}(A/I),\mathbb{Z})$, où $\Lambda^d$ désigne la $d$-ème puissance extérieure. Ces isomorphismes, pour $d$ variable, sont compatibles aux structures multiplicatives et comultiplicatives.
 \end{cor}

 \begin{rem}
  On peut utiliser les puissances extérieures plutôt que le quotient des puissances tensorielles sous l'action signée des groupes symétriques, dans l'énoncé précédent, parce que ces foncteurs coïncident dans la catégorie quotient considérée. On notera également que, sur les groupes abéliens, le $n$-groupe d'homologie est isomorphe à $\Lambda^n$ modulo les endofoncteurs polynomiaux de degré strictement inférieur à $n$ des groupes abéliens. Par conséquent, on peut reformuler le corollaire précédent en disant que la projection d'un groupe de congruence sur son abélianisation induit un isomorphisme en homologie en passant aux catégories quotients appropriées. 
 \end{rem}

 Signalons également un résultat à coefficients tordus :
 
 \begin{cor}\label{cor-ctf}
  Supposons que $\kk$ est de dimension homologique au plus $1$. Soient $\Phi : \A\to\B$ un morphisme de $\mathbf{Add}^{ps}$ et $d\geq 0$, $n\geq 0$, $e>0$ des entiers. On suppose que $\mathbf{L}_i^{add}(\Phi_*)\circ\Phi^*$ est nul sur $\B\md_\kk$ pour $0<i<e$. Pour tout $F$ dans $\pol_d(\mathbf{H}(\A),\kk)$, $\Hh_{ne}(\Gamma_{\mathbf{H}(\Phi)};F)$ appartient à $\pol_{2n+d}(\mathbf{H}(\B),\kk)$ et est naturellement isomorphe à
  $$\mathbf{H}(\Phi)_*(F)\underset{\kk}{\otimes}\Hh_{ne}(\Gamma_{\mathbf{H}(\Phi)};\kk)$$
  dans la catégorie $\pol_{2n+d}(\mathbf{H}(\B),\kk)/\pol_{2(n-1)+d}(\mathbf{H}(\B),\kk)$.
  
  De façon équivalente, le produit
 $$\Hh_0(\Gamma_{\mathbf{H}(\Phi)};F)\underset{[\mathbf{H}(\B),\overset{\perp}{\oplus}]}{\otimes}\Hh_{ne}(\Gamma_{\mathbf{H}(\Phi)};\kk)\to\Hh_{ne}(\Gamma_{\mathbf{H}(\Phi)};F)$$
 et le coproduit
 $$\Hh_{ne}(\Gamma_{\mathbf{H}(\Phi)};F)\to\Hh_0(\Gamma_{\mathbf{H}(\Phi)};F)\underset{\kk}{\otimes}\Hh_{ne}(\Gamma_{\mathbf{H}(\Phi)};\kk)$$
 induisent des isomorphismes dans la catégorie quotient.
 \end{cor}

 (Le coproduit est bien défini dans la catégorie de foncteurs si $\kk$ est un corps ; il l'est toujours dans la catégorie quotient --- cf. la discussion précédant le théorème~\ref{thcong3}.)
 
\begin{proof}
 La  suite spectrale $E'(F,\Hh_q(\Gamma_{\mathbf{H}(\Phi)};\kk);\mathbf{H}(\Phi))$, la proposition~\ref{hne} et le théorème~\ref{thp-ptens} montrent que, pour tous entiers naturels $p$ et $q$, $E^2_{p,q}(F;\mathbf{H}(\Phi))$ appartient à $\pol_{2[q/e]+d}(\mathbf{H}(\B),\kk)$. Il en découle que, en degré total $ne$, seul le terme $E^2_{0,ne}(F;\mathbf{H}(\Phi))$ de $E^2(F;\mathbf{H}(\Phi))$ est non nul dans $\pol_{2n+d}(\mathbf{H}(\B),\kk)/\pol_{2n+d-2}(\mathbf{H}(\B),\kk)$, et qu'il persiste à l'infini. Comme il est isomorphe à $\mathbf{H}(\Phi)_*(F)\underset{\kk}{\otimes}\Hh_{ne}(\Gamma_{\mathbf{H}(\Phi)};\kk)$ dans la catégorie quotient (toujours grâce au théorème~\ref{thp-ptens}), le corollaire s'ensuit.
\end{proof}

\begin{ex}
 Soit $I$ un idéal bilatère d'un anneau $A$. On dispose pour tout $n\in\mathbb{N}$ d'un isomorphisme naturel
 $$H_n({\rm Ker}\,\big(GL(A)\to GL(A/I));\mathfrak{gl}(A)\big)\simeq\mathfrak{gl}(A/I)\otimes\Lambda^n\big(\mathfrak{gl}(I/I^2)\big)$$
 dans $\pol_{2n+2}(\mathbf{S}(A/I),\mathbf{Ab})/\pol_{2n}(\mathbf{S}(A/I),\mathbf{Ab})$.
\end{ex}

 \appendix
 
 \section{Conjectures pour l'homologie des groupes $IA$}\label{a1}

 Nous nous penchons désormais (suivant la fin du mémoire \cite{HDR}) sur le cas du morphisme de CMHS $\Phi : \G\to\mathbf{S}(\mathbb{Z})$ donné par l'abélianisation, où $\G$ désigne comme dans \cite{DV2} la CMHS associée aux groupes d'automorphismes des groupes libres (cf. la remarque~\ref{rq-quillen}), de sorte que $\Gamma_\Phi(F_n)$ (où $F_n$ est un groupe libre de rang $n$) est le groupe, souvent noté $IA_n$, noyau de l'épimorphisme de groupes canonique ${\rm Aut}(F_n)\twoheadrightarrow GL_n(\mathbb{Z})$.
 
 On dispose donc pour tout $d\in\mathbb{N}$ d'un objet $\Hh_d(IA;\kk)$ de $\St(\mathbf{S}(\mathbb{Z});\kk)$ (cf. la proposition-définition~\ref{prdfh}), et même d'un foncteur $H_d(IA;\kk)$ dans $(\mathbf{S}(\mathbb{Z});\kk)$ (voir la remarque~\ref{rreg}).
 
 Le résultat classique suivant (qui découle sans peine de travaux anciens de Magnus \cite{Mag} et est démontré explicitement par Kawazumi \cite[théorème~6.1]{K-Magnus}) fournit une description complète en degré $1$ :
 
 \begin{pr}[Magnus, Kawazumi...]
  Il existe un isomorphisme
  $$H_1(IA;\mathbb{Z})(V)\simeq {\rm Hom}_\mathbb{Z}(V,\Lambda^2(V))$$
naturel en l'objet $V$ de $\mathbf{S}(\mathbb{Z})$.
 \end{pr}

 En particulier, $H_1(IA;\kk)$ est polynomial de degré (faible ou fort) exactement $3$ (lorsque l'anneau $\kk$ est non nul) : $\Phi$ n'est pas excisif pour l'homologie stable.
 
La compréhension de $\Hh_d(IA;\mathbb{Z})$ (et même de $\Hh_d(IA;\mathbb{Q})$) devient dramatiquement compliquée pour $d\geq 2$.

Au vu des résultats de finitude connus pour l'homologie des groupes de congruence associés à un anneau sans unité raisonnable, et de la stabilité homologique pour les groupes ${\rm Aut}(F_n)$, il paraît raisonnable de conjecturer :

\begin{conj}\label{cIA1}
 Pour tout $d\in\mathbb{N}$, le foncteur $H_d(IA;\mathbb{Z})$ de $(\mathbf{S}(\mathbb{Z});\mathbb{Z})$ est fortement polynomial.
\end{conj}

Au-delà de cette conjecture, que l'on ne s'attend pas à pouvoir aborder par les méthodes {\em stables} du présent article, on soupçonne :
\begin{conj}\label{cIA2}
 Pour tout $d\in\mathbb{N}$, $\Hh_d(IA;\mathbb{Z})$ appartient à $\pol_{3d}(\mathbf{S}(\mathbb{Z}),\mathbb{Z})$, mais pas à $\pol_{3d-1}(\mathbf{S}(\mathbb{Z}),\mathbb{Z})$.
\end{conj}

(On pourrait espérer ensuite décrire l'image de cet objet dans la catégorie quotient $\pol_{3d}(\mathbf{S}(\mathbb{Z}),\mathbb{Z})/\pol_{3d-1}(\mathbf{S}(\mathbb{Z}),\mathbb{Z})$, voire $\pol_{3d}(\mathbf{S}(\mathbb{Z}),\mathbb{Z})/\pol_{3d-2}(\mathbf{S}(\mathbb{Z}),\mathbb{Z})$.)
 
 Les résultats partiels de Pettet \cite{Pet} sur le degré homologique $2$ montrent que $\Hh_2(IA;\mathbb{Q})$ n'appartient pas à $\pol_5(\mathbf{S}(\mathbb{Z}),\mathbb{Q})$, mais on ne semble guère savoir dire davantage à l'heure actuelle (voir néanmoins Miller-Patzt-Wilson \cite[théorème~A]{MPW}).
 
 La stratégie qu'on espère pouvoir mettre en \oe uvre pour établir la conjecture~\ref{cIA2} reposerait sur le même schéma général que celle du §\,\ref{sconc} : utiliser des arguments de données et extensions triangulaires appropriés, s'appuyer sur les propriétés de ${\rm Tor}_\bullet^{[\mathbf{S}(\mathbb{Z}),\oplus]}$ établies au §\,\ref{spt} et étudier le comportement (stable) sur les foncteurs polynomiaux de l'extension de Kan dérivée à gauche le long de $\Phi : \G\to\mathbf{S}(\mathbb{Z})$. Cela pourrait être décomposé en plusieurs étapes :
 \begin{enumerate}
  \item étudier le comportement de l'extension de Kan dérivée à gauche le long du foncteur canonique $\G\to\mathbf{S}_c(\gr)$ en généralisant \cite{Dja-hodge} (notamment son théorème~1.5), dont on conserve les notations ($\mathbf{S}_c(\gr)$ est comme $\G$ une catégorie ayant les mêmes objets que la catégorie usuelle $\gr$ --- sous-catégorie pleine de $\mathbf{Grp}$ --- des groupes libres de rang fini ; le foncteur $\Phi$ se factorise par le foncteur canonique $\G\to\mathbf{S}_c(\gr)$, qui est l'identité sur les objets) ; 
  \item étudier l'extension de Kan dérivée à gauche le long du foncteur (toujours induit par l'abélianisation) $\mathbf{S}_c(\gr)\to\mathbf{S}(\mathbb{Z})$ sur des foncteurs polynomiaux en se ramenant à l'analogue entre catégories de factorisations $\mathbf{F}(\gr)\to\mathbf{F}(\mathbf{ab})$. On pourrait pour cela s'inspirer du §\,\ref{shd}, ainsi que de \cite[§\,5, 6]{DV2} et \cite[théorème~1.3]{Dja-hodge} ;
  \item les catégories de factorisations étant des catégories d'éléments, se ramener finalement à l'étude d'extensions de Kan dérivées à gauche le long du foncteur usuel d'abélianisation, ce que fait la proposition suivante.
 \end{enumerate}

\begin{pr}\label{pr-ekdab}
 Soient $i$, $d$ des entiers naturels et $\Psi : \gr\to\mathbf{P}(\mathbb{Z})$ le foncteur d'abélianisation entre les catégories usuelles de groupes (abéliens au but) libres de rang fini. 
 \begin{enumerate}
  \item Le foncteur $\mathbf{L}_i(\Psi_*) : (\gr,\kk)\to (\mathbf{P}(\mathbb{Z}),\kk)$ envoie $\pol_d(\gr,\kk)$ dans $\pol_{d+i}(\mathbf{P}(\mathbb{Z}),\kk)$.
  \item Le foncteur $\mathbf{L}_i(\Psi^{op}_*) : (\gr^{op},\kk)\to (\mathbf{P}(\mathbb{Z})^{op},\kk)$ envoie $\pol_d(\gr^{op},\kk)$ dans $\pol_d(\mathbf{P}(\mathbb{Z})^{op},\kk)$.
  \item Le foncteur $\mathbf{L}_i((\Psi^{op}\times\Psi)_*) : (\gr^{op}\times\gr,\kk)\to (\mathbf{P}(\mathbb{Z})^{op}\times\mathbf{P}(\mathbb{Z}),\kk)$ envoie $\pol_d(\gr^{op}\times\gr,\kk)$ dans $\pol_{d+i}(\mathbf{P}(\mathbb{Z})^{op}\times\mathbf{P}(\mathbb{Z}),\kk)$.
 \end{enumerate}
\end{pr}

 \begin{proof}[Esquisse de démonstration]
 La résolution barre et la formule de Künneth montrent que $\mathbf{L}_i(\Psi_*)$ envoie la $d$-ème puissance tensorielle de l'abélianisation tensorisée par $\kk$ sur un foncteur de degré $d+i$ (pour $\kk=\mathbb{Z}$, $\mathbf{L}_i(\Psi_*)$ envoie l'abélianisation sur l'homologie en degré $i+1$ sur les groupes abéliens libres, c'est-à-dire sur la $i+1$-ème puissance extérieure). Le premier point s'en déduit, par récurrence sur le degré polynomial, en utilisant \cite[§\,3]{DV2}.

 La deuxième assertion découle formellement de ce que les catégories $\gr$ et $\mathbf{P}(\mathbb{Z})$ possèdent un objet nul et des coproduits finis que le foncteur $\Psi$ préserve, par un argument d'adjonction somme/diagonale.
 
 La dernière assertions se déduit formellement des deux précédentes et de la suite spectrale de Künneth
 $$E^2_{p,q}={\rm Tor}^{\kk[\C]}_p\big(A,c\mapsto {\rm Tor}^{\kk[\D]}_q(B,X(c,-))\big)\Rightarrow {\rm Tor}^{\kk[\C\times\D]}_{p+q}(A\boxtimes B,X)$$
 où $\C$ et $\D$ sont des petites catégories, $X$ est un objet de $(\C\times\D,\kk)$ et $A$ et $B$ des objets de $(\C^{op},\kk)$ et $(\D^{op},\kk)$ respectivement dont l'un prend des valeurs $\kk$-plates.
 \end{proof}

 \section{Foncteurs quadratiques pseudo-diagonalisables}\label{a2}
 
 Dans toute cette section, $\A$ désigne une petite catégorie additive et $\M$ une catégorie de Grothendieck.
 
 Si $F$ est un foncteur de $(\A,\M)$, on note ${\rm cr}_2(F)$ le foncteur de $(\A,\M)$ obtenu en précomposant l'effet croisé $cr_2(F) : \A\times\A\to\M$ par la diagonale $\A\to\A\times\A$. Ce foncteur est muni d'une involution naturelle et l'on dispose de morphismes naturels ${\rm cr}_2(F)_{\mathfrak{S}_2}\to F\to {\rm cr}_2(F)^{\mathfrak{S}_2}$ (grâce à l'adjonction somme/diagonale), dont la composée n'est autre que la norme. Ces morphismes sont des isomorphismes dans $\pol_2(\A,\M)/\pol_1(\A,\M)$ si $F$ est quadratique (i.e. dans $\pol_2(\A,\M)$).
 
 On rappelle qu'un foncteur de $(\A,\M)$ est dit {\em diagonalisable} s'il est isomorphe à la composée de la diagonale $\A\to\A\times\A$ et d'un bifoncteur $B : \A\to\times\A\to\M$ réduit en chaque variable, c'est-à-dire tel que $B(x,0)$ et $B(0,x)$ soient nuls pour tout objet $x$ de $\A$.
 
 \begin{prdef}\label{df-fqpd}
  Soit $F : \A\to\M$ un foncteur quadratique réduit (i.e. nul en $0$). Les assertions suivantes sont équivalentes :
  \begin{enumerate}
   \item\label{idp1} $F$ est facteur direct d'un foncteur diagonalisable de $(\A,\M)$ ;
   \item\label{idp2} ${\rm Ext}_{(\A,\M)}^\bullet(A,F)$ et ${\rm Ext}_{(\A,\M)}^\bullet(F,A)$ sont nuls pour tout foncteur additif $A : \A\to\M$ ;
   \item\label{idp3} ${\rm Ext}_{(\A,\M)}^i(A,F)$ et ${\rm Ext}_{(\A,\M)}^i(F,A)$ sont nuls pour $i\in\{0,1\}$ et pour tout foncteur additif $A : \A\to\M$ ;
   \item\label{idp4} les morphismes naturels ${\rm cr}_2(F)_{\mathfrak{S}_2}\to F\to {\rm cr}_2(F)^{\mathfrak{S}_2}$ sont des isomorphismes ;
   \item\label{idp5} le morphisme naturel ${\rm cr}_2(F)\to F$ est un épimorphisme scindé.
  \end{enumerate}

  Si ces conditions sont vérifiées, on dira que $F$ est {\em pseudo-diagonalisable}.
 \end{prdef}
 
 \begin{proof}
  Les implications \ref{idp2}$\,\Rightarrow\,$\ref{idp3} et \ref{idp5}$\,\Rightarrow\,$\ref{idp1} sont évidentes.
  
  L'implication \ref{idp1}$\,\Rightarrow\,$\ref{idp2} découle du lemme d'annulation de Pirashvili.
  
  Supposant \ref{idp3} vérifié, le morphisme naturel ${\rm cr}_2(F)_{\mathfrak{S}_2}\to F$ est un épimorphisme car son conoyau $C$ est additif et ${\rm Hom}(F,C)=0$, son noyau $N$ est additif donc cet épimorphisme est scindé puisque ${\rm Ext}^1(F,N)=0$, d'où l'on déduit que c'est un isomorphisme puisque ${\rm cr}_2(F)_{\mathfrak{S}_2}$, quotient de ${\rm cr}_2(F)$, n'a pas de quotient additif non nul. De même on voit que $F\to {\rm cr}_2(F)^{\mathfrak{S}_2}$ est un isomorphisme si ${\rm Ext}^i(A,F)=0$ pour $i\leq 1$ et $A$ additif, d'où \ref{idp4}.
  
  Si \ref{idp4} est vérifié, le morphisme $F\xrightarrow{\simeq}{\rm cr}_2(F)^{\mathfrak{S}_2}\hookrightarrow {\rm cr}_2(F)$ est une section du morphisme canonique ${\rm cr}_2(F)\to F$, d'où \ref{idp5}.
 \end{proof}
 
 Si $2$ est inversible dans $\M$, tout foncteur quadratique réduit $\A\to\M$ est pseudo-diagonalisable.

 \begin{cor}\label{pbcr}
  Soit $B : \A\times\A\to\M$ un bifoncteur symétrique (i.e. muni d'isomorphismes naturels $\sigma_{x,y} : B(x,y)\simeq B(y,x)$ tels que $\sigma_{y,x}=\sigma_{x,y}^{-1}$) et additif par rapport à chaque variable. Notons $F : \A\to\M$ le foncteur quadratique réduit défini par $F(x)=B(x,x)_{\mathfrak{S}_2}$ (l'action de $\mathfrak{S}_2$ étant donnée par $\sigma_{x,x}$). Alors $F$ est pseudo-diagonalisable si et seulement si $B$ est facteur direct {\em dans la catégorie des bifoncteurs symétriques} (où les morphismes sont les transformations naturelles {\em compatibles aux isomorphismes} $\sigma_{x,y}$)  de $(x,y)\mapsto B(x,y)\oplus B(y,x)$ muni de la symétrie qui échange les deux facteurs de la somme directe.
 \end{cor}
 
 \begin{cor}\label{pdcr}
  Supposons que $\A$ est munie d'un foncteur de dualité. Alors cette catégorie vérifie l'hypothèse~\ref{h2inv} si et seulement si le foncteur $\A\times\A\to\mathbf{Ab}\quad (x,y)\mapsto\A(Dx,y)$ est facteur direct de $(x,y)\mapsto\A(Dx,y)\oplus\A(Dy,x)$ {\em dans la catégorie des bifoncteurs symétriques}, où le premier est muni de la symétrie induite par l'auto-dualité du foncteur $D$ tordue par le signe $\epsilon$ et le deuxième de celle qui échange les deux facteurs de la somme directe.
 \end{cor}

 \begin{proof}
  Le premier bifoncteur symétrique est l'effet croisé $cr_2(T)$ (au changement de variance près jouant sur le foncteur $D : \A^{op}\xrightarrow{\simeq}\A$), d'où la conclusion par le corollaire précédent.
 \end{proof}

 Le résultat suivant est utilisé dans la démonstration du théorème~\ref{thcong3}.

 \begin{pr}\label{prd1q}
  Supposons que $\A$ est munie d'un foncteur de dualité. Soient $F$ un foncteur quadratique réduit pseudo-diagonalisable de $(\A,\kk)$, $n>0$ un entier et $x$ un objet de $\A$. Alors $\mathbf{L}_1({\rm d}_x)(F^{\otimes n})=0$.
 \end{pr}

 \begin{proof}
 On a
 $$\mathbf{L}_1({\rm d}_x)(F^{\otimes n})(t)\simeq {\rm Tor}^{\kk[\A]}_1\big(\kk\otimes\A(-,x),\tau_t(F^{\otimes m})\big)$$
 (cf. (\ref{ediv})), qui est isomorphe par adjonction somme/diagonale à la somme directe de $m$ copies de
 $$\mathbf{L}_1({\rm d}_x)(F^{\otimes n})(t)\simeq {\rm Tor}^{\kk[\A]}_1\big(\kk\otimes\A(-,x),F(t)^{\otimes (m-1)}\otimes\tau_t(F)\big).$$
 
 Le foncteur $\tau_t(F)$ est isomorphe à la somme directe d'un foncteur constant, d'un foncteur additif et de $F$. Or on a
 $$\mathbf{L}_1({\rm d}_x)(F^{\otimes n})(t)\simeq {\rm Tor}^{\kk[\A]}_1\big(\kk\otimes\A(-,x),A)=0$$
 pour tout foncteur $A$ de $\pol_1(\A,\kk)$ (grâce à la suite exacte (\ref{ppa}) par exemple), et
 $$\mathbf{L}_1({\rm d}_x)(F^{\otimes n})(t)\simeq {\rm Tor}^{\kk[\A]}_1\big(\kk\otimes\A(-,x),F(t)^{\otimes (m-1)}\otimes F\big)=0$$
 par le lemme d'annulation de Pirashvili, puisque $F$ est facteur direct d'un foncteur diagonalisable, d'où la conclusion.
 \end{proof}

\bibliographystyle{plain}
\bibliography{bibli-hosgc.bib}
\end{document}